\documentclass{amsart}%
\usepackage{amsfonts}
\usepackage{amsmath}
\usepackage{amssymb}
\usepackage{graphicx}%

\usepackage[colorlinks, linkcolor=blue,  citecolor=blue, urlcolor=blue]{hyperref}%
\hypersetup{pdfstartview=FitH}
\usepackage{url}
\usepackage{xspace}
\usepackage[alphabetic]{amsrefs}

\usepackage{verbatim}
\usepackage{showtags}

\usepackage{alltt}
\usepackage{graphicx}

\providecommand{\U}[1]{\protect\rule{.1in}{.1in}}

\newtheorem{theorem}{Theorem}
\theoremstyle{plain}

\newtheorem{definition}{Definition}

\newtheorem{lemma}{Lemma}

\newtheorem{remark}{Remark}

\numberwithin{equation}{section}

\author{Yuanlong Ruan}
\address{Beijing University of Aeronautics and Astronautics, 100191, Beijing, China}
\email{ruanyl@buaa.edu.cn}



\begin{document}
\title[Limit of quasilinear equations and related extremal problems]{Limit of a class of quasilinear equations and related extremal problems}

\begin{abstract}

We perform a complete analysis of the limiting behaviour of a class of quasilinear problems with Dirichlet boundary data $g$,
$$
-\operatorname{div}\mathcal{A}_{n}\left(  x,\nabla u\left(  x\right)  \right)
=0,\text{ }x\in\Omega.
$$
We show that the Lipschitz constant of $g$ plays a role in controlling the $\Gamma$-convergence of the natural energies. However the solutions converge uniformly to solution of a limiting equation whenever $g \in Lip(\partial \Omega)$, irrelevant to the Lipschitz constant of $g$ . The limiting equation has no coercivity in $u$. We prove that the limiting equation admits a weak comparison principle and has a unique viscosity solution. We also obtain a Poincare’s inequality in the Sobolev-Orlicz space for discontinuous $\mathcal{A}_n$, which paves the way for our study of an extremal problem where its operator $\mathcal{A}_{\infty}$ becomes unbounded in a subdomain $\Omega_\infty \subset \Omega $. Upon giving proper meaning to its solution, we show that the extremal problem has a unique solution. It turns out the solution belongs to $C\left(  \bar{\Omega}\right)  \cap W^{1,\infty}\left(  \Omega\right)  $, although $\mathcal{A}_{\infty}$ is discontinuous.  In the appendix we provide some technical inequalities which play crucial roles in the proof of uniqueness and we believe will be of independent interest.

\end{abstract}
\maketitle

\section{\label{sec_intro}Introduction}

We study the limiting behaviour as $n\rightarrow\infty$ of a class of
quasilinear problems in the nonlinear potential theory%
\begin{equation}
-\operatorname{div}\mathcal{A}_{n}\left(  x,\nabla u\left(  x\right)  \right)
=0,\text{ }x\in\Omega, \label{prototype_quasilinear_n}%
\end{equation}
and a related extremal problem%
\begin{equation}
-\operatorname{div}\mathcal{A}_{\infty}\left(  x,\nabla u\left(  x\right)
\right)  =0,\text{ }x\in\Omega, \label{prototype_quasilinear}%
\end{equation}
where $\mathcal{A}_{\infty}$ is unbounded in a subdomain of $\Omega$. These
quasilinear problems have been the prototype of many researches. The most
studied one being the $p$-Laplacian for large $p>1$,%
\[
-\Delta_{p}u=\operatorname{div}\left(  \left\vert \nabla u\right\vert
^{p-2}\nabla u\right)  =0,\text{ }x\in\Omega.
\]
Early studies in this respect dates back to 1980s, e.g. the torsional creep
problem \cite{bhattacharya1989limits} where the main interest lies in the
properties of solutions when $p$ is large. An interesting result is that for
large $p$, the $p$-Laplacian equation is close to the infinity Laplacian
equation,%
\[
-\Delta_{\infty}u=0,\text{ }x\in\Omega.
\]
This observation is also exploited in the study of sandpiles
\cite{aronsson1996fast}\cite{evans1997fast} where $p$-Laplacian with large $p$
is used as a simplistic approximation of collapsing sandpile. The asymptotic
analysis for the $p$-Laplacian was later extended to the $p\left(  x\right)
$-Laplacian \cite{manfredi2010limits}. When $p\left(  x\right)  $ is unbounded
in a subdomain $D\subset\Omega$ \cite{manfredi2009p}, the $p\left(  x\right)
$-harmonic function is related to the infinity Laplacian equation subject to a
mixed Dirichlet-Neumann boundary condition. In the subdmain $D$, the mixed
Dirichlet-Neumann problem%
\[
\left\{
\begin{array}
[c]{ll}%
-\Delta_{\infty}u=0, & x\in D;\\
u=g, & x\in\partial D;\\
\dfrac{\partial u}{\partial\nu}=0, & x\in\partial\Omega\backslash\partial D,
\end{array}
\right.
\]
coincides with the characterization of the tug-of-war game \cite{peres2009tug}%
\cite{charro2009mixed}. In addition to the $p\left(  x\right)  $-Laplacian,
limit of anisotropic $p$-Laplacian \cite{perez2011anisotropic} and isotropic
$\varphi$-Laplacian \cite{stancu2018asymptotic} were also considered. Another
thread of studies on limiting behaviour of $\left(
\ref{prototype_quasilinear}\right)  $ is mostly variational and involves the
convergence of the associated energies, e.g. \cite{bocea2015gamma}.

Motivated by these works, we intend to present a complete study of $\left(
\ref{prototype_quasilinear_n}\right)  $ and the related extremal problem
$\left(  \ref{prototype_quasilinear}\right)  $, from the perspectives of both
partial differential equations and variational convergences. The quasilinear
problems $\left(  \ref{prototype_quasilinear_n}\right)  $ and the extremal
problem $\left(  \ref{prototype_quasilinear}\right)  $ encompass the $p\left(
x\right)  $-Laplacian problem considered in \cite{manfredi2009p}. However, as
noted in \cite{harjulehto2016generalized}, going from the variable exponent
case to the general case requires new techniques on the level of both function
spaces and partial differential equations. There are also some new features
here. The general structure of the equation $\left(
\ref{prototype_quasilinear_n}\right)  $ allows complex interplay between the
$x$-variable and the gradient $\nabla u\left(  x\right)  $. Many properties
and results present for the case of $p\left(  x\right)  $-Laplacian are not
available for $\left(  \ref{prototype_quasilinear_n}\right)  $. Most
prominently, several new building blocks, particularly the uniqueness for the
limiting equation, some subtle estimates and the embedding in the generalized
Sobolev-Orlicz space with discontinuous coefficients $\mathcal{A}_{\infty}$
have to be built in the course of solving $\left(  \ref{prototype_quasilinear}%
\right)  $.

We first give a rough sketch of our approaches and results. Here in this
article, we do not try to pursue full generality, so we restrict ourselves to
quasilinear problems of the following form%

\begin{equation}
\left\{
\begin{array}
[c]{ll}%
-\operatorname{div}\left(  \dfrac{a_{n}\left(  x,\left\vert \nabla u\left(
x\right)  \right\vert \right)  }{\left\vert \nabla u\left(  x\right)
\right\vert }\nabla u\left(  x\right)  \right)  =0, & x\in\Omega;\\
u\left(  x\right)  =g\left(  x\right)  , & x\in\partial\Omega,
\end{array}
\right.  \label{general_quasilinear_n_intro}%
\end{equation}
where $\Omega$ is a bounded smooth domains of $\mathbb{R}^{d}$, $g\in
Lip\left(  \partial\Omega\right)  $ and%
\[
a_{n}\left(  x,s\right)  =\frac{\varphi_{n}\left(  x,s\right)  }{\varphi
_{n}\left(  x,1\right)  }.
\]
The assumptions on $\varphi_{n}$ will be made clear in section \ref{sec_pre}%
\ (ref. $\left(  \ref{assumption_phi_x}\right)  $). The natural energies
associated with this equation is denoted by%
\[
E_{n}\left(  u\right)  :u\in W^{1,\Phi_{n}\left(  \cdot\right)  }\left(
\Omega\right)  \mapsto
{\displaystyle\int_{\Omega}}
\dfrac{\Phi_{n}\left(  x,\left\vert \nabla u\right\vert \right)  }{\varphi
_{n}\left(  x,1\right)  }dx,
\]
where%
\[
\Phi_{n}\left(  x,s\right)  =\int_{0}^{s}\varphi_{n}\left(  x,t\right)  dt.
\]
We intend to let $n\rightarrow\infty$ and get limits of solutions $u_{n}\ $of
$\left(  \ref{general_quasilinear_n_intro}\right)  $, so require some
estimates independent of $n$. A key ingredient is a uniform gradient estimate
for the weak solutions of $\left(  \ref{general_quasilinear_n_intro}\right)
$. In obtaining the uniform gradient estimate, the boundary data $g$ seems to
play a role. The first evidence for $g$'s effects is revealed by the $\Gamma
$-convergence of the extended natural energies $\tilde{E}_{n}\left(  u\right)
$ associated with $\left(  \ref{general_quasilinear_n_intro}\right)  ,$%
\[
\tilde{E}_{n}\left(  u\right)  :u\in L^{1}\left(  \Omega\right)
\mapsto\left\{
\begin{array}
[c]{ll}
{\displaystyle\int_{\Omega}}
\dfrac{\Phi_{n}\left(  x,\left\vert \nabla u\right\vert \right)  }{\varphi
_{n}\left(  x,1\right)  }dx, & u\in W^{1,\Phi_{n}\left(  \cdot\right)
}\left(  \Omega\right)  ;\\
\infty, & \text{otherwise.}%
\end{array}
\right.
\]
Hereafter when there is no confusion, we will use the same symbol $E_{n}$ to
denote the natural energy on $W^{1,\Phi_{n}\left(  \cdot\right)  }\left(
\Omega\right)  $ and its extension to $L^{1}\left(  \Omega\right)  $. The
$\Gamma$-convergence result suggests that, in order for $E_{n}$ to be bounded,
we assume%
\begin{equation}
\left\Vert g\right\Vert _{Lip\left(  \partial\Omega\right)  }\leqslant1.
\label{assumption_g_lip}%
\end{equation}
Here $\left\Vert g\right\Vert _{Lip\left(  \partial\Omega\right)  }$ is the
Lipschitz constant of $g\in Lip\left(  \partial\Omega\right)  $, i.e.,%
\[
\left\Vert g\right\Vert _{Lip\left(  \partial\Omega\right)  }\triangleq
\sup_{\substack{x,y\in\partial\Omega\\x\neq y}}\frac{\left\vert g\left(
x\right)  -g\left(  y\right)  \right\vert }{\left\vert x-y\right\vert }.
\]
Previous studies for $p\left(  x\right)  $-Laplacian are mostly restricted to
the assumption $\left(  \ref{assumption_g_lip}\right)  $, e.g.
\cite{manfredi2009p}\cite{manfredi2010limits}. In fact we show that, for
convex domains, $\left(  \ref{assumption_g_lip}\right)  $ is a necessary and
sufficient condition for $E_{n}$ to be bounded. We also show that the required
uniform gradient estimate can be obtained without the assumption $\left\Vert
g\right\Vert _{Lip\left(  \partial\Omega\right)  }\leqslant1$, although
$\left\Vert g\right\Vert _{Lip\left(  \partial\Omega\right)  }>1$ might cause
the energies $E_{n}$ to explode as $n\rightarrow\infty$.

Once we get the required estimates, the convergence of solutions can be proved
and the limiting equation for $\left(  \ref{general_quasilinear_n_intro}%
\right)  $ turns out to be%
\begin{equation}
\left\{
\begin{array}
[c]{ll}%
-\Delta_{\infty}u-\left\vert \nabla u\right\vert \left\langle \Lambda\left(
x,\left\vert \nabla u\right\vert \right)  ,\nabla u\right\rangle =0, &
x\in\Omega;\\
u=g, & x\in\partial\Omega.
\end{array}
\right.  \label{equ_limit_intro}%
\end{equation}
Now here is one difficulty. Since this equation is not strictly coercive in
$u$, uniqueness for such equation is usually difficult to obtain. The
classical comparison principle does not readily apply and the perturbation
technique suggested by the User's guide \cite{crandall1992user}\ does not seem
to work either. So we have to prove a comparison principle for this limiting
equation. In the course of proving the uniqueness, some technical inequalities
have a crucial role to play. Once the uniqueness is obtained, the convergence
of the solutions for $\left(  \ref{general_quasilinear_n_intro}\right)  $ follows.

Next we look into the extremal equation%

\begin{equation}
\left\{
\begin{array}
[c]{ll}%
-\operatorname{div}\left(  \dfrac{a_{\infty}\left(  x,\left\vert \nabla
u\left(  x\right)  \right\vert \right)  }{\left\vert \nabla u\left(  x\right)
\right\vert }\nabla u\left(  x\right)  \right)  =0, & x\in\Omega;\\
u\left(  x\right)  =g\left(  x\right)  , & x\in\partial\Omega,
\end{array}
\right.  \label{equ_extremal_intro}%
\end{equation}
where $a_{\infty}\left(  x,s\right)  :\Omega\times\left[  0,\infty\right)
\rightarrow\mathbb{R\cup}\left\{  +\infty\right\}  $ is defined piecewise as
below. Let $\Omega_{\infty}$ be a smooth subdomain of $\Omega$, $a\left(
x,s\right)  $ an appropriate function on $\left(  \Omega\backslash
\Omega_{\infty}\right)  \times\left[  0,\infty\right)  $. We write
\begin{equation}
a_{\infty}\left(  x,s\right)  =\left\{
\begin{array}
[c]{ll}%
0, & s\leqslant1,x\in\Omega_{\infty};\\
\infty, & s>1,x\in\Omega_{\infty}.
\end{array}
\right.  \label{equ_extremal_a1}%
\end{equation}
and%
\begin{equation}
a_{\infty}\left(  x,s\right)  =a\left(  x,s\right)  \text{ for }x\in
\Omega\backslash\Omega_{\infty},\text{ }s\geqslant0. \label{equ_extremal_a2}%
\end{equation}
The basic question about understanding $\left(  \ref{equ_extremal_intro}%
\right)  $ is whether it has a solution and in what sense. If it does have a
solution in some reasonable sense, then how regular it is and what values does
it assume on the boundary $\partial\Omega_{\infty}$ where $a_{\infty}\left(
x,s\right)  $ becomes discontinuous.

The idea of solving $\left(  \ref{equ_extremal_intro}\right)  $ is to choose
$\tilde{a}_{n}\left(  x,s\right)  $ which approaches $\infty$ on $\bar{\Omega
}\times\left(  1,\infty\right)  $ and $0$ on $\bar{\Omega}\times\left[
0,1\right)  $ in a way that suits our purpose$.$ Such $\tilde{a}_{n}$ clearly
approximates $a_{\infty}$ in $\bar{\Omega}\times\left[  0,\infty\right)  .$
Then the study of $\left(  \ref{equ_extremal_intro}\right)  $ comes down to
that of the limiting behaviour of $\left(  \ref{general_quasilinear_n_intro}%
\right)  $ with $a_{n}$ replaced by $\tilde{a}_{n}$. In view of the form of
$a_{\infty}\left(  x,s\right)  ,$ i.e. $\left(  \ref{equ_extremal_a1}\right)
\left(  \ref{equ_extremal_a2}\right)  $, the approximations $\tilde{a}%
_{n}\left(  x,s\right)  $ are not continuous on $\bar{\Omega}\times\left[
0,\infty\right)  $ unless $\Omega_{\infty}=\Omega$. However, the classical
Poincare's inequality does not apply to the generalized Sobolev-Orlicz
space associated with discontinuous $\tilde{a}_{n}\left(  x,s\right)  .$ We
have to prove a Poincare's inequality for discontinuous discontinuous
$\tilde{a}_{n}\left(  x,s\right)  $ so that we can find weak solutions.

In the appendix, we prove several technical inequalities which play essential
roles in proving the uniqueness for the limiting equation and the convergence
of solutions. These inequalities will also be of independent interest.

\begin{remark}
\label{Rmk_g}Throughout this article we assume $g\in Lip\left(  \partial
\Omega\right)  $. To save us from redundant notations, we will use the same
symbol $g$ to mean both the function on $\partial\Omega$ and its absolutely
minimal Lipschitz extension (AMLE for short) into $\Omega$. Then we generally
have%
\[
\left\Vert \nabla g\right\Vert _{L^{\infty}\left(  \Omega\right)  }%
\leqslant\left\Vert g\right\Vert _{Lip\left(  \partial\Omega\right)  }.
\]

\end{remark}

The organization of this article is as follows. Section \ref{sec_pre} provides
the basic facts about the generalized Sobolev-Orlicz space, some utility
results are proved. In section \ref{sec_gamma}, we study the $\Gamma
$-convergence of the natural energies associated with $\left(
\ref{general_quasilinear_n_intro}\right)  $. Section \ref{sec_lim_g1} and
section \ref{sec_lim_g2} are devoted to the convergence of solutions of
$\left(  \ref{general_quasilinear_n_intro}\right)  $, both the case with
$\left\Vert g\right\Vert _{Lip\left(  \partial\Omega\right)  }\leqslant1$ and
without $\left\Vert g\right\Vert _{Lip\left(  \partial\Omega\right)
}\leqslant1$ are discussed. In section \ref{sec_comparison}, we prove a
comparison principle for the limiting equation $\left(  \ref{equ_limit_intro}%
\right)  $, from which the uniqueness for $\left(  \ref{equ_limit_intro}%
\right)  $ is derived. In the the remaining sections, we come back to the
extremal problem $\left(  \ref{equ_extremal_intro}\right)  $, a version of
Poincare's inequality is proved in section \ref{sec_Poincare} which
eventually enables us to find a solution for the extremal problem. Uniqueness
and gradient estimate for the solution of $\left(  \ref{equ_extremal_intro}%
\right)  $ are obtained. Finally in the appendix we collect some useful
technical inequalites.

\section{Acknowledgement}

The author is grateful to Prof. Yoshikazu Giga for his kindness and for the
careful reading by his colleague of the previous version of this paper and
generously providing insightful detailed comments which substantially improved
the paper. The proofs of several of the main results (e.g. Theorem
\ref{thm_weaksol_viscositysol}, Theorem \ref{thm_convergence_main}, Lemma
\ref{Lm_orlicz_embedding}, Theorem \ref{Thm_poincare_discontinuous} to name a
few) have greatly benefited from the comments. An alternative short proof is
suggested for the Lemma \ref{inequality_main} $\left(
\ref{inequality_main_inq2}\right)  $, although the author has kept the
original proof since it is referenced by other parts.

\section{\label{sec_pre}Preliminaries}

This section introduces some basics about Sobolev-Orlicz space
\cite{harjulehto2019orlicz,harjulehto2016generalized} and a few derived
properties specific to our problems, which will prepare us for the work that follows.

\subsection{$\Phi$-function and Sobolev-Orlicz space}

Assume that $\varphi\in C\left(  \left[  0,\infty\right)  \right)  \cap
C^{1}\left(  \left(  0,\infty\right)  \right)  $, $\varphi\left(  s\right)
>0$ for $s>0$ and there are constants $p^{-},$ $p^{+}$ such that
\begin{equation}
0<p^{-}-1\leqslant\frac{s\varphi^{\prime}\left(  s\right)  }{\varphi\left(
s\right)  }\leqslant p^{+}-1\text{ for }s>0. \label{assumption_phi}%
\end{equation}

\begin{definition}
\label{def_Phi_func}Let $\varphi$ satisfy $\left(  \ref{assumption_phi}%
\right)  $. The function,%
\begin{equation}
\Phi\left(  s\right)  =\int_{0}^{s}\varphi\left(  t\right)  dt
\label{equ_PHI_phi}%
\end{equation}
is\ referred to as a $\Phi$\textit{-function}.
\end{definition}

It is clear that a $\Phi$\textit{-function }$\Phi\left(  s\right)  :\left[
0,\infty\right)  \mapsto\left[  0,\infty\right)  $ is a convex, (strictly)
increasing function such that $\Phi\left(  0\right)  =0$. In addition
$\lim_{s\rightarrow\infty}\Phi\left(  s\right)  =\infty$, since non-constant
convex function is unbounded.

A $\Phi$-function is a $N$-function in the sense of \cite{adams2003sobolev},
but not vice versa since the definition of $N$-functions does not require
convexity and differentiability. We cite the following Lemma from
\cite{lieberman1991natural}.

\begin{lemma}
\label{lemma_lieberman}Let $\varphi\in C^{1}\left(  0,\infty\right)  $ satisfy
$\left(  \ref{assumption_phi}\right)  $. Then $\Phi\in C^{2}$ is convex and
\[
\frac{s\varphi\left(  s\right)  }{p^{+}}\leqslant\Phi\left(  s\right)
\leqslant\frac{s\varphi\left(  s\right)  }{p^{-}}\text{ for }s>0;
\]
or equivalently%
\begin{equation}
p^{-}\leqslant\frac{s\varphi\left(  s\right)  }{\Phi\left(  s\right)
}\leqslant p^{+}\text{ for }s>0; \label{inequality_PHI}%
\end{equation}

\end{lemma}

\begin{lemma}
\label{lemma_delta2_condition}Let $\alpha>0$ and $H,$ $h:\left[
0,\infty\right)  \mapsto\left[  0,\infty\right)  $ be differentiable functions
such that $h=H^{\prime}$. Then the function $s\mapsto s^{-\alpha}H\left(
s\right)  $ is%
\[
\text{increasing if and only if }\alpha\leqslant\inf_{t>0}\frac{th\left(
t\right)  }{H\left(  t\right)  }\text{,}%
\]
and%
\[
\text{decreasing if and only if }\alpha\geqslant\sup_{t>0}\frac{th\left(
t\right)  }{H\left(  t\right)  }\text{.}%
\]
Moreover, let $p_{1},p_{2}>1$, then
\begin{equation}
p_{1}\leqslant\frac{sh\left(  s\right)  }{H\left(  s\right)  }\leqslant
p_{2}\text{ for }s>0 \label{inequality_PHI_H}%
\end{equation}
if and only if%
\begin{equation}
\min\left\{  \varrho^{p_{1}},\varrho^{p_{2}}\right\}  H\left(  s\right)
\leqslant H\left(  \varrho s\right)  \leqslant\max\left\{  \varrho^{p_{1}%
},\varrho^{p_{2}}\right\}  H\left(  s\right)  \text{ for }\varrho,s\geqslant0.
\label{lemma_delta2_condition_inq0}%
\end{equation}

\end{lemma}

\begin{proof}
The first part of the conclusion follows easily by observing
\[
\left(  s^{-\alpha}H\left(  s\right)  \right)  ^{\prime}=s^{-\alpha}h\left(
s\right)  -\alpha s^{-\alpha-1}H\left(  s\right)  =s^{-\alpha-1}\left(
sh\left(  s\right)  -\alpha H\left(  s\right)  \right)  .
\]

It then follows that $\left(  \ref{inequality_PHI_H}\right)  $ is equivalent
to%
\[
s^{-p_{1}}H\left(  s\right)  \text{ is non-decreasing, }s^{-p_{2}}H\left(
s\right)  \text{ is non-increasing.}%
\]
Moreover%
\[
s^{-p_{1}}H\left(  s\right)  \text{ is non-decreasing}\iff\left\{
\begin{array}
[c]{cl}%
\left(  \varrho s\right)  ^{-p_{1}}H\left(  \varrho s\right)  \geqslant
s^{-p_{1}}H\left(  s\right)  , & \varrho\geqslant1\\
\left(  \varrho s\right)  ^{-p_{1}}H\left(  \varrho s\right)  \leqslant
s^{-p_{1}}H\left(  s\right)  , & 0\leqslant\varrho<1
\end{array}
\right.  ,
\]
and%
\[
s^{-p_{2}}H\left(  s\right)  \text{ is non-increasing}\iff\left\{
\begin{array}
[c]{cl}%
\left(  \varrho s\right)  ^{-p_{2}}H\left(  \varrho s\right)  \leqslant
s^{-p_{2}}H\left(  s\right)  , & \varrho\geqslant1\\
\left(  \varrho s\right)  ^{-p_{2}}H\left(  \varrho s\right)  \geqslant
s^{-p_{2}}H\left(  s\right)  , & 0\leqslant\varrho<1
\end{array}
\right.  .
\]
Putting these together yields that $\left(  \ref{inequality_PHI_H}\right)  $
is equivalent to%
\[
\left\{
\begin{array}
[c]{cl}%
\varrho^{p_{1}}H\left(  s\right)  \leqslant H\left(  \varrho s\right)
\leqslant\varrho^{p_{2}}H\left(  s\right)  , & \varrho\geqslant1\\
\varrho^{p_{2}}H\left(  s\right)  \leqslant H\left(  \varrho s\right)
\leqslant\varrho^{p_{1}}H\left(  s\right)  , & 0\leqslant\varrho<1
\end{array}
\right.  .
\]
Therefore $\left(  \ref{inequality_PHI_H}\right)  $ is equivalent to $\left(
\ref{lemma_delta2_condition_inq0}\right)  .$
\end{proof}

In view of Lemma \ref{lemma_lieberman} and Lemma \ref{lemma_delta2_condition},
any $\Phi$\textit{-}function\textit{ }$\Phi$ satisfies%
\begin{equation}
\min\left\{  \varrho^{p^{-}},\varrho^{p^{+}}\right\}  \Phi\left(  s\right)
\leqslant\Phi\left(  \varrho s\right)  \leqslant\max\left\{  \varrho^{p^{-}%
},\varrho^{p^{+}}\right\}  \Phi\left(  s\right)  \text{ for }\varrho
,s\geqslant0. \label{lemma_delta2_condition_inq1}%
\end{equation}
One implication of $\left(  \ref{lemma_delta2_condition_inq1}\right)  $ is
that $\Phi\left(  \varrho\right)  $ roughly stays between $\varrho^{p^{+}}%
$\ and $\varrho^{p^{-}}$ for all $\varrho\geqslant0,$ as can be seen by taking
$s=1$ in $\left(  \ref{lemma_delta2_condition_inq1}\right)  $. Another
implication of $\left(  \ref{lemma_delta2_condition_inq1}\right)  $ is the
following $\Delta_{2}$ condition.

\begin{definition}
A function $G\left(  s\right)  $ satisfies the $\Delta_{2}$ condition
(\cite[8.6]{adams2003sobolev}) or doubling condition (\cite[Definition
2.2.5]{harjulehto2019orlicz}) if for some constant $K>2,$%
\[
G\left(  2s\right)  \leqslant KG\left(  s\right)  \text{ for }s>0.
\]

\end{definition}

Setting $\varrho=2$ in $\left(  \ref{lemma_delta2_condition_inq1}\right)  $,
we get%
\begin{equation}
\Phi\left(  2s\right)  \leqslant2^{p^{+}}\Phi\left(  s\right)  \text{ for
}s>0, \label{delta2_Inq}%
\end{equation}
i.e. $\Phi$-functions as defined in Definition \ref{def_Phi_func} satisfy the
$\Delta_{2}$ condition.

Each $\Phi$-function\textit{ }$\Phi$ is associated with a conjugate function%
\begin{equation}
\Phi^{\ast}\left(  t\right)  =\sup_{s\geqslant0}\left\{  ts-\Phi\left(
s\right)  \right\}  , \label{def_conjugate_phi}%
\end{equation}
which again is a $\Phi$-function \cite[Lemma 2.5]{fukagai2006positive}
satisfying a condition similar to $\left(  \ref{inequality_PHI}\right)  $
\begin{equation}
\frac{p^{+}}{p^{+}-1}\leqslant\frac{s\left(  \Phi^{\ast}\right)  ^{\prime
}\left(  s\right)  }{\Phi^{\ast}\left(  s\right)  }\leqslant\frac{p^{-}}%
{p^{-}-1}\text{ for }s>0. \label{inequality_PHI_*}%
\end{equation}
This together with Lemma \ref{lemma_delta2_condition} gives%
\[
\min\left\{  \varrho^{p_{1}},\varrho^{p_{2}}\right\}  \Phi^{\ast}\left(
s\right)  \leqslant\Phi^{\ast}\left(  \varrho s\right)  \leqslant\max\left\{
\varrho^{p_{1}},\varrho^{p_{2}}\right\}  \Phi^{\ast}\left(  s\right)  \text{
for }\varrho,s\geqslant0,
\]
where $p_{1}=p^{+}\left/  \left(  p^{+}-1\right)  \right.  ,$ $p_{2}%
=p^{-}\left/  \left(  p^{-}-1\right)  \right.  .$ Noting $p_{1}\leqslant
p_{2},$ we get%
\begin{equation}
\Phi^{\ast}\left(  2t\right)  \leqslant2^{p_{2}}\Phi^{\ast}\left(  t\right)
=2^{p^{-}\left/  \left(  p^{-}-1\right)  \right.  }\Phi^{\ast}\left(
t\right)  \text{ for }t>0. \label{delta2*_Inq}%
\end{equation}
Therefore $\Phi^{\ast}$ satisfies the $\Delta_{2}$ condition.

\begin{lemma}
\label{lemma_convexity}Assume $\Phi$ is a $\Phi$-function. The mapping%
\[
s\mapsto\Phi\left(  s^{1/2}\right)
\]
is convex if $p^{-}\geqslant2$. The mapping
\[
s\mapsto\Phi^{1+\gamma}\left(  s^{1/2}\right)
\]
is convex if $\gamma\geqslant\left(  \inf_{t>0}\frac{t\varphi\left(  t\right)
}{\Phi\left(  t\right)  }\right)  ^{-1}$. In particular, under condition
$\left(  \ref{assumption_phi}\right)  $, $\Phi^{1+\gamma}\left(
s^{1/2}\right)  $ is a convex function of $s$ for any $\gamma\geqslant1.$
\end{lemma}

\begin{proof}
The proof can be completed by straightforward calculations.%
\[
\frac{d}{ds}\Phi\left(  s^{1/2}\right)  =\varphi\left(  s^{1/2}\right)
\frac{1}{2s^{1/2}},
\]%
\[
\frac{d^{2}}{ds^{2}}\Phi\left(  s^{1/2}\right)  =\varphi^{\prime}\left(
s^{1/2}\right)  \frac{1}{4s}-\varphi\left(  s^{1/2}\right)  \frac{1}{4s^{3/2}%
}=\frac{1}{4s^{3/2}}\left(  s^{1/2}\varphi^{\prime}\left(  s^{1/2}\right)
-\varphi\left(  s^{1/2}\right)  \right)  .
\]
The last term is nonegative if $p^{-}\geqslant2,$ hence $\Phi\left(
s^{1/2}\right)  $ is convex whenever $p^{-}\geqslant2$. In the same way,
\[
\frac{d}{ds}\Phi^{1+\gamma}\left(  s^{1/2}\right)  =\left(  1+\gamma\right)
\Phi^{\gamma}\left(  s^{1/2}\right)  \varphi\left(  s^{1/2}\right)  \frac
{1}{2s^{1/2}},
\]%
\begin{align*}
&  \frac{d^{2}}{ds^{2}}\Phi^{1+\gamma}\left(  s^{1/2}\right) \\
&  =\left(  1+\gamma\right)  \left[  \gamma\Phi^{\gamma-1}\left(
s^{1/2}\right)  \varphi^{2}\left(  s^{1/2}\right)  \frac{1}{4s}+\Phi^{\gamma
}\left(  s^{1/2}\right)  \varphi^{\prime}\left(  s^{1/2}\right)  \frac{1}%
{4s}-\Phi^{\gamma}\left(  s^{1/2}\right)  \varphi\left(  s^{1/2}\right)
\frac{1}{4s^{3/2}}\right] \\
&  =\left(  1+\gamma\right)  \frac{1}{4s^{3/2}}\Phi^{\gamma-1}\left(
s^{1/2}\right)  \left[  \gamma s^{1/2}\varphi^{2}\left(  s^{1/2}\right)
+s^{1/2}\Phi\left(  s^{1/2}\right)  \varphi^{\prime}\left(  s^{1/2}\right)
-\Phi\left(  s^{1/2}\right)  \varphi\left(  s^{1/2}\right)  \right]
\end{align*}
Note that the middle term in the last bracket is positive. By assumption%
\begin{align*}
\gamma s^{1/2}\varphi^{2}\left(  s^{1/2}\right)  -\Phi\left(  s^{1/2}\right)
\varphi\left(  s^{1/2}\right)   &  =\varphi\left(  s^{1/2}\right)  \Phi\left(
s^{1/2}\right)  \left[  \gamma\frac{s^{1/2}\varphi\left(  s^{1/2}\right)
}{\Phi\left(  s^{1/2}\right)  }-1\right] \\
&  \geqslant\varphi\left(  s^{1/2}\right)  \Phi\left(  s^{1/2}\right)  \left[
\gamma\inf_{t>0}\frac{t\varphi\left(  t\right)  }{\Phi\left(  t\right)
}-1\right]  \geqslant0.
\end{align*}
Hence $\Phi^{1+\gamma}\left(  s^{1/2}\right)  $ is convex whenever $\gamma
\inf_{t>0}\frac{t\varphi\left(  t\right)  }{\Phi\left(  t\right)  }%
-1\geqslant0.$ By $\left(  \ref{assumption_phi}\right)  $ and Lemma
\ref{lemma_lieberman},%
\[
\inf_{t>0}\frac{t\varphi\left(  t\right)  }{\Phi\left(  t\right)  }\geqslant
p^{-}>1.
\]
So we have%
\[
\gamma\geqslant\left(  \inf_{t>0}\frac{t\varphi\left(  t\right)  }{\Phi\left(
t\right)  }\right)  ^{-1}\text{ for all }\gamma\geqslant1,
\]
therefore $\Phi^{1+\gamma}\left(  s^{1/2}\right)  $ is convex in $s$ for any
$\gamma\geqslant1.$
\end{proof}

\subsection{Generalized $\Phi$-function and Musielak-Orlicz space}

Let $\Phi\left(  x,s\right)  :\Omega\times\left[  0,\infty\right)
\mapsto\mathbb{R}$. Write%
\[
\varphi\left(  x,s\right)  =\frac{\partial}{\partial s}\Phi\left(  x,s\right)
.
\]
Assume that $\varphi\in C\left(  \bar{\Omega}\times\left[  0,\infty\right)
\right)  \cap C^{1}\left(  \Omega\times\left(  0,\infty\right)  \right)  ,$
$\varphi\left(  x,s\right)  >0$ for $x\in\Omega$, $s>0$ and there are
constants $p^{-},$ $p^{+}>0$ such that
\begin{equation}
0<p^{-}-1\leqslant\frac{s\partial_{s}\varphi\left(  x,s\right)  }%
{\varphi\left(  x,s\right)  }\leqslant p^{+}-1\text{ for }x\in\Omega\text{,
}s>0. \label{assumption_phi_x}%
\end{equation}

\begin{definition}
\label{def_gen_Phi}Let $\varphi$ satisfy $\left(  \ref{assumption_phi_x}%
\right)  $. The function,%
\begin{equation}
\Phi\left(  x,s\right)  =\int_{0}^{s}\varphi\left(  x,t\right)  dt
\label{def_PHI}%
\end{equation}
is\ referred to as a generalized $\Phi$\textit{-function}. Clearly
$\Phi\left(  x,\cdot\right)  :\left[  0,\infty\right)  \mapsto\left[
0,\infty\right)  $ is a $\Phi$-function for each $x\in\Omega.$
\end{definition}

By Lemma \ref{lemma_lieberman}, $\Phi\left(  x,s\right)  $ satisfies%
\begin{equation}
p^{-}\leqslant\frac{s\varphi\left(  x,s\right)  }{\Phi\left(  x,s\right)
}\leqslant p^{+}\text{ for }x\in\Omega\text{, }s>0. \label{inequality_PHI_x}%
\end{equation}

Sometimes we will also need to work with%
\[
\frac{\Phi\left(  x,s\right)  }{\varphi\left(  x,1\right)  }\text{ and }%
\frac{\varphi\left(  x,s\right)  }{\varphi\left(  x,1\right)  }\text{,}%
\]
for which we fix the notation%
\begin{equation}
a\left(  x,s\right)  =\frac{\varphi\left(  x,s\right)  }{\varphi\left(
x,1\right)  }. \label{eq:def_a_xs}%
\end{equation}

\begin{remark}
\label{rmk_phi_a}It is useful to observe that $\Phi\left(  x,s\right)  $
(resp. $\varphi\left(  x,s\right)  $) satisfies $\left(
\ref{assumption_phi_x}\right)  \left(  \ref{inequality_PHI_x}\right)  $ if and
only if $\frac{\Phi\left(  x,s\right)  }{\varphi\left(  x,1\right)  }$ (resp.
$\frac{\varphi\left(  x,s\right)  }{\varphi\left(  x,1\right)  }$) satisfies
$\left(  \ref{assumption_phi_x}\right)  \left(  \ref{inequality_PHI_x}\right)
$.
\end{remark}

Another assumption which accompanies an arbitrary generalized $\Phi$-function
$\Phi$ in the\ Musielak-Orlicz space setting is the following
\textit{normalization condition} \cite[Definition 3.7.1]{harjulehto2019orlicz}%
, i.e. for $\beta\in\left(  0,1\right]  ,$%
\begin{equation}
\beta\leqslant\Phi^{-1}\left(  x,1\right)  \leqslant\beta^{-1}\text{ for
}a.e.\text{ }x\in\Omega, \label{PHI_normalization_x}%
\end{equation}
where $\Phi^{-1}\left(  x,\cdot\right)  $ is the generalized inverse of
$\Phi\left(  x,\cdot\right)  $. This condition, which is intended to exclude
difficulty created by such function as $\Phi\left(  x,s\right)  =\left\vert
x\right\vert ^{2}s^{2}$, roughly means that for each $x\in\Omega$, $\Phi^{-1}$
is \textit{unweighted}, i.e. $\Phi^{-1}\left(  x,1\right)  \approx1$. In
general, this is not the same as saying that $\Phi\left(  x,1\right)
\approx1$. However this is correct in our context, since $\Phi\left(
x,s\right)  $ is strictly increasing in $s$ and hence has an inverse. As we
see in the following Lemma, the normalization condition $\left(
\ref{PHI_normalization_x}\right)  $ is equivalent to $\left(
\ref{assumption_PHI_1}\right)  $ for any generalized $\Phi$-function
satisfiying $\left(  \ref{inequality_PHI_x}\right)  $.

\begin{lemma}
\label{lm_normalization}Let $\varphi\in C\left(  \bar{\Omega}\times\left[
0,\infty\right)  \right)  \cap C^{1}\left(  \Omega\times\left(  0,\infty
\right)  \right)  $ satisfy $\left(  \ref{assumption_phi_x}\right)  $\ and
$\Phi$ be given by $\left(  \ref{def_PHI}\right)  $. Then $\Phi$ satisfies the
normalization condition $\left(  \ref{PHI_normalization_x}\right)  $ if and
only if there are constants $c^{-},$ $c^{+}$ with $0<c^{-}\leqslant1\leqslant
c^{+}$ such that%
\begin{equation}
c^{-}\leqslant\Phi\left(  x,1\right)  \leqslant c^{+}\text{ for }x\in\Omega.
\label{assumption_PHI_1}%
\end{equation}
In particular, the constants in $\left(  \ref{assumption_PHI_1}\right)  $ can
be identified as $c^{-}=\beta^{p^{+}}$ and $c^{+}=\beta^{-p^{+}}$. Conversely
if $\left(  \ref{assumption_PHI_1}\right)  $ holds, then the normalization
condition $\left(  \ref{PHI_normalization_x}\right)  $ is fulfilled with%
\[
\beta=\min\left\{  \left(  c^{-}\right)  ^{\frac{1}{p^{-}}},\left(
c^{+}\right)  ^{-\frac{1}{p^{-}}}\right\}  .
\]

\end{lemma}

\begin{proof}
First observe that $\left(  \ref{PHI_normalization_x}\right)  $ is equivalent
to the following
\begin{equation}
\Phi\left(  x,\beta\right)  \leqslant1\text{ and }\Phi\left(  x,\beta
^{-1}\right)  \geqslant1\text{ for }a.e.x\in\Omega\label{PHI_normalization_x_}%
\end{equation}
for some $\beta\in\left(  0,1\right]  ,$ since the mapping $s\mapsto
\Phi\left(  x,s\right)  $ is strictly increasing on $\left(  0,\infty\right)
$ and has an inverse for each $x$. Now it suffices to show that $\left(
\ref{PHI_normalization_x_}\right)  $ is equivalent to $\left(
\ref{assumption_PHI_1}\right)  $. Suppose that $\left(  \ref{assumption_PHI_1}%
\right)  $ holds$.$ By virtue of Lemma \ref{lemma_delta2_condition}, for each
$x\in\Omega$, $\varrho>0$,
\begin{equation}
\min\left\{  \varrho^{p^{+}},\varrho^{p^{-}}\right\}  \Phi\left(  x,1\right)
\leqslant\Phi\left(  x,\varrho\right)  \leqslant\max\left\{  \varrho^{p^{+}%
},\varrho^{p^{-}}\right\}  \Phi\left(  x,1\right)  .
\label{lm_normalization_Inq1}%
\end{equation}
Using $\left(  \ref{assumption_PHI_1}\right)  $ we find%
\[
\min\left\{  \varrho^{p^{+}},\varrho^{p^{-}}\right\}  c^{-}\leqslant
\Phi\left(  x,\varrho\right)  \leqslant\max\left\{  \varrho^{p^{+}}%
,\varrho^{p^{-}}\right\}  c^{+}\text{ for }x\in\Omega.
\]
Hence setting $\varrho$ to
\[
\beta=\min\left\{  \left(  c^{-}\right)  ^{\frac{1}{p^{-}}},\left(
c^{+}\right)  ^{-\frac{1}{p^{-}}}\right\}  \in\left(  0,1\right]
\]
gives us
\begin{align*}
\Phi\left(  x,\beta\right)   &  \leqslant\max\left\{  \beta^{p^{+}}%
,\beta^{p^{-}}\right\}  c^{+}=\beta^{p^{-}}c^{+}\\
&  =\min\left\{  c^{-},\left(  c^{+}\right)  ^{-1}\right\}  c^{+}=\min\left\{
c^{-}c^{+},\left(  c^{+}\right)  ^{-1}c^{+}\right\}  \leqslant1,\text{ for
}x\in\Omega,
\end{align*}
while setting $\varrho$ to $\beta^{-1}$ yields
\begin{align*}
\Phi\left(  x,\beta^{-1}\right)   &  \geqslant\min\left\{  \beta^{-p^{+}%
},\beta^{-p^{-}}\right\}  c^{-}=\beta^{-p^{-}}c^{-}\\
&  =\max\left\{  \left(  c^{-}\right)  ^{-1},c^{+}\right\}  c^{-}=\max\left\{
\left(  c^{-}\right)  ^{-1}c^{-},c^{+}c^{-}\right\}  \geqslant1,\text{ for
}x\in\Omega.
\end{align*}
As a consequence $\left(  \ref{assumption_PHI_1}\right)  $ implies $\left(
\ref{PHI_normalization_x_}\right)  $. Conversely assume that $\Phi$ satisfies
$\left(  \ref{PHI_normalization_x_}\right)  $, which in combination with
$\left(  \ref{lm_normalization_Inq1}\right)  $ yields that$,$%
\[
\beta^{p^{+}}\Phi\left(  x,1\right)  =\min\left\{  \beta^{p^{+}},\beta^{p^{-}%
}\right\}  \Phi\left(  x,1\right)  \leqslant\Phi\left(  x,\beta\right)
\leqslant1,
\]
and%
\[
\beta^{-p^{+}}\Phi\left(  x,1\right)  =\max\left\{  \beta^{-p^{+}}%
,\beta^{-p^{-}}\right\}  \Phi\left(  x,1\right)  \geqslant\Phi\left(
x,\beta^{-1}\right)  \geqslant1.
\]
Hence%
\[
\beta^{p^{+}}\leqslant\Phi\left(  x,1\right)  \leqslant\beta^{-p^{+}},
\]
proving $\left(  \ref{assumption_PHI_1}\right)  $ with $c^{-}=\beta^{p^{+}%
}\leqslant1$ and $c^{+}=\beta^{-p^{+}}\geqslant1$.
\end{proof}

A useful observation derived from Lemma \ref{lm_normalization} and $\left(
\ref{inequality_PHI_x}\right)  $ is that $\varphi\left(  x,1\right)  $ has a
lower bound and a upper bounded independent of $x,$ i.e.
\begin{equation}
c^{+}p^{+}\geqslant\Phi\left(  x,1\right)  p^{+}\geqslant\varphi\left(
x,1\right)  \geqslant\Phi\left(  x,1\right)  p^{-}\geqslant c^{-}p^{-}>0\text{
for all }x\in\Omega. \label{assumption_phi_1_lbd}%
\end{equation}

\begin{definition}
Each $\Phi$ defines a modular for $u\in L^{1}\left(  \Omega\right)  ,$
\[
\varrho_{\Phi\left(  \cdot\right)  }\left(  u\right)  =\int_{\Omega}%
\Phi\left(  x,\left\vert u\left(  x\right)  \right\vert \right)  dx.
\]
The generalized Orlicz space (also Musielak-Orlicz space) $L^{\Phi\left(
\cdot\right)  }\left(  \Omega\right)  $ is defined as the subspace%
\[
L^{\Phi\left(  \cdot\right)  }\left(  \Omega\right)  =\left\{  u\in
L^{1}\left(  \Omega\right)  :\lim_{\lambda\rightarrow0}\varrho_{\Phi\left(
\cdot\right)  }\left(  \lambda u\right)  =0\right\}  ,
\]
equipped with the norm%
\[
\left\Vert u\right\Vert _{L^{\Phi\left(  \cdot\right)  }}=\inf\left\{
\lambda>0:\varrho_{\Phi\left(  \cdot\right)  }\left(  \frac{u}{\lambda
}\right)  \leqslant1\right\}  .
\]

\end{definition}

We refer to \cite[Lemma 3.2.2 (b)]{harjulehto2019orlicz} for the proof that
$\left\Vert \cdot\right\Vert _{L^{\Phi\left(  \cdot\right)  }}$ is indeed a
norm. Hereafter, instead of $\Phi$, we will use $\Phi\left(  \cdot\right)  $
in the subscripts and superscripts to emphasize the dependence on $x$.

\begin{lemma}
\label{lm_modulo_finite}If $u\in L^{\Phi\left(  \cdot\right)  }\left(
\Omega\right)  $, then%
\[
\min\left\{  \left\Vert u\right\Vert _{L^{\Phi\left(  \cdot\right)  }}^{p^{-}%
},\left\Vert u\right\Vert _{L^{\Phi\left(  \cdot\right)  }}^{p+}\right\}
\leqslant\int_{\Omega}\Phi\left(  x,\left\vert u\left(  x\right)  \right\vert
\right)  dx\leqslant\max\left\{  \left\Vert u\right\Vert _{L^{\Phi\left(
\cdot\right)  }}^{p^{-}},\left\Vert u\right\Vert _{L^{\Phi\left(
\cdot\right)  }}^{p+}\right\}  .
\]

\end{lemma}

\begin{proof}
The conclusion clearly holds if $\left\Vert u\right\Vert _{L^{\Phi\left(
\cdot\right)  }}=0$, which implies $u=0$ $a.e.$ $x\in\Omega$. Suppose
$\left\Vert u\right\Vert _{L^{\Phi\left(  \cdot\right)  }}>0$, let
$\lambda=\left\Vert u\right\Vert _{L^{\Phi\left(  \cdot\right)  }}$. Since
$u\in L^{\Phi\left(  \cdot\right)  }\left(  \Omega\right)  $, $\lambda<\infty
$. By virtue of Lemma \ref{lemma_delta2_condition} and the definition of
$\left\Vert \cdot\right\Vert _{L^{\Phi\left(  \cdot\right)  }},$ we have%
\[
\min\left\{  \lambda^{-p^{-}},\lambda^{-p^{+}}\right\}  \int_{\Omega}%
\Phi\left(  x,\left\vert u\left(  x\right)  \right\vert \right)
dx\leqslant\int_{\Omega}\Phi\left(  x,\frac{1}{\lambda}\left\vert u\left(
x\right)  \right\vert \right)  dx=\varrho_{\Phi\left(  \cdot\right)  }\left(
\frac{u}{\lambda}\right)  \leqslant1.
\]
Then%
\[
\int_{\Omega}\Phi\left(  x,\left\vert u\left(  x\right)  \right\vert \right)
dx\leqslant\max\left\{  \lambda^{p^{-}},\lambda^{p^{+}}\right\}  <\infty,
\]
and the right hand side of the desired inequality follows. Similarly using
Lemma \ref{lemma_delta2_condition} again, we have that for $\mu>0$,%
\begin{align*}
\int_{\Omega}\Phi\left(  x,\frac{1}{\mu}\left\vert u\left(  x\right)
\right\vert \right)  dx  &  \leqslant\max\left\{  \mu^{-p^{-}},\mu^{-p^{+}%
}\right\}  \int_{\Omega}\Phi\left(  x,\left\vert u\left(  x\right)
\right\vert \right)  dx\\
&  =\frac{\int_{\Omega}\Phi\left(  x,\left\vert u\left(  x\right)  \right\vert
\right)  dx}{\min\left\{  \mu^{p^{-}},\mu^{p^{+}}\right\}  }.
\end{align*}
If $\mu$ is taken such that%
\begin{equation}
\min\left\{  \mu^{p^{-}},\mu^{p^{+}}\right\}  =\int_{\Omega}\Phi\left(
x,\left\vert u\left(  x\right)  \right\vert \right)  dx,
\label{lm_modulo_finite_Inq1}%
\end{equation}
then%
\[
\int_{\Omega}\Phi\left(  x,\frac{1}{\mu}\left\vert u\left(  x\right)
\right\vert \right)  dx\leqslant1.
\]
Hence by the definition of $\left\Vert \cdot\right\Vert _{L^{\Phi\left(
\cdot\right)  }}$%
\begin{equation}
\left\Vert u\right\Vert _{L^{\Phi\left(  \cdot\right)  }}\leqslant\mu.
\label{lm_modulo_finite_Inq2}%
\end{equation}
Therefore we obtain from $\left(  \ref{lm_modulo_finite_Inq1}\right)  ,$
$\left(  \ref{lm_modulo_finite_Inq2}\right)  $ that%
\[
\min\left\{  \left\Vert u\right\Vert _{L^{\Phi\left(  \cdot\right)  }}^{p^{-}%
},\left\Vert u\right\Vert _{L^{\Phi\left(  \cdot\right)  }}^{p+}\right\}
\leqslant\min\left\{  \mu^{p^{-}},\mu^{p^{+}}\right\}  =\int_{\Omega}%
\Phi\left(  x,\left\vert u\left(  x\right)  \right\vert \right)  dx.
\]

\end{proof}

Similar to $\left(  \ref{def_conjugate_phi}\right)  $, we have a conjugate for
each generalized $\Phi$-function $\Phi,$%
\[
\Phi^{\ast}\left(  x,t\right)  =\sup_{s\geqslant0}\left\{  ts-\Phi\left(
x,s\right)  \right\}  .
\]
Recalling the discussion following Lemma \ref{lemma_delta2_condition}, we see
that both $\Phi$ and $\Phi^{\ast}$ satisfy $\Delta_{2}$ condition.

\begin{definition}
The Sobolev-Orlicz space $W^{1,\Phi\left(  \cdot\right)  }\left(
\Omega\right)  $ is the set of functions $u\in L^{\Phi\left(  \cdot\right)
}\left(  \Omega\right)  $ such that its weak derivatives $\partial
_{1}u,...,\partial_{d}u$ are members of $L^{\Phi\left(  \cdot\right)  }\left(
\Omega\right)  ,$ i.e.
\[
\left\Vert u\right\Vert _{W^{1,\Phi\left(  \cdot\right)  }}=\inf\left\{
\lambda>0:\varrho_{1,\Phi\left(  \cdot\right)  }\left(  \frac{u}{\lambda
}\right)  \leqslant1\right\}  <\infty,
\]
where%
\[
\varrho_{1,\Phi\left(  \cdot\right)  }\left(  u\right)  =\varrho_{\Phi\left(
\cdot\right)  }\left(  u\right)  +\sum_{k=1}^{d}\varrho_{\Phi\left(
\cdot\right)  }\left(  \partial_{k}u\right)
\]
The closure of $C_{0}^{\infty}\left(  \Omega\right)  $ in $W^{1,\Phi\left(
\cdot\right)  }\left(  \Omega\right)  $ under the norm $\left\Vert
\cdot\right\Vert _{W^{1,\Phi\left(  \cdot\right)  }}$ is denoted by
$W_{0}^{1,\Phi\left(  \cdot\right)  }\left(  \Omega\right)  .$
\end{definition}

Now we are in a position to state the following properties of $W^{1,\Phi
\left(  \cdot\right)  }\left(  \Omega\right)  $.

\begin{theorem}
[{\cite[Theorem 6.1.4]{harjulehto2019orlicz}}]\label{thm_reflexive}Assume that
$\Phi$ is a generalized $\Phi$-function verifying conditions $\left(
\ref{inequality_PHI_x}\right)  $ and $\left(  \ref{assumption_PHI_1}\right)
.$ Then $L^{\Phi\left(  \cdot\right)  }\left(  \Omega\right)  $, hence
$W^{1,\Phi\left(  \cdot\right)  }\left(  \Omega\right)  $, is a reflexive
separable Banach space.
\end{theorem}

\begin{theorem}
[{\cite[Lemma 6.1.6]{harjulehto2019orlicz}}]\label{thm_embedding}Assume that
$\Phi$ is a generalized $\Phi$-function verifying conditions $\left(
\ref{inequality_PHI_x}\right)  $ and $\left(  \ref{assumption_PHI_1}\right)
.$ There is $p>1$ such that for each $x\in\Omega$, $s\longmapsto s^{-p}%
\Phi\left(  x,s\right)  $ is increasing. Then $W^{1,\Phi\left(  \cdot\right)
}\left(  \Omega\right)  \subset W^{1,p}\left(  \Omega\right)  $.
\end{theorem}

One last ingredient we need is Poincare's inequality for $W_{0}%
^{1,\Phi\left(  \cdot\right)  }\left(  \Omega\right)  $. In general
Poincare's inequality does not hold for any generalized $\Phi$-function.
Additional continuity assumptions must be imposed. The function $\Phi\left(
x,s\right)  $ is said to satisfiy the \textit{local continuity condition} if
there exists $c\in\left(  0,1\right)  $ for which%
\begin{equation}
c\Phi^{-1}\left(  x,s\right)  \leqslant\Phi^{-1}\left(  y,s\right)  ,
\label{PHI_condition_A1}%
\end{equation}
for all $s\in\left[  1,\left\vert B\right\vert ^{-1}\right]  $, almost all
$x,y\in B\cap\Omega$ and any ball $B$ with $\left\vert B\right\vert
\leqslant1.$ For generalized $\Phi$-functions satisfying $\left(
\ref{inequality_PHI_x}\right)  $ and $\left(  \ref{assumption_PHI_1}\right)
$, it is straightforward to see that the local continuity condition $\left(
\ref{PHI_condition_A1}\right)  $ is equivalent (see also \cite[Corollary
4.1.6]{harjulehto2019orlicz}) to the following: there exists $c\in\left(
0,1\right)  $ for which%
\begin{equation}
\Phi\left(  x,cs\right)  \leqslant\Phi\left(  y,s\right)  ,
\label{PHI_condition_A11}%
\end{equation}
for all $\Phi\left(  y,s\right)  \in\left[  1,\left\vert B\right\vert
^{-1}\right]  $, almost all $x,y\in B\cap\Omega$ and any ball $B$ with
$\left\vert B\right\vert \leqslant1.$ Note that for the choice
\[
\Phi\left(  x,s\right)  =s^{p\left(  x\right)  },\text{for }x\in\Omega\text{
and }s>0,
\]
the local continuity condition is equivalent to the log-H\"{o}lder continuity
of $1/p\left(  x\right)  $ \cite[Proposition 7.1.2]{harjulehto2019orlicz},
which is required for most results of Lebesgue spaces with variable exponents
to work \cite{diening2011lebesgue}.

We now have Poincare's inequality for generalized $\Phi$-functions.

\begin{theorem}
[{\cite[Theorem 6.2.8]{harjulehto2019orlicz}}]\label{thm_poincare_inequality}%
Let $\Phi$ be a generalized $\Phi$-function satisfying $\left(
\ref{inequality_PHI_x}\right)  $, the normalization condition $\left(
\ref{assumption_PHI_1}\right)  $ and the local continuity condition $\left(
\ref{PHI_condition_A11}\right)  $. Then for every $u\in W_{0}^{1,\Phi\left(
\cdot\right)  }\left(  \Omega\right)  ,$%
\[
\left\Vert u\right\Vert _{L^{\Phi\left(  \cdot\right)  }}\leqslant C\left\Vert
\nabla u\right\Vert _{L^{\Phi\left(  \cdot\right)  }},
\]
where $C>0$ depends on $\Omega$ and the constant in the local continuity condition.
\end{theorem}

\section{\label{sec_gamma}$\Gamma$-convergence of the Natural Energies}

This section devotes to the study of the convergence of the following natural
energies associated with the quasilinear problems $\left(
\ref{general_quasilinear_n_intro}\right)  ,$ i.e.%
\begin{equation}
E_{n}\left(  u\right)  =\int_{\Omega}\frac{\Phi_{n}\left(  x,\left\vert \nabla
u\right\vert \right)  }{\varphi_{n}\left(  x,1\right)  }dx,\text{ }u\in
W^{1,\Phi_{n}\left(  \cdot\right)  }\left(  \Omega\right)  . \label{eq:En}%
\end{equation}
Assume that $\varphi_{n}\in C\left(  \bar{\Omega}\times\left[  0,\infty
\right)  \right)  \cap C^{1}\left(  \Omega\times\left(  0,\infty\right)
\right)  $ for $n\geqslant1$, $\varphi_{n}\left(  x,s\right)  >0$ for
$x\in\Omega$, $s>0$, $n\geqslant1$ and\ there are constants $p_{n}^{-},$
$p_{n}^{+}>0$ such that
\begin{equation}
0<p_{n}^{-}-1\leqslant\frac{s\partial_{s}\varphi_{n}\left(  x,s\right)
}{\varphi_{n}\left(  x,s\right)  }\leqslant p_{n}^{+}-1\text{ for }x\in
\Omega\text{, }s>0. \label{assumption_phi_n}%
\end{equation}
Then for each $n$,%
\[
\Phi_{n}\left(  x,s\right)  =\int_{0}^{s}\varphi_{n}\left(  x,t\right)  dt
\]
a generalized $\Phi$-function. Moreover the inequalites in Lemma
\ref{lemma_lieberman} are satisfied uniformly in $x$ and $s$, i.e.%
\begin{equation}
p_{n}^{-}\leqslant\frac{s\varphi_{n}\left(  x,s\right)  }{\Phi_{n}\left(
x,s\right)  }\leqslant p_{n}^{+}\text{ for }x\in\Omega\text{, }s>0\text{,
}n\geqslant1. \label{assumption_Phi_n}%
\end{equation}
We also assume that the following Harnack type inequality holds,
\begin{equation}
\lim_{n\rightarrow\infty}p_{n}^{-}=\infty\text{ and }p_{n}^{+}\leqslant\mu
p_{n}^{-}\text{ for some constant }\mu>1. \label{assumption_c_n}%
\end{equation}

\begin{definition}
\label{def_gamma_convergence}Let $X$ be a metric space. A sequence
$J_{n}:X\longmapsto\mathbb{R\cup}\left\{  \infty\right\}  $ of functionals
$\Gamma$-converges to $J,$ denoted by $\Gamma\left(  X\right)  $%
-$\lim\limits_{n\rightarrow\infty}J_{n}=J$, if for $u\in X$ the following are satisfied:

(i) (asymptotic common lower bound) for any $v_{n}\rightarrow u$ in X,%
\[
J\left(  u\right)  \leqslant\liminf_{n\rightarrow\infty}J_{n}\left(
v_{n}\right)  ;
\]

(ii) (existence of recovery sequence) there is $v_{n}\rightarrow u$ in X,%
\[
J\left(  u\right)  \geqslant\limsup_{n\rightarrow\infty}J_{n}\left(
v_{n}\right)  .
\]

\end{definition}

The simple result below is quite useful in a number of situations.

\begin{lemma}
\label{lm_bounded_W1d1}Assume that $\left(  \ref{assumption_c_n}\right)  $ is
satisfied and $u_{n}\in W^{1,\Phi_{n}\left(  \cdot\right)  }\left(
\Omega\right)  $. Then the following are true.

(i) It holds that%
\[
\left\Vert \nabla u_{n}\right\Vert _{L^{p_{n}^{-}}}^{p_{n}^{-}}\leqslant
\left\vert \Omega\right\vert +E_{n}\left(  u_{n}\right)  p_{n}^{+}%
\leqslant\left\vert \Omega\right\vert +E_{n}\left(  u_{n}\right)  \mu
p_{n}^{-}.
\]

(ii) Assume that there is $\sigma>0$ such that
\[
E_{n}\left(  u_{n}\right)  \leqslant\sigma\text{ for large }n.
\]
Then there is a constant $c\left(  \Omega,\sigma,\mu\right)  >0$ such that,
for all $m\geqslant1$ and $n\geqslant1$ for which $p_{n}^{-}\geqslant m$ and
$E_{n}\left(  u_{n}\right)  \leqslant\sigma,$%
\[
\left\Vert \nabla u_{n}\right\Vert _{L^{m}}\leqslant c\left(  \Omega
,\sigma,\mu\right)  .
\]

(iii) Assume that $p_{n}^{-}\geqslant m,$ $v\in W^{1,\infty}\left(
\Omega\right)  $ and $u_{n}\in v+W_{0}^{1,\Phi_{n}\left(  \cdot\right)
}\left(  \Omega\right)  $. Then there is $c\left(  \Omega,m\right)  >0$ such
that
\[
\left\Vert u_{n}\right\Vert _{W^{1,m}}\leqslant c\left(  \Omega,m\right)
\left(  \left\Vert \nabla u_{n}\right\Vert _{L^{m}}+\left\Vert v\right\Vert
_{W^{1,\infty}}\right)  .
\]
In particular, if $\left\{  \left\Vert \nabla u_{n}\right\Vert _{L^{m}%
}\right\}  $ is bounded, then $\left\{  u_{n}\right\}  $ is bounded in
$W^{1,m}\left(  \Omega\right)  $.
\end{lemma}

\begin{proof}
(i) By virtue of Lemma \ref{lemma_delta2_condition}$,$%
\[
\left\vert \nabla u_{n}\right\vert ^{p_{n}^{-}}1_{\left\{  \left\vert \nabla
u_{n}\right\vert >1\right\}  }\leqslant\frac{\Phi_{n}\left(  x,\left\vert
\nabla u_{n}\right\vert \right)  }{\Phi_{n}\left(  x,1\right)  }1_{\left\{
\left\vert \nabla u_{n}\right\vert >1\right\}  }\text{ for }x\in\Omega.
\]
It then follows that%
\begin{align*}
\left\Vert \nabla u_{n}\right\Vert _{L^{p_{n}^{-}}}^{p_{n}^{-}}  &
=\int_{\Omega}\left\vert \nabla u_{n}\right\vert ^{p_{n}^{-}}1_{\left\{
\left\vert \nabla u_{n}\right\vert \leqslant1\right\}  }dx+\int_{\Omega
}\left\vert \nabla u_{n}\right\vert ^{p_{n}^{-}}1_{\left\{  \left\vert \nabla
u_{n}\right\vert >1\right\}  }dx\\
&  \leqslant\left\vert \Omega\right\vert +\int_{\Omega}\frac{\Phi_{n}\left(
x,\left\vert \nabla u_{n}\right\vert \right)  }{\Phi_{n}\left(  x,1\right)
}1_{\left\{  \left\vert \nabla u_{n}\right\vert >1\right\}  }dx\\
&  \leqslant\left\vert \Omega\right\vert +\int_{\Omega}\frac{\varphi
_{n}\left(  x,1\right)  }{\Phi_{n}\left(  x,1\right)  }\frac{\Phi_{n}\left(
x,\left\vert \nabla u_{n}\right\vert \right)  }{\varphi_{n}\left(  x,1\right)
}dx\\
&  \leqslant\left\vert \Omega\right\vert +E_{n}\left(  u_{n}\right)  p_{n}%
^{+}\\
&  \leqslant\left\vert \Omega\right\vert +E_{n}\left(  u_{n}\right)  \mu
p_{n}^{-},
\end{align*}
where the last inequality is due to $\left(  \ref{assumption_c_n}\right)  $.
(ii) Now the energies $E_{n}\left(  u_{n}\right)  $ are bounded by $\sigma>0$.
Note that the function $x\mapsto\left(  1+x\right)  ^{\frac{1}{x}}$ is
decreasing on $\left(  0,\infty\right)  $. Using H\"{o}lder's inequality and
the result of (i) we get, for $p_{n}^{-}\geqslant m,$%
\begin{align*}
\left\Vert \nabla u_{n}\right\Vert _{L^{m}}  &  \leqslant\left\vert
\Omega\right\vert ^{\frac{1}{m}-\frac{1}{p_{n}^{-}}}\left\Vert \nabla
u_{n}\right\Vert _{L^{p_{n}^{-}}}\\
&  \leqslant\left\vert \Omega\right\vert ^{\frac{1}{m}-\frac{1}{p_{n}^{-}}%
}\left(  \left\vert \Omega\right\vert +\sigma\mu p_{n}^{-}\right)  ^{\frac
{1}{p_{n}^{-}}}\\
&  =\left\vert \Omega\right\vert ^{\frac{1}{m}}\left(  1+\frac{\sigma\mu
p_{n}^{-}}{\left\vert \Omega\right\vert }\right)  ^{\frac{1}{p_{n}^{-}}}\\
&  \leqslant\left(  1+\left\vert \Omega\right\vert \right)  \left(
1+\frac{\sigma\mu}{\left\vert \Omega\right\vert }\right) \\
&  =c_{1}\left(  \Omega,\sigma,\mu\right)  .
\end{align*}
(iii) In view of Lemma \ref{lemma_delta2_condition}, the mapping $s\mapsto
s^{-m}\Phi_{n}\left(  x,s\right)  $ is increasing. Hence by Theorem
\ref{thm_embedding}
\[
W^{1,\Phi_{n}\left(  \cdot\right)  }\left(  \Omega\right)  \subset
W^{1,m}\left(  \Omega\right)  \text{ for all }n\text{ for which }p_{n}%
^{-}\geqslant m.
\]
Using Poincar\'{e}'s inequality we get%
\begin{align*}
\left\Vert u_{n}\right\Vert _{W^{1,m}}  &  \leqslant\left\Vert u_{n}%
-v\right\Vert _{W_{0}^{1,m}}+\left\Vert v\right\Vert _{W^{1,m}}\\
&  \leqslant c_{2}\left(  \Omega,m\right)  \left\Vert \nabla u_{n}-\nabla
v\right\Vert _{L^{m}}+c_{3}\left(  \Omega\right)  \left\Vert v\right\Vert
_{W^{1,\infty}}\\
&  \leqslant c_{4}\left(  \Omega,m\right)  \left(  \left\Vert \nabla
u_{n}\right\Vert _{L^{m}}+\left\Vert v\right\Vert _{W^{1,\infty}}\right)  .
\end{align*}
Therefore $\left\Vert u_{n}\right\Vert _{W^{1,m}}$ is bounded whenever
$\left\{  \left\Vert \nabla u_{n}\right\Vert _{L^{m}}\right\}  $ is bounded.
The proof is completed.
\end{proof}

\begin{theorem}
\label{Thm_gamma_convergence}Extend the functional $E_{n}\left(  u\right)  $
from $W^{1,\Phi_{n}\left(  \cdot\right)  }\left(  \Omega\right)  $ to
$L^{1}\left(  \Omega\right)  $ as below%
\[
E_{n}\left(  u\right)  =\left\{
\begin{array}
[c]{ll}
{\displaystyle\int_{\Omega}}
\dfrac{\Phi_{n}\left(  x,\left\vert \nabla u\right\vert \right)  }{\varphi
_{n}\left(  x,1\right)  }dx, & u\in W^{1,\Phi_{n}\left(  \cdot\right)
}\left(  \Omega\right)  ;\\
\infty, & \text{otherwise},
\end{array}
\right.
\]
and let%
\[
E_{\infty}\left(  u\right)  =\left\{
\begin{array}
[c]{ll}%
0, & u\in W^{1,\infty}\left(  \Omega\right)  \text{ and }\left\Vert \nabla
u\right\Vert _{L^{\infty}}\leqslant1;\\
\infty, & \text{otherwise}.
\end{array}
\right.
\]
Then $\Gamma\left(  L^{1}\left(  \Omega\right)  \right)  $-$\lim
\limits_{n\rightarrow\infty}E_{n}=E_{\infty}.$
\end{theorem}

\begin{proof}
1. Existence of recovery sequence. Fix $u\in L^{1}\left(  \Omega\right)  $. If
$E_{\infty}\left(  u\right)  =$ $\infty$, then there is nothing to prove. If
$E_{\infty}\left(  u\right)  <\infty$, then $E_{\infty}\left(  u\right)  =0$,
$u\in W^{1,\infty}\left(  \Omega\right)  $ and $\left\Vert \nabla u\right\Vert
_{L^{\infty}}\leqslant1.$ It follows that%
\[
E_{\infty}\left(  u\right)  =0\leqslant E_{n}\left(  u\right)  =\int_{\Omega
}\frac{\Phi_{n}\left(  x,\left\vert \nabla u\right\vert \right)  }{\varphi
_{n}\left(  x,1\right)  }dx\leqslant\int_{\Omega}\frac{\Phi_{n}\left(
x,1\right)  }{\varphi_{n}\left(  x,1\right)  }dx\leqslant\frac{\left\vert
\Omega\right\vert }{p_{n}^{-}}\rightarrow0.
\]
Hence $v_{n}\equiv u$ is a recovery sequence.

2. $E_{n}$ has asymptotic common lower bound. Consider an arbitrary sequence
$v_{n}\rightarrow u$ in $L^{1}\left(  \Omega\right)  $. If there are only a
finite number of $v_{n}$ belonging to $W^{1,\Phi_{n}\left(  \cdot\right)
}\left(  \Omega\right)  $, then $E_{n}\left(  v_{n}\right)  =\infty$ for all
$n$ large. Hence $\liminf\limits_{n\rightarrow\infty}E_{n}\left(
v_{n}\right)  =\infty$, the desired inequality clearly holds. Now we assume up
to a subsequence that $v_{n}\in W^{1,\Phi_{n}\left(  \cdot\right)  }\left(
\Omega\right)  $ for $n\geqslant1$ and that
\[
E_{n}\left(  v_{n}\right)  \text{ is a bounded sequence.}%
\]
In view of Lemma \ref{lm_bounded_W1d1}, we can further assume that $\nabla
v_{n}$ converges weakly to some $\varpi$ $\in L^{m}\left(  \Omega\right)  $,
$m\geqslant1$. Since $v_{n}\rightarrow u$ in $L^{1}\left(  \Omega\right)  ,$
we infer that $\varpi=\nabla u$ and $u\in W^{1,m}\left(  \Omega\right)  $. We
need to show%
\[
E_{\infty}\left(  u\right)  \leqslant\liminf_{n\rightarrow\infty}E_{n}\left(
v_{n}\right)  .
\]
This amounts to showing%
\[
\left\vert \nabla u\left(  x\right)  \right\vert \leqslant1\text{ for
}a.e.\text{ }x\in\Omega.
\]
For $\epsilon>0,$ let%
\[
\Omega_{\epsilon}=\left\{  x\in\Omega:\left\vert \nabla u\left(  x\right)
\right\vert >1+\epsilon\right\}  .
\]
By H\"{o}lder's inequality,%
\begin{align*}
\int_{\Omega_{\epsilon}}\left\vert \nabla u\left(  y\right)  \right\vert dy
&  \leqslant\left\vert \Omega_{\epsilon}\right\vert ^{1-\frac{1}{m}}\left\Vert
\nabla u\right\Vert _{L^{m}}\\
&  \leqslant\left\vert \Omega_{\epsilon}\right\vert ^{1-\frac{1}{m}}%
\liminf_{n\rightarrow\infty}\left\Vert \nabla v_{n}\right\Vert _{L^{m}}\\
&  \leqslant\left\vert \Omega_{\epsilon}\right\vert ^{1-\frac{1}{m}}\left\vert
\Omega\right\vert ^{\frac{1}{m}-\frac{1}{p_{n}^{-}}}\liminf_{n\rightarrow
\infty}\left\Vert \nabla v_{n}\right\Vert _{L^{p_{n}^{-}}}\\
&  \leqslant\left\vert \Omega_{\epsilon}\right\vert ^{1-\frac{1}{m}}\left\vert
\Omega\right\vert ^{\frac{1}{m}-\frac{1}{p_{n}^{-}}}\left(  \left\vert
\Omega\right\vert +E_{n}\left(  v_{n}\right)  \mu p_{n}^{-}\right)  ^{\frac
{1}{p_{n}^{-}}}.
\end{align*}
Note that the energies $E_{n}\left(  v_{n}\right)  $ are bounded. Taking the
limit $p_{n}^{-}\rightarrow\infty$ and then $m\rightarrow\infty,$ we get
\[
\left(  1+\epsilon\right)  \left\vert \Omega_{\epsilon}\right\vert
\leqslant\int_{\Omega_{\epsilon}}\left\vert \nabla u\left(  y\right)
\right\vert dy\leqslant\left\vert \Omega_{\epsilon}\right\vert .
\]
So $\left\vert \Omega_{\epsilon}\right\vert =0$. Since $\epsilon>0$ is
arbitrary, this shows that $\left\vert \left\{  \left\vert \nabla u\right\vert
>1\right\}  \right\vert =0.$ The proof is completed.
\end{proof}

\section{\label{sec_lim_g1}Limit of the Solutions with $\left\Vert
g\right\Vert _{Lip\left(  \partial\Omega\right)  }\leqslant1$}

Throughout this section, $g$ is assumed to satisfy $\left(
\ref{assumption_g_lip}\right)  ,$ i.e.%
\[
\left\Vert g\right\Vert _{Lip\left(  \partial\Omega\right)  }\leqslant1.
\]
Then the AMLE $g\in W^{1,\infty}\left(  \Omega\right)  $ (ref. Remark
\ref{Rmk_g}). Consider natural energy%
\begin{equation}
E_{n}\left(  u\right)  =\int_{\Omega}\frac{\Phi_{n}\left(  x,\left\vert \nabla
u\right\vert \right)  }{\varphi_{n}\left(  x,1\right)  }dx,\text{ }u\in
g+W_{0}^{1,\Phi_{n}\left(  \cdot\right)  }\left(  \Omega\right)
\label{eq:En_g}%
\end{equation}
where $\varphi_{n}$, $\Phi_{n}$ are given at the beginning of Section
\ref{sec_gamma}$,$ and $\Phi_{n}$ satisfies the equivalent normalization
condition $\left(  \ref{assumption_PHI_1}\right)  $ (see Lemma
\ref{lm_normalization}) uniformly in $n,$ i.e.,%
\begin{equation}
c^{-}\leqslant\Phi_{n}\left(  x,1\right)  \leqslant c^{+}\text{ for }%
x\in\Omega\text{ and }n\geqslant1. \label{assumption_PHI_n_1}%
\end{equation}
The critical points of $E_{n}$ in $g+W_{0}^{1,\Phi_{n}\left(  \cdot\right)
}\left(  \Omega\right)  $ are weak solutions of the equaiton%
\begin{equation}
\left\{
\begin{array}
[c]{ll}%
-\operatorname{div}\left(  \dfrac{a_{n}\left(  x,\left\vert \nabla u\left(
x\right)  \right\vert \right)  }{\left\vert \nabla u\left(  x\right)
\right\vert }\nabla u\left(  x\right)  \right)  =0, & x\in\Omega\\
u\left(  x\right)  =g\left(  x\right)  , & x\in\partial\Omega
\end{array}
\right.  \label{quasilinear_PDE_n}%
\end{equation}
where%
\begin{equation}
a_{n}\left(  x,s\right)  =\frac{\varphi_{n}\left(  x,s\right)  }{\varphi
_{n}\left(  x,1\right)  }. \label{eq:def_a_xs_n}%
\end{equation}

We consider the convergence as $n\rightarrow\infty$\ of the solutions of the
equation $\left(  \ref{quasilinear_PDE_n}\right)  $. The results in the last
seciton roughly suggest that the solutions of equation $\left(
\ref{quasilinear_PDE_n}\right)  $ converge as $n\rightarrow\infty$ to a
solution of the equation associated with $E_{\infty}$. This section is devoted
to making this intuition rigorous.

\begin{theorem}
\label{Thm_existence}Assume that $\Phi_{n}$ satisfies the normalization
condition $\left(  \ref{assumption_PHI_n_1}\right)  ,$ the local continuity
condition $\left(  \ref{PHI_condition_A11}\right)  $ and $p_{n}^{-}\geqslant
d+1$. Then the equation $\left(  \ref{quasilinear_PDE_n}\right)  $ has a
unique weak solution $u$ which minimizes $E_{n}$ in $g+W_{0}^{1,\Phi
_{n}\left(  \cdot\right)  }\left(  \Omega\right)  .$ In addition $u\in
C\left(  \bar{\Omega}\right)  \cap W^{1,\Phi_{n}\left(  \cdot\right)  }\left(
\Omega\right)  .$
\end{theorem}

\begin{proof}
First observe that $E_{n}$ is coercive by Theorem
\ref{thm_poincare_inequality} (see also Remark \ref{rmk_phi_a} and $\left(
\ref{assumption_phi_1_lbd}\right)  $). Let $v_{k}\in g+W_{0}^{1,\Phi
_{n}\left(  \cdot\right)  }\left(  \Omega\right)  $ be a minimizing sequence
for $E_{n}$. Then it is bounded in $g+W_{0}^{1,\Phi_{n}\left(  \cdot\right)
}\left(  \Omega\right)  $. Since $W^{1,\Phi_{n}\left(  \cdot\right)  }\left(
\Omega\right)  $ is reflexive by Theorem \ref{thm_reflexive}, we can assume
without loss of generality that $v_{k}$ converges as $k\rightarrow\infty$
weakly to some $u\in g+W_{0}^{1,\Phi_{n}\left(  \cdot\right)  }\left(
\Omega\right)  $. By the form of $E_{n}$, it is clear that $E_{n}$ can be
regarded as an equivalent norm on
\[
\left\{  \nabla v\in L^{\Phi_{n}\left(  \cdot\right)  }\left(  \Omega\right)
:v\in W^{1,\Phi_{n}\left(  \cdot\right)  }\left(  \Omega\right)  \right\}  .
\]
It follows that, since $\nabla v_{k}$ converges weakly to $\nabla u$ in
$L^{\Phi_{n}\left(  \cdot\right)  }$,%
\[
E_{n}\left(  u\right)  \leqslant\liminf_{k\rightarrow\infty}E_{n}\left(
v_{k}\right)  .
\]
This shows that $u$ is a minimizer of $E_{n}$. Since the mapping%
\[
\xi\mapsto\frac{\Phi_{n}\left(  x,\left\vert \xi\right\vert \right)  }%
{\varphi_{n}\left(  x,1\right)  }%
\]
is strictly convex, the uniqueness of the minimizer follows from the standard
argument, e.g. \cite{dacorogna2007direct}. Finally noting that $s\mapsto
s^{-\left(  d+1\right)  }\Phi_{n}\left(  x,s\right)  $ is increasing by Lemma
\ref{lemma_delta2_condition}, we get by Theorem \ref{thm_embedding} that%
\[
W^{1,\Phi_{n}\left(  \cdot\right)  }\left(  \Omega\right)  \subset
W^{1,d+1}\left(  \Omega\right)  \text{.}%
\]
Also note that $g+W_{0}^{1,d+1}\left(  \Omega\right)  \subset C\left(
\bar{\Omega}\right)  $. Under our assumption, the energy $E_{n}\left(
u\right)  $ is strictly convex in $u$, thus for each $n$ the energy has a
unique critical point which corresponds to the unique weak solution of the
equaiton $\left(  \ref{quasilinear_PDE_n}\right)  $. \ Hence $u$, which
satisfies the Dirichlet condition in the classical sense, is the unique weak
solution for equation $\left(  \ref{quasilinear_PDE_n}\right)  .$
\end{proof}

Given a \textit{continuous }function $F:\bar{\Omega}\times\mathbb{R}^{d}%
\times\mathbb{R}^{2d}\mapsto\mathbb{R}$, a \textit{continuous }%
function\textit{ }$g:\partial\Omega\mapsto\mathbb{R}$, we consider degenerate
elliptic Dirichlet problem of the form%
\begin{equation}
\left\{
\begin{array}
[c]{ll}%
F\left(  x,\nabla u,D^{2}u\right)  =0, & x\in\Omega\\
u=g, & x\in\partial\Omega
\end{array}
\right.  .\label{dirichlet_PDE}%
\end{equation}
The solution is understood in the viscosity sense \cite{crandall1992user}$.$

\begin{remark}
\label{rmk_operator_continuity}For our equation, we have%
\[
F\left(  x,\nabla u,D^{2}u\right)  =-\operatorname{div}\left(  \frac{a\left(
x,\left\vert \nabla u\left(  x\right)  \right\vert \right)  }{\left\vert
\nabla u\left(  x\right)  \right\vert }\nabla u\left(  x\right)  \right)  .
\]
As is required by $\left(  \ref{dirichlet_PDE}\right)  $, we want $\left(
x,p,X\right)  \in\left(  \mathbb{R}^{d},\mathbb{R}^{d},\mathbb{R}^{d\times
d}\right)  \mapsto F\left(  x,p,X\right)  $ to be continuous. Recall that
$a\left(  x,s\right)  =\left.  \varphi\left(  x,s\right)  \right/
\varphi\left(  x,1\right)  $ (see $\left(  \ref{eq:def_a_xs}\right)  $) with
$\varphi$ satisfying $\left(  \ref{assumption_phi_x}\right)  .$
Straightforward calculation yields that%
\begin{align*}
F\left(  x,p,X\right)   &  =-\frac{a\left(  x,p\right)  }{p}\text{tr}\left(
X\right)  -\partial_{s}a\left(  x,\left\vert p\right\vert \right)
\text{tr}\left(  \frac{p}{\left\vert p\right\vert }\otimes\frac{p}{\left\vert
p\right\vert }X\right)  \\
&  +\frac{a\left(  x,\left\vert p\right\vert \right)  }{\left\vert
p\right\vert }\text{tr}\left(  \frac{p}{\left\vert p\right\vert }\otimes
\frac{p}{\left\vert p\right\vert }X\right)  -\partial_{x}a\left(  x,\left\vert
p\right\vert \right)  \cdot p.
\end{align*}
To fulfill the continuity requirement, it is enough to ensure the continuity
at $p=0.$ For this we assume that%
\begin{equation}
p^{-}>3\text{ and there exists }a_{0}\in C\left(  \Omega\right)  \text{ so
that }\lim_{s\rightarrow0+}\partial_{x}a\left(  x,s\right)  =a_{0}\left(
x\right)  \text{ in }\Omega,\label{eq:assume_F_cont}%
\end{equation}
where $p^{-}$ is the index as in $\left(  \ref{assumption_phi_x}\right)  $.
Indeed, under the assumption $\left(  \ref{assumption_phi_x}\right)  $,%
\[
\left(  p^{-}-1\right)  \frac{a\left(  x,s\right)  }{s}\leqslant\partial
_{s}a\left(  x,s\right)  \leqslant\left(  p^{+}-1\right)  \frac{a\left(
x,s\right)  }{s}\text{ for }x\in\Omega,\text{ }s>0.
\]
Hence we readily have that with $\left(  \ref{eq:assume_F_cont}\right)  ,$%
\begin{align*}
\left.  \frac{a\left(  x,p\right)  }{p}\right\vert _{p=0} &  =0,\text{
}\left.  \partial_{s}a\left(  x,\left\vert p\right\vert \right)  \frac
{p}{\left\vert p\right\vert }\otimes\frac{p}{\left\vert p\right\vert
}\right\vert _{p=0}=0,\\
\left.  \frac{a\left(  x,\left\vert p\right\vert \right)  }{\left\vert
p\right\vert }\text{tr}\left(  \frac{p}{\left\vert p\right\vert }\otimes
\frac{p}{\left\vert p\right\vert }X\right)  \right\vert _{p=0} &  =0,\text{
}\left.  \partial_{x}a\left(  x,\left\vert p\right\vert \right)  \cdot
p\right\vert _{p=0}=0.
\end{align*}
By this understanding $F\left(  x,p,X\right)  $ is defined everywhere as a
continuous function. In particular,%
\[
F\left(  x,0,X\right)  =-a_{0}\cdot0=0.
\]

\end{remark}

\begin{theorem}
\label{thm_weaksol_viscositysol}Assume that $\Phi_{n}$ satisfies the
normalization condition $\left(  \ref{assumption_PHI_n_1}\right)  ,$ the local
continuity condition $\left(  \ref{PHI_condition_A11}\right)  ,$ $a_{n}$
satisfies $\left(  \ref{eq:assume_F_cont}\right)  $ and $p_{n}^{-}\geqslant
d+1$. If $u\in W^{1,\Phi_{n}\left(  \cdot\right)  }\left(  \Omega\right)  $ is
a continuous weak solution for equation $\left(  \ref{quasilinear_PDE_n}%
\right)  $, then it is a viscosity solution for equation $\left(
\ref{quasilinear_PDE_n}\right)  $.
\end{theorem}

\begin{proof}
For notational simplicity, we will drop the subscripts $n$ in $a_{n},$
$\varphi_{n},$ $\Phi_{n},$ $p_{n}^{-},$ $p_{n}^{+}$ in the proof. In view of
Theorem \ref{Thm_existence}, $u$ takes the boundary values continuously. So,
to show that $u$ is a viscosity subsolution, it suffices to check viscosity
inequality for interior points. Assume that $u-\phi$ has a strict local
maximum at $x_{0}\in\Omega$ for some $\phi\in C^{2}\left(  \bar{\Omega
}\right)  $, $u\left(  x_{0}\right)  =\phi\left(  x_{0}\right)  $. We intend
to show that%
\[
F=\left(  x,\nabla\phi\left(  x_{0}\right)  ,D^{2}\phi\left(  x_{0}\right)
\right)  =-\operatorname{div}\left(  \frac{a\left(  x_{0},\left\vert
\nabla\phi\left(  x_{0}\right)  \right\vert \right)  }{\left\vert \nabla
\phi\left(  x_{0}\right)  \right\vert }\nabla\phi\left(  x_{0}\right)
\right)  \leqslant0.
\]
If $\nabla\phi\left(  x_{0}\right)  =0,$ then the poof is automatic since
$F\left(  x,\nabla\phi\left(  x_{0}\right)  ,D^{2}\phi\left(  x_{0}\right)
\right)  =0$ in view of Remark \ref{rmk_operator_continuity}. If $\nabla
\phi\left(  x_{0}\right)  \neq0,$ we prove by contradition. Assume that%
\[
-\operatorname{div}\left(  \frac{a\left(  x_{0},\left\vert \nabla\phi\left(
x_{0}\right)  \right\vert \right)  }{\left\vert \nabla\phi\left(
x_{0}\right)  \right\vert }\nabla\phi\left(  x_{0}\right)  \right)  >0.
\]
then there is a small open ball $B_{r}\left(  x_{0}\right)  \subset\Omega$
with radius $r>0$ centering at $x_{0}$ such that $F$ has the continuity as
discussed in Remark \ref{rmk_operator_continuity} for $x$ in the closure of
$B_{r}\left(  x_{0}\right)  $ and that
\[
-\operatorname{div}\left(  \frac{a\left(  x,\left\vert \nabla\phi\left(
x\right)  \right\vert \right)  }{\left\vert \nabla\phi\left(  x\right)
\right\vert }\nabla\phi\left(  x\right)  \right)  >0\text{ for }x\in
B_{r}\left(  x_{0}\right)  .
\]
Let%
\[
\zeta\left(  x\right)  =\phi\left(  x\right)  -\frac{1}{2}\inf_{\left\vert
x-x_{0}\right\vert =r}\left(  \phi-u\right)  .
\]
Clearly $\zeta\left(  x_{0}\right)  <u\left(  x_{0}\right)  ,$ $\zeta-u>0$ on
$\partial B_{r}\left(  x_{0}\right)  $ and
\[
-\operatorname{div}\left(  \frac{a\left(  x,\left\vert \nabla\zeta\left(
x\right)  \right\vert \right)  }{\left\vert \nabla\zeta\left(  x\right)
\right\vert }\nabla\zeta\left(  x\right)  \right)  >0\text{ for }x\in
B_{r}\left(  x_{0}\right)  .
\]
Multiply this inequality by $\max\left\{  u-\zeta,0\right\}  $ and integrate,
then we get%
\begin{equation}
\int_{B_{r}\left(  x_{0}\right)  \cap\left\{  \zeta<u\right\}  }\frac{a\left(
x,\left\vert \nabla\zeta\right\vert \right)  }{\left\vert \nabla
\zeta\right\vert }\left\langle \nabla\zeta,\nabla\left(  u-\zeta\right)
\right\rangle dx>0. \label{thm_weaksol_viscositysol_inq1}%
\end{equation}
Since $u$ is a weak solution, for $\xi\in C_{0}^{\infty}\left(  \Omega\right)
,$%
\[
\int_{\Omega}\frac{a\left(  x,\left\vert \nabla u\right\vert \right)
}{\left\vert \nabla u\right\vert }\left\langle \nabla u,\nabla\xi\right\rangle
dx=0.
\]
At this point if we can use the continuous function $g_{0}=\max\left\{
u-\zeta,0\right\}  $, which is extended by zero outside $B_{r}\left(
x_{0}\right)  ,$ as a test function in place of $\xi$, then the proof would be
finished. To see this, we take for granted for the moment that this is valid
and replace $\xi$ with $\max\left\{  u-\zeta,0\right\}  $ in the above
equality, i.e.,
\begin{equation}
\int_{B_{r}\left(  x_{0}\right)  \cap\left\{  \zeta<u\right\}  }\frac{a\left(
x,\left\vert \nabla u\right\vert \right)  }{\left\vert \nabla u\right\vert
}\left\langle \nabla u,\nabla\left(  u-\zeta\right)  \right\rangle dx=0,
\label{thm_weaksol_viscositysol_inq2}%
\end{equation}
Then subtracting $\left(  \ref{thm_weaksol_viscositysol_inq1}\right)  $ from
$\left(  \ref{thm_weaksol_viscositysol_inq2}\right)  $ yields%
\begin{equation}
0>\int_{B_{r}\left(  x_{0}\right)  \cap\left\{  \zeta<u\right\}  }\left\langle
\frac{a\left(  x,\left\vert \nabla u\right\vert \right)  }{\left\vert \nabla
u\right\vert }\nabla u-\frac{a\left(  x,\left\vert \nabla\zeta\right\vert
\right)  }{\left\vert \nabla\zeta\right\vert }\nabla\zeta,\nabla\left(
u-\zeta\right)  \right\rangle dx. \label{thm_weaksol_viscositysol_inq3}%
\end{equation}
However the right hand side of $\left(  \ref{thm_weaksol_viscositysol_inq3}%
\right)  $ is nonnegative due to Lemma \ref{inequality_main_x}, which is a
contradiction. Therefore to complete the proof, it remains to justify using
$\max\left\{  u-\zeta,0\right\}  $ as a test function. Recalling the
assumption $\left(  \ref{assumption_phi_n}\right)  $ and the lower bound
$\left(  \ref{assumption_phi_1_lbd}\right)  $ on $\varphi\left(  x,1\right)
$, we have for $\xi\in C^{\infty}\left(  \Omega\right)  ,$%
\begin{align*}
\int_{\Omega}a\left(  x,\left\vert \nabla u\right\vert \right)  \frac
{\left\vert \nabla\xi\right\vert }{\lambda}dx  &  =\int_{\Omega}\frac
{\varphi\left(  x,\left\vert \nabla u\right\vert \right)  }{\varphi\left(
x,1\right)  }\frac{\left\vert \nabla\xi\right\vert }{\lambda}dx\\
&  \leqslant\frac{1}{c^{-}p^{-}}\int_{\Omega}\varphi\left(  x,\left\vert
\nabla u\right\vert \right)  \frac{\left\vert \nabla\xi\right\vert }{\lambda
}dx\\
&  \leqslant\frac{1}{c^{-}p^{-}}\int_{\Omega}\left[  \Phi^{\ast}\left(
x,\varphi\left(  x,\left\vert \nabla u\right\vert \right)  \right)
+\Phi\left(  x,\frac{\left\vert \nabla\xi\right\vert }{\lambda}\right)
\right]  dx,
\end{align*}
where $\lambda>\left\Vert \nabla\xi\right\Vert _{\Phi\left(  \cdot\right)  }$
so that then $\left\Vert \nabla\xi/\lambda\right\Vert _{\Phi\left(
\cdot\right)  }<1.$ By the definition of the norm $\left\Vert \cdot\right\Vert
_{\Phi\left(  \cdot\right)  },$ we have%
\[
\int_{\Omega}\Phi\left(  x,\frac{\left\vert \nabla\xi\right\vert }{\lambda
}\right)  dx\leqslant1.
\]
Hence%
\begin{align*}
\int_{\Omega}a\left(  x,\left\vert \nabla u\right\vert \right)  \frac
{\left\vert \nabla\xi\right\vert }{\lambda}dx  &  \leqslant\frac{1}{c^{-}%
p^{-}}\int_{\Omega}\left[  \Phi^{\ast}\left(  x,\varphi\left(  x,\left\vert
\nabla u\right\vert \right)  \right)  +1\right]  dx\\
&  \leqslant\frac{1}{c^{-}p^{-}}\int_{\Omega}\left[  \varphi\left(
x,\left\vert \nabla u\right\vert \right)  \left\vert \nabla u\right\vert
+1\right]  dx\\
&  \leqslant\frac{1}{c^{-}p^{-}}\int_{\Omega}\left[  p^{+}\Phi\left(
x,\left\vert \nabla u\right\vert \right)  +1\right]  dx.
\end{align*}
The RHS is finite by virtue of Lemma \ref{lm_modulo_finite}. It follows that%
\[
\int_{\Omega}a\left(  x,\left\vert \nabla u\right\vert \right)  \left\vert
\nabla\xi\right\vert dx\leqslant\lambda\frac{1}{c^{-}p^{-}}\int_{\Omega
}\left[  p^{+}\Phi\left(  x,\left\vert \nabla u\right\vert \right)  +1\right]
dx
\]
Let $\lambda\downarrow\left\Vert \nabla\xi\right\Vert _{\Phi\left(
\cdot\right)  },$ we get%
\[
\int_{\Omega}a\left(  x,\left\vert \nabla u\right\vert \right)  \left\vert
\nabla\xi\right\vert dx\leqslant\left\Vert \nabla\xi\right\Vert _{\Phi\left(
\cdot\right)  }\frac{1}{c^{-}p^{-}}\int_{\Omega}\left[  p^{+}\Phi\left(
x,\left\vert \nabla u\right\vert \right)  +1\right]  dx.
\]
This implies that%
\[
v\in W^{1,\Phi\left(  \cdot\right)  }\left(  \Omega\right)  \mapsto
\int_{\Omega}a\left(  x,\left\vert \nabla u\right\vert \right)  vdx
\]
is a continuous linear functional. Note that $u,$ $\zeta\in W^{1,\Phi\left(
\cdot\right)  }\left(  \Omega\right)  $ and it is readily seen that the test
function $g_{0}\in W^{1,\Phi\left(  \cdot\right)  }\left(  \Omega\right)  $ by
the chain rule of weak derivative \cite[Lemma 7.6]{gilbarg2001elliptic}. Thus
we can approximate $g_{0}$ by a sequence of smooth functions $\xi_{k}\in
C_{0}^{\infty}\left(  B_{r}\left(  x_{0}\right)  \right)  $ and employ the
continuity of the above functional to justify the equation $\left(
\ref{thm_weaksol_viscositysol_inq2}\right)  .$ This completes the proof that
$u$ is a viscosity subsolution. That $u$ is a viscosity supersolution can be
proved analogously.
\end{proof}

To show that the right hand side of the inequality $\left(
\ref{thm_weaksol_viscositysol_inq3}\right)  $ is positive, we can also use
Cauchy-Schwarz inequality to obtain
\begin{align*}
&  \left\langle \frac{a\left(  x,\left\vert \nabla u\right\vert \right)
}{\left\vert \nabla u\right\vert }\nabla u-\frac{a\left(  x,\left\vert
\nabla\zeta\right\vert \right)  }{\left\vert \nabla\zeta\right\vert }%
\nabla\zeta,\nabla\left(  u-\zeta\right)  \right\rangle \\
&  =\frac{a\left(  x,\left\vert \nabla u\right\vert \right)  }{\left\vert
\nabla u\right\vert }\left\vert \nabla u\right\vert ^{2}-\left\langle
\frac{a\left(  x,\left\vert \nabla u\right\vert \right)  }{\left\vert \nabla
u\right\vert }\nabla u,\nabla\zeta\right\rangle -\left\langle \frac{a\left(
x,\left\vert \nabla\zeta\right\vert \right)  }{\left\vert \nabla
\zeta\right\vert }\nabla\zeta,\nabla u\right\rangle +\frac{a\left(
x,\left\vert \nabla\zeta\right\vert \right)  }{\left\vert \nabla
\zeta\right\vert }\left\vert \nabla\zeta\right\vert ^{2}\\
&  \geqslant\frac{a\left(  x,\left\vert \nabla u\right\vert \right)
}{\left\vert \nabla u\right\vert }\left\vert \nabla u\right\vert ^{2}%
-\frac{a\left(  x,\left\vert \nabla u\right\vert \right)  }{\left\vert \nabla
u\right\vert }\left\vert \nabla u\right\vert \left\vert \nabla\zeta\right\vert
-\frac{a\left(  x,\left\vert \nabla\zeta\right\vert \right)  }{\left\vert
\nabla\zeta\right\vert }\left\vert \nabla\zeta\right\vert \left\vert \nabla
u\right\vert +\frac{a\left(  x,\left\vert \nabla\zeta\right\vert \right)
}{\left\vert \nabla\zeta\right\vert }\left\vert \nabla\zeta\right\vert ^{2}\\
&  =\left\langle \frac{a\left(  x,\left\vert \nabla u\right\vert \right)
}{\left\vert \nabla u\right\vert }\left\vert \nabla u\right\vert
-\frac{a\left(  x,\left\vert \nabla\zeta\right\vert \right)  }{\left\vert
\nabla\zeta\right\vert }\left\vert \nabla\zeta\right\vert ,\left\vert \nabla
u\right\vert -\left\vert \nabla\zeta\right\vert \right\rangle \\
&  =\left\langle a\left(  x,\left\vert \nabla u\right\vert \right)  -a\left(
x,\left\vert \nabla\zeta\right\vert \right)  ,\left\vert \nabla u\right\vert
-\left\vert \nabla\zeta\right\vert \right\rangle
\end{align*}
which is nonnegative in view of the monotonicity of $s\mapsto a\left(
x,s\right)  $.

Now we show that the sequence of minimizers of $E_{n}$ is bounded in
$W^{1,d+1}\left(  \Omega\right)  $, which eventually enables us to derive the
limiting equation as $n\rightarrow\infty.$

\begin{theorem}
\label{thm_bounded_W1d}Assume that $\Phi_{n}$ satisfies the normalization
condition $\left(  \ref{assumption_PHI_n_1}\right)  ,$ the local continuity
condition $\left(  \ref{PHI_condition_A11}\right)  ,$ $a_{n}$ satisfies
$\left(  \ref{eq:assume_F_cont}\right)  $ and $p_{n}^{-}\geqslant m$. Assume
also that the Harnack-type condition $\left(  \ref{assumption_c_n}\right)  $
is satisfied and $u_{n}$ is the unique weak solution for equation $\left(
\ref{quasilinear_PDE_n}\right)  $ obtained via Theorem \ref{Thm_existence}.
Then, for all $m>d$, $\left\{  u_{n}\right\}  $ is bounded in $W^{1,m}\left(
\Omega\right)  $. Moreover there exists a subsequence of $u_{n}$ converging
uniformly to some $u\in C\left(  \bar{\Omega}\right)  \cap W^{1,\infty}\left(
\Omega\right)  $ and there is a constant $c\left(  \Omega,g,\mu\right)  >0$
such that%
\[
\left\Vert \nabla u\right\Vert _{L^{\infty}}\leqslant c\left(  \Omega
,g,\mu\right)  \text{.}%
\]

\end{theorem}

\begin{proof}
Since $u_{n}$ minimizes $E_{n}$ over $g+W_{0}^{1,\Phi_{n}}\left(
\Omega\right)  ,$ we have
\[
E_{n}\left(  u_{n}\right)  \leqslant E_{n}\left(  v\right)  \text{ for all
}v\in g+W_{0}^{1,\Phi_{n}}\left(  \Omega\right)  \text{.}%
\]
In view of Remark \ref{Rmk_g}, $\left\Vert \nabla g\right\Vert _{L^{\infty
}\left(  \Omega\right)  }\leqslant\left\Vert g\right\Vert _{Lip\left(
\partial\Omega\right)  }\leqslant1$. Then%
\[
E_{n}\left(  u_{n}\right)  \leqslant E_{n}\left(  g\right)  \leqslant
\int_{\Omega}\frac{\Phi_{n}\left(  x,1\right)  }{\varphi_{n}\left(
x,1\right)  }dx\leqslant\int_{\Omega}\frac{1}{p_{n}^{-}}dx=\frac{\left\vert
\Omega\right\vert }{p_{n}^{-}}<\left\vert \Omega\right\vert .
\]
Hence from Lemma \ref{lm_bounded_W1d1} (ii)(iii)$,$ we see that $\left\{
u_{n}\right\}  $ is bounded in $W^{1,m}\left(  \Omega\right)  $. Now let
$m>d$. Since the embedding $W^{1,m}\left(  \Omega\right)  \subset C\left(
\bar{\Omega}\right)  $ is compact, we may assume up to a subsequence that
$u_{n}$ converges weakly in $W^{1,m}\left(  \Omega\right)  $ and uniformly to
some $u\in C\left(  \bar{\Omega}\right)  \cap W^{1,m}\left(  \Omega\right)  $.
In addition, $u_{n}$ assumes the boundary value $g$ continuously and $u=g$ on
$\partial\Omega$. Now invoking Lemma \ref{lm_bounded_W1d1} (ii) again, we have
a constant $c\left(  \Omega,g,\mu\right)  >0$ such that, for all $n$ for which
$p_{n}^{-}\geqslant m,$
\[
\left\Vert \nabla u_{n}\right\Vert _{L^{m}}\leqslant c\left(  \Omega
,g,\mu\right)  \text{.}%
\]
So, for all $m>d$,%
\[
\left\Vert \nabla u\right\Vert _{L^{m}}\leqslant\liminf_{n\rightarrow\infty
}\left\Vert \nabla u_{n}\right\Vert _{L^{m}}\leqslant c\left(  \Omega
,g,\mu\right)  \text{,}%
\]
This implies $\nabla u\in L^{\infty}$ and%
\[
\left\Vert \nabla u\right\Vert _{L^{\infty}}\leqslant c\left(  \Omega
,g,\mu\right)  \text{.}%
\]
Since $u\in C\left(  \bar{\Omega}\right)  $, hence $u\in C\left(  \bar{\Omega
}\right)  \cap W^{1,\infty}\left(  \Omega\right)  $.
\end{proof}

In order to study the limit of the quasilinear problems $\left(
\ref{quasilinear_PDE_n}\right)  $, some additional assumptions are needed. We
assume that there is $\Lambda\left(  x,s\right)  :\bar{\Omega}\times\left(
0,\infty\right)  \mapsto\mathbb{R}^{d}$ such that, for each $c>0,$%
\begin{equation}
\lim_{n\rightarrow\infty}\frac{\partial_{x}a_{n}\left(  x,s\right)  }%
{\partial_{s}a_{n}\left(  x,s\right)  }=\Lambda\left(  x,s\right)  \text{
uniformly for all }x\in\bar{\Omega}\text{ and }s\geqslant c.
\label{assumption_limit_1}%
\end{equation}
Together with the asssumptions on $a_{n}\left(  x,s\right)  $, $\left(
\ref{assumption_limit_1}\right)  $ implies that
\begin{equation}
\Lambda\left(  x,s\right)  \in C\left(  \bar{\Omega}\times\left(
0,\infty\right)  ,\mathbb{R}^{d}\right)  .
\label{assumption_Lambda_continuous}%
\end{equation}
Note that the continuity at $s=0$ is not mandatory, $\Lambda\left(
x,s\right)  $ could even be unbounded there. But we require some control near
$s=0$, i.e.
\begin{equation}
\lim_{\substack{n\rightarrow\infty\\s\rightarrow0+}}s^{2}\frac{\partial
_{x}a_{n}\left(  x,s\right)  }{\partial_{s}a_{n}\left(  x,s\right)  }=0\text{
uniformly for all }x\in\bar{\Omega}. \label{assumption_limit_0}%
\end{equation}
In combination with $\left(  \ref{assumption_limit_1}\right)  ,$ this implies
that%
\begin{equation}
\lim_{s\rightarrow0+}s^{2}\Lambda\left(  x,s\right)  =0\text{ uniformly for
all }x\in\bar{\Omega}. \label{Nassumption_limit_2}%
\end{equation}
We also assume, for each $n$,%
\begin{equation}
\lim_{s\rightarrow0+}s^{2}\frac{\partial_{x}a_{n}\left(  x,s\right)
}{\partial_{s}a_{n}\left(  x,s\right)  }\text{ exists uniformly for all }%
x\in\bar{\Omega}. \label{assumption_limit_3}%
\end{equation}
Note that in general $\left(  \ref{assumption_limit_3}\right)  $ cannot be
derived from $\left(  \ref{assumption_limit_0}\right)  $.

Before proceeding to the main result of this section, we need to transform the
equation $\left(  \ref{quasilinear_PDE_n}\right)  $\ into a more manageable
form. Expand the operator in the equation $\left(  \ref{quasilinear_PDE_n}%
\right)  $,%
\begin{align*}
&  \operatorname{div}\left(  \frac{a_{n}\left(  x,\left\vert \nabla
u\right\vert \right)  }{\left\vert \nabla u\right\vert }\nabla u\right)
=\left\langle \frac{\partial_{x}a_{n}\left(  x,\left\vert \nabla u\right\vert
\right)  }{\left\vert \nabla u\right\vert },\nabla u\right\rangle +\frac
{a_{n}\left(  x,\left\vert \nabla u\right\vert \right)  }{\left\vert \nabla
u\right\vert }\Delta u\\
&  +\frac{a_{n}\left(  x,\left\vert \nabla u\right\vert \right)  }{\left\vert
\nabla u\right\vert }\left(  \frac{\left\vert \nabla u\right\vert \partial
_{s}a_{n}\left(  x,\left\vert \nabla u\right\vert \right)  }{a_{n}\left(
x,\left\vert \nabla u\right\vert \right)  }-1\right)  \frac{1}{\left\vert
\nabla u\right\vert ^{2}}\Delta_{\infty}u.
\end{align*}
We first take a look at the coefficient of $\Delta_{\infty}u$,
\[
C_{\Delta_{\infty}u}=\frac{a_{n}\left(  x,\left\vert \nabla u\right\vert
\right)  }{\left\vert \nabla u\right\vert }\left(  \frac{\left\vert \nabla
u\right\vert \partial_{s}a_{n}\left(  x,\left\vert \nabla u\right\vert
\right)  }{a_{n}\left(  x,\left\vert \nabla u\right\vert \right)  }-1\right)
\frac{1}{\left\vert \nabla u\right\vert ^{2}}.
\]
In view of $\left(  \ref{assumption_phi_x}\right)  $ and Lemma
\ref{lemma_delta2_condition}, as long as $p_{n}^{-}\geqslant2$, $a_{n}\left(
x,\left\vert \nabla u\right\vert \right)  /\left\vert \nabla u\right\vert $ is
well-defined for all $\left\vert \nabla u\right\vert \geqslant0$, and strictly
positive for $\left\vert \nabla u\right\vert >0$. Also the middle term of
$C_{\Delta_{\infty}u}$ in the bracket, ranging between $p_{n}^{-}-2$ and
$p_{n}^{+}-2$, is strictly positive if $p_{n}^{-}>2$. Since we are only
interested in aymptotic behaviour, we may assume that $p_{n}^{-}$ is large so
that all terms appearing in $C_{\Delta_{\infty}u}$ are well-defined and
strictly positive whenever $\left\vert \nabla u\right\vert \neq0$. Hence it is
safe to divide both sides of the equation $\left(  \ref{quasilinear_PDE_n}%
\right)  $ by $C_{\Delta_{\infty}u}$ whenever $\left\vert \nabla u\right\vert
\neq0$, then we get%
\begin{align*}
&  -\left\vert \nabla u\right\vert ^{2}\left(  \frac{\left\vert \nabla
u\right\vert \partial_{s}a_{n}\left(  x,\left\vert \nabla u\right\vert
\right)  }{a_{n}\left(  x,\left\vert \nabla u\right\vert \right)  }-1\right)
^{-1}\left\langle \frac{\partial_{x}a_{n}\left(  x,\left\vert \nabla
u\right\vert \right)  }{a_{n}\left(  x,\left\vert \nabla u\right\vert \right)
},\nabla u\right\rangle \\
&  -\left\vert \nabla u\right\vert ^{2}\left(  \frac{\left\vert \nabla
u\right\vert \partial_{s}a_{n}\left(  x,\left\vert \nabla u\right\vert
\right)  }{a_{n}\left(  x,\left\vert \nabla u\right\vert \right)  }-1\right)
^{-1}\Delta u-\Delta_{\infty}u=0,
\end{align*}
or equivalently%
\begin{equation}
-\left\vert \nabla u\right\vert ^{2}\left\langle \frac{\frac{\partial_{x}%
a_{n}\left(  x,\left\vert \nabla u\right\vert \right)  }{\left\vert \nabla
u\right\vert \partial_{s}a_{n}\left(  x,\left\vert \nabla u\right\vert
\right)  }}{1-\frac{a_{n}\left(  x,\left\vert \nabla u\right\vert \right)
}{\left\vert \nabla u\right\vert \partial_{s}a_{n}\left(  x,\left\vert \nabla
u\right\vert \right)  }},\nabla u\right\rangle -\left\vert \nabla u\right\vert
^{2}\left(  \frac{\left\vert \nabla u\right\vert \partial_{s}a_{n}\left(
x,\left\vert \nabla u\right\vert \right)  }{a_{n}\left(  x,\left\vert \nabla
u\right\vert \right)  }-1\right)  ^{-1}\Delta u-\Delta_{\infty}u=0.
\label{general_quasilinear_transformed}%
\end{equation}
Therefore a (viscosity) solution of $\left(  \ref{quasilinear_PDE_n}\right)  $
is also a (viscosity) solution of the transformed equation $\left(
\ref{general_quasilinear_transformed}\right)  $.

The next the main result of this section.

\begin{theorem}
\label{thm_convergence_main}Assume that $g$ satisfies $\left(
\ref{assumption_g_lip}\right)  ,$ $\Phi_{n}$ satisfies the normalization
condition $\left(  \ref{assumption_PHI_n_1}\right)  ,$ the local continuity
condition $\left(  \ref{PHI_condition_A11}\right)  ,$ the associated
Harnack-type condition $\left(  \ref{assumption_c_n}\right)  ,$ $a_{n}$
satisfies $\left(  \ref{eq:assume_F_cont}\right)  $ and $p_{n}^{-}\geqslant
m$. Assume also that $\left(  \ref{assumption_limit_1}\right)  \left(
\ref{assumption_limit_0}\right)  $ and $\left(  \ref{assumption_limit_3}%
\right)  $ are satisfied. Let $u_{n}$ be the unique continuous weak solution
for equation $\left(  \ref{quasilinear_PDE_n}\right)  $. Then there exists a
subseqence of $u_{n}$ such that $u_{n}$ converges uniformly to $u$ which is a
viscosity solution of the limiting equation,%
\[
\left\{
\begin{array}
[c]{ll}%
-\Delta_{\infty}u-\left\vert \nabla u\right\vert \left\langle \Lambda\left(
x,\left\vert \nabla u\right\vert \right)  ,\nabla u\right\rangle =0, &
x\in\Omega;\\
u=g, & x\in\partial\Omega.
\end{array}
\right.
\]
In particular $u\in C\left(  \bar{\Omega}\right)  \cap W^{1,\infty}\left(
\Omega\right)  $ and there is a constant $c\left(  \Omega,g,\mu\right)  >0$
such that%
\[
\left\Vert \nabla u\right\Vert _{L^{\infty}}\leqslant c\left(  \Omega
,g,\mu\right)  \text{.}%
\]

\end{theorem}

\begin{remark}
In fact, the whole sequence $u_{n}$ converges uniformly to $u$. But this
requires the unqiueness of the limiting equation, which will be proved in
Theorem \ref{thm_limit_equ_unique_main}.
\end{remark}

\begin{remark}
In view of $\left(  \ref{assumption_limit_0}\right)  \left(
\ref{assumption_limit_1}\right)  $, the value and continuity of $\Lambda$ at
$s=0$ is not important. We could, for simplicity, stipulate $\Lambda\left(
x,0\right)  =0$ for all $x\in\bar{\Omega}.$
\end{remark}

\begin{remark}
\label{rmk:inf_operator_continuity}Note that $\left(
\ref{Nassumption_limit_2}\right)  ,$ as an implication of $\left(
\ref{assumption_limit_0}\right)  \left(  \ref{assumption_limit_1}\right)  ,$
indicates that the operator in the limiting equation
\[
\left(  x,p,X\right)  \mapsto F\left(  x,p,X\right)  =-\text{tr}\left(
p\otimes pX\right)  -\left\vert p\right\vert \left\langle \Lambda\left(
x,\left\vert p\right\vert \right)  ,p\right\rangle
\]
is defined continuously everywhere. In particular $F\left(  x,0,X\right)  =0.$
\end{remark}

\begin{proof}
By virtue of Theorem \ref{thm_bounded_W1d}, there is a subsequence of $u_{n}$
converging uniformly to some $u\in C\left(  \bar{\Omega}\right)  \cap
W^{1,\infty}\left(  \Omega\right)  .$ It sufficies to show that $u$ is a
viscosity solution of the limit equation. We only prove that $u$ is a
viscosity supersolution, the proof for subsolution is analagous. Suppose that
$u-\phi$ has a strict local minimum at $x_{0}\in\Omega$ for some $\phi\in
C^{2}\left(  \bar{\Omega}\right)  $, $u\left(  x_{0}\right)  =\phi\left(
x_{0}\right)  $. There is no need to consider $x_{0}\in\partial\Omega$, since
$u$ is continuous up to the boundary and equals $g$ there. The proof is
divided into two cases, if $\nabla\phi\left(  x_{0}\right)  =0,$ the
conclusion is immediate in view of Remark \ref{rmk:inf_operator_continuity},
otherwise assume that $\nabla\phi\left(  x_{0}\right)  \neq0$. Since $u_{n}$
converges to $u$ uniformly, there exists $x_{n}\rightarrow x_{0}$ such that
$u_{n}-\phi$ attains a local minimum at $x_{n}$. We may also assume w.l.g.
that $\nabla\phi\left(  x_{n}\right)  \neq0$. By Theorem
\ref{thm_weaksol_viscositysol} $u_{n}$ is a viscosity solution of $\left(
\ref{quasilinear_PDE_n}\right)  $, so%
\begin{align*}
&  -\operatorname{div}\left(  \frac{a_{n}\left(  x_{n},\left\vert \nabla
\phi\left(  x_{n}\right)  \right\vert \right)  }{\left\vert \nabla\phi\left(
x_{n}\right)  \right\vert }\nabla\phi\left(  x_{n}\right)  \right) \\
&  =-\left\langle \frac{\partial_{x}a_{n}\left(  x_{n},\left\vert \nabla
\phi\left(  x_{n}\right)  \right\vert \right)  }{\left\vert \nabla\phi\left(
x_{n}\right)  \right\vert },\nabla\phi\left(  x_{n}\right)  \right\rangle
-\frac{a_{n}\left(  x_{n},\left\vert \nabla\phi\left(  x_{n}\right)
\right\vert \right)  }{\left\vert \nabla\phi\left(  x_{n}\right)  \right\vert
}\Delta\phi\left(  x_{n}\right)  -\\
&  \frac{a_{n}\left(  x_{n},\left\vert \nabla\phi\left(  x_{n}\right)
\right\vert \right)  }{\left\vert \nabla\phi\left(  x_{n}\right)  \right\vert
}\left(  \frac{\left\vert \nabla\phi\left(  x_{n}\right)  \right\vert
\partial_{s}a_{n}\left(  x_{n},\left\vert \nabla\phi\left(  x_{n}\right)
\right\vert \right)  }{a_{n}\left(  x_{n},\left\vert \nabla\phi\left(
x_{n}\right)  \right\vert \right)  }-1\right)  \frac{1}{\left\vert \nabla
\phi\left(  x_{n}\right)  \right\vert ^{2}}\Delta_{\infty}\phi\left(
x_{n}\right)  \geqslant0.
\end{align*}
Based on the foregoing discussion, $u_{n}$ is also a viscosity solution of the
transformed equation $\left(  \ref{general_quasilinear_transformed}\right)  $,
i.e.%
\begin{align*}
&  -\left\vert \nabla\phi\left(  x_{n}\right)  \right\vert ^{2}\left\langle
\frac{\frac{\partial_{x}a_{n}\left(  x_{n},\left\vert \nabla\phi\left(
x_{n}\right)  \right\vert \right)  }{\left\vert \nabla\phi\left(
x_{n}\right)  \right\vert \partial_{s}a_{n}\left(  x_{n},\left\vert \nabla
\phi\left(  x_{n}\right)  \right\vert \right)  }}{1-\frac{a_{n}\left(
x_{n},\left\vert \nabla\phi\left(  x_{n}\right)  \right\vert \right)
}{\left\vert \nabla\phi\left(  x_{n}\right)  \right\vert \partial_{s}%
a_{n}\left(  x_{n},\left\vert \nabla\phi\left(  x_{n}\right)  \right\vert
\right)  }},\nabla\phi\left(  x_{n}\right)  \right\rangle -\\
&  \left\vert \nabla\phi\left(  x_{n}\right)  \right\vert ^{2}\left(
\frac{\left\vert \nabla\phi\left(  x_{n}\right)  \right\vert \partial_{s}%
a_{n}\left(  x_{n},\left\vert \nabla\phi\left(  x_{n}\right)  \right\vert
\right)  }{a_{n}\left(  x_{n},\left\vert \nabla\phi\left(  x_{n}\right)
\right\vert \right)  }-1\right)  ^{-1}\Delta\phi\left(  x_{n}\right)
-\Delta_{\infty}\phi\left(  x_{n}\right)  \geqslant0.
\end{align*}
In view of $\left(  \ref{assumption_phi_n}\right)  $, for $n$ large,%
\[
0\leqslant\frac{a_{n}\left(  x_{n},\left\vert \nabla\phi\left(  x_{n}\right)
\right\vert \right)  }{\left\vert \nabla\phi\left(  x_{n}\right)  \right\vert
\partial_{s}a_{n}\left(  x_{n},\left\vert \nabla\phi\left(  x_{n}\right)
\right\vert \right)  }=\frac{\varphi_{n}\left(  x_{n},\left\vert \nabla
\phi\left(  x_{n}\right)  \right\vert \right)  }{\left\vert \nabla\phi\left(
x_{n}\right)  \right\vert \partial_{s}\varphi_{n}\left(  x_{n},\left\vert
\nabla\phi\left(  x_{n}\right)  \right\vert \right)  }\leqslant\frac{1}%
{p_{n}^{-}}\rightarrow0.
\]
Therefore upon sending $n$ to infinity we obtain from $\left(
\ref{assumption_limit_1}\right)  \left(  \ref{assumption_limit_0}\right)  $
that%
\[
-\left\langle \left\vert \nabla\phi\left(  x_{0}\right)  \right\vert
\Lambda\left(  x_{0},\left\vert \nabla\phi\left(  x_{0}\right)  \right\vert
\right)  ,\nabla\phi\left(  x_{0}\right)  \right\rangle -\Delta_{\infty}%
\phi\left(  x_{0}\right)  \geqslant0,
\]
proving that $u$ is a viscosity supersolution.
\end{proof}

Due to the uniqueness result by Jensen \cite{jensen1993uniqueness}, if
$\Lambda\left(  x,s\right)  \equiv0$, then the whole sequence $u_{n}$
converges uniformly to the unique solutions of the Dirichlet problem for
infinity Laplacian with boundary data $g$.

\section{\label{sec_lim_g2}Limit of the Solutions without $\left\Vert
g\right\Vert _{Lip\left(  \partial\Omega\right)  }\leqslant1$}

We revisit in this section the Lipschitz assumption on the boundary value $g$.
Recall that in showing the existence of the limit of the weak solutions
$u_{n}$, the boundedness of the energies $E_{n}\left(  u_{n}\right)  $, which
sufficiently ensures that $u_{n}$ is bounded in $W^{1,d+1}\left(
\Omega\right)  $, plays a major role. This is a result of the assumption that
$\left\Vert g\right\Vert _{Lip\left(  \partial\Omega\right)  }\leqslant1$, as
we see in the proof of Theorem \ref{thm_bounded_W1d}. We will see shortly
that, under appropriate conditions, $E_{n}\left(  u_{n}\right)  $ is unbounded
whenever $\left\Vert g\right\Vert _{Lip\left(  \partial\Omega\right)  }>1$.
This clearly echoes the $\Gamma$-convergence result of Theorem
\ref{Thm_gamma_convergence}$.$

In general, for smooth domain $\Omega$ and $v\in W^{1,\infty}\left(
\Omega\right)  $, we only have
\[
\left\Vert \nabla v\right\Vert _{L^{\infty}\left(  \Omega\right)  }%
\leqslant\left\Vert v\right\Vert _{Lip\left(  \Omega\right)  }.
\]
In the follow lemma, which is a generalization of \cite[Theorem 3.1]%
{manfredi2009p}, we will assume that $\Omega$ is convex so that%
\[
\left\Vert \nabla v\right\Vert _{L^{\infty}\left(  \Omega\right)  }=\left\Vert
v\right\Vert _{Lip\left(  \Omega\right)  }.
\]

\begin{lemma}
\label{lm_energies_unbounded}Assume $\Omega$ is convex and $\left(
\ref{assumption_c_n}\right)  $ holds. Let $u_{n}$ be the minimizer of $E_{n}$
in $g+W_{0}^{1,\Phi_{n}}\left(  \Omega\right)  $. If $\left\Vert g\right\Vert
_{Lip\left(  \partial\Omega\right)  }>1$, then $E_{n}\left(  u_{n}\right)  $
is unbounded.
\end{lemma}

\begin{proof}
Suppose to the contrary that $E_{n}\left(  u_{n}\right)  $ is bounded by some
constant $M>0.$ Fix $m>d$. Let $p_{n}^{-}\geqslant m$. Recalling Remark
\ref{Rmk_g}, $g\in W^{1,\infty}\left(  \Omega\right)  $. As in the proof of
Lemma \ref{lm_bounded_W1d1} we have%
\begin{equation}
\left\Vert \nabla u_{n}\right\Vert _{L^{m}}\leqslant\left\vert \Omega
\right\vert ^{\frac{1}{m}-\frac{1}{p_{n}^{-}}}\left\Vert \nabla u_{n}%
\right\Vert _{L^{p_{n}^{-}}}\leqslant\left\vert \Omega\right\vert ^{\frac
{1}{m}-\frac{1}{p_{n}^{-}}}\left(  \left\vert \Omega\right\vert +\mu p_{n}%
^{-}E_{n}\left(  u_{n}\right)  \right)  ^{\frac{1}{p_{n}^{-}}}\text{,}
\label{lm_energies_unbounded_Inq1}%
\end{equation}
and%
\begin{equation}
\left\Vert u_{n}\right\Vert _{W^{1,m}}\leqslant c\left(  \Omega,m\right)
\left(  \left\Vert \nabla u_{n}\right\Vert _{L^{m}}+\left\Vert g\right\Vert
_{W^{1,\infty}}\right)  \text{.} \label{lm_energies_unbounded_Inq2}%
\end{equation}
Note that%
\[
\left\vert \Omega\right\vert ^{\frac{1}{p_{n}^{-}}}\leqslant\left(  \left\vert
\Omega\right\vert +\mu p_{n}^{-}E_{n}\left(  u_{n}\right)  \right)  ^{\frac
{1}{p_{n}^{-}}}\leqslant\left(  \left\vert \Omega\right\vert +\mu Mp_{n}%
^{-}\right)  ^{\frac{1}{p_{n}^{-}}}.
\]
It follows that%
\[
\lim_{n\rightarrow\infty}\left(  \left\vert \Omega\right\vert +\mu p_{n}%
^{-}E_{n}\left(  u_{n}\right)  \right)  ^{\frac{1}{p_{n}^{-}}}=1.
\]
By $\left(  \ref{lm_energies_unbounded_Inq1}\right)  $ and $\left(
\ref{lm_energies_unbounded_Inq2}\right)  ,$ $\left\Vert \nabla u_{n}%
\right\Vert _{L^{m}}$, hence $\left\Vert u_{n}\right\Vert _{W^{1,m}}$ is
bounded. Extracting a subsequence when necessary, we assume that $u_{n}$
converges weakly in $W^{1,m}$ to some $u_{\infty}\in W^{1,m}.$ Then taking
limit $n\rightarrow\infty$ in $\left(  \ref{lm_energies_unbounded_Inq1}%
\right)  $ yields%
\[
\left\Vert \nabla u_{\infty}\right\Vert _{L^{m}}\leqslant\liminf
_{n\rightarrow\infty}\left\Vert \nabla u_{n}\right\Vert _{L^{m}}%
\leqslant\left\vert \Omega\right\vert ^{\frac{1}{m}}\text{ for all }m>d,
\]
from which we obtain, by taking limit $m\rightarrow\infty,$ that $\nabla
u_{\infty}\in L^{\infty}$ and
\begin{equation}
\left\Vert \nabla u_{\infty}\right\Vert _{L^{\infty}}\leqslant1.
\label{lm_energies_unbounded_Inq3}%
\end{equation}
Since $p_{n}^{-}>d$, $u_{n}$ assumes the boundary data $g$ continuously, hence
$u_{\infty}=g$ on $\partial\Omega$. Since $\Omega$ is convex$,$ the Lipschitz
constant of any Lipschitz extension of $g$ into $\Omega$ is no less than
$\left\Vert g\right\Vert _{Lip\left(  \partial\Omega\right)  }$. Therefore
\[
1<\left\Vert g\right\Vert _{Lip\left(  \partial\Omega\right)  }=\left\Vert
\nabla g\right\Vert _{L^{\infty}\left(  \Omega\right)  }\leqslant\left\Vert
\nabla u_{\infty}\right\Vert _{L^{\infty}}.
\]
This is a contradiction to $\left(  \ref{lm_energies_unbounded_Inq3}\right)  $
which indicates that $E_{n}\left(  u_{n}\right)  $ is unbounded.
\end{proof}

Now we see that without $\left\Vert g\right\Vert _{Lip\left(  \partial
\Omega\right)  }\leqslant1$, the energies \textit{might} explode, and then,
according to the proof of Theorem \ref{thm_bounded_W1d}, there is no guarantee
that $u_{n}$ is bounded in $W^{1,m}\left(  \Omega\right)  $ $\left(
m>d\right)  $. But our purpose here is to show that, without the assumption
$\left\Vert g\right\Vert _{Lip\left(  \partial\Omega\right)  }\leqslant1$,
$u_{n}$ is still bounded in $W^{1,m}\left(  \Omega\right)  $ $\left(
m>d\right)  $, although $E_{n}\left(  u_{n}\right)  $ might be unbounded. Even
more so, the upper bound of $\left\Vert \nabla u_{n}\right\Vert _{L^{m}}$
$\left(  m>d\right)  $ can be written down with all contributing constants
explicitly identified. Before we can do this, we need one more assumption
which we call \textit{unit normalization condition}:
\begin{equation}
\Phi_{n}\left(  x,1\right)  =1\text{ for }a.e.\text{ }x\in\Omega,\text{
}n\geqslant1. \label{assumption_PHI_1_n}%
\end{equation}
This amounts to requiring that $\Phi_{n}$ uniformly satisfies the
normalization condition (ref. $\left(  \ref{PHI_normalization_x}\right)  $).

Despite its look, the unit normalization condition is not restrictive. It
aleady includes a large class of models seen in the literature, the most
common one being $\Phi\left(  x,s\right)  =s^{p\left(  x\right)  }$.

\begin{theorem}
\label{thm_bounded_inW1d_all_g}Assume that $\left(  \ref{assumption_c_n}%
\right)  $ holds and $\Phi_{n}$ satisfies $\left(  \ref{assumption_PHI_1_n}%
\right)  .$\ Let $\eta=\left\Vert g\right\Vert _{Lip\left(  \partial
\Omega\right)  }<\infty$ and $u_{n}$ be the unique minimizer of $E_{n}$ in
$g+W_{0}^{1,\Phi_{n}}\left(  \Omega\right)  $. Then, for all $m\geqslant1$ and
$n\geqslant1$ for which $p_{n}^{-}\geqslant m,$%
\[
\left\Vert \nabla u_{n}\right\Vert _{L^{m}}\leqslant12\mu\left(  1+\eta^{\mu
}\right)  \left(  1+\left\vert \Omega\right\vert \right)  ^{2}.
\]

\end{theorem}

\begin{proof}
Since%
\[
p_{n}^{-}\leqslant\frac{\varphi_{n}\left(  x,1\right)  }{\Phi_{n}\left(
x,1\right)  }\leqslant p_{n}^{+},
\]
in view of $\left(  \ref{assumption_PHI_1_n}\right)  $, we have, for large
$n$,%
\begin{equation}
p_{n}^{+}=p_{n}^{+}\Phi_{n}\left(  x,1\right)  \geqslant\varphi_{n}\left(
x,1\right)  \geqslant p_{n}^{-}\Phi_{n}\left(  x,1\right)  =p_{n}^{-}%
\geqslant1. \label{assumption_phi_1_n_lbd}%
\end{equation}
Then by virtue of the minimality of $u_{n}$ and Remark \ref{Rmk_g},
\begin{align*}
\frac{1}{p_{n}^{+}}\int_{\Omega}\Phi_{n}\left(  x,\left\vert \nabla
u_{n}\right\vert \right)  dx  &  \leqslant E_{n}\left(  u_{n}\right)
=\int_{\Omega}\frac{\Phi_{n}\left(  x,\left\vert \nabla u_{n}\right\vert
\right)  }{\varphi_{n}\left(  x,1\right)  }dx\\
&  \leqslant E_{n}\left(  g\right)  =\int_{\Omega}\frac{\Phi_{n}\left(
x,\left\vert \nabla g\right\vert \right)  }{\varphi_{n}\left(  x,1\right)
}dx\\
&  \leqslant\int_{\Omega}\Phi_{n}\left(  x,\left\vert \nabla g\right\vert
\right)  dx\\
&  \leqslant\int_{\Omega}\Phi_{n}\left(  x,\eta\right)  dx\leqslant
\max\left\{  \eta^{p_{n}^{-}},\eta^{p_{n}^{+}}\right\}  \int_{\Omega}\Phi
_{n}\left(  x,1\right)  dx\\
&  =\max\left\{  \eta^{p_{n}^{-}},\eta^{p_{n}^{+}}\right\}  \left\vert
\Omega\right\vert .
\end{align*}
Invoking Lemma \ref{lm_modulo_finite}, we get%
\[
\min\left\{  \left\Vert \nabla u_{n}\right\Vert _{L^{\Phi_{n}\left(
\cdot\right)  }}^{p_{n}^{-}},\left\Vert \nabla u_{n}\right\Vert _{L^{\Phi
_{n}\left(  \cdot\right)  }}^{p_{n}^{+}}\right\}  \leqslant p_{n}^{+}%
\max\left\{  \eta^{p_{n}^{-}},\eta^{p_{n}^{+}}\right\}  \left\vert
\Omega\right\vert .
\]
It follows that%
\[
\left\Vert \nabla u_{n}\right\Vert _{L^{\Phi_{n}\left(  \cdot\right)  }%
}\leqslant\max\left\{  \eta^{\frac{p_{n}^{+}}{p_{n}^{-}}},\eta^{\frac
{p_{n}^{-}}{p_{n}^{+}}}\right\}  \max\left\{  \left(  p_{n}^{+}\left\vert
\Omega\right\vert \right)  ^{\frac{1}{p_{n}^{-}}},\left(  p_{n}^{+}\left\vert
\Omega\right\vert \right)  ^{\frac{1}{p_{n}^{+}}}\right\}  .
\]
The maximum of the function $x\mapsto x^{\frac{1}{x}}$ on $\left[
1,\infty\right)  $ is $e^{\frac{1}{e}}$ which is less than $3$.\ Recalling
$\left(  \ref{assumption_c_n}\right)  $ and $p_{n}^{-}\geqslant1$, we get
\begin{align*}
&  \left\Vert \nabla u_{n}\right\Vert _{L^{\Phi_{n}\left(  \cdot\right)  }%
}\leqslant\max\left\{  \eta^{\frac{p_{n}^{+}}{p_{n}^{-}}},\eta^{\frac
{p_{n}^{-}}{p_{n}^{+}}}\right\}  \max\left\{  \left(  p_{n}^{-}\right)
^{\frac{1}{p_{n}^{-}}}\left(  \mu\left\vert \Omega\right\vert \right)
^{\frac{1}{p_{n}^{-}}},\left(  p_{n}^{+}\right)  ^{\frac{1}{p_{n}^{+}}%
}\left\vert \Omega\right\vert ^{\frac{1}{p_{n}^{+}}}\right\} \\
&  \leqslant3\max\left\{  \eta^{\frac{p_{n}^{+}}{p_{n}^{-}}},\eta^{\frac
{p_{n}^{-}}{p_{n}^{+}}}\right\}  \max\left\{  \left(  \mu\left\vert
\Omega\right\vert \right)  ^{\frac{1}{p_{n}^{-}}},\left\vert \Omega\right\vert
^{\frac{1}{p_{n}^{+}}}\right\} \\
&  \leqslant3\mu\max\left\{  1,\eta^{\mu}\right\}  \max\left\{  1,\left\vert
\Omega\right\vert \right\} \\
&  \leqslant3\mu\left(  1+\eta^{\mu}\right)  \left(  1+\left\vert
\Omega\right\vert \right)
\end{align*}
This shows that $\left\Vert \nabla u_{n}\right\Vert _{L^{\Phi_{n}\left(
\cdot\right)  }}$ is bounded, although the energies $E_{n}\left(
u_{n}\right)  $ might not be. Note that the upper bound of $\left\Vert \nabla
u_{n}\right\Vert _{L^{\Phi_{n}\left(  \cdot\right)  }}$ is independent of $n$.
Further examination of the assumptions then reveals that $\nabla u_{n}$ is
bounded in $L^{m}$. Indeed, by $\left(  \ref{assumption_PHI_1_n}\right)  $ and
Lemma \ref{lm_normalization}, the normalization $\left(
\ref{PHI_normalization_x}\right)  $ for $\Phi_{n}\left(  \cdot\right)  $ is
fulfilled with
\[
\beta_{n}=1\text{ for all }n.
\]
By virture of Lemma \ref{lemma_delta2_condition}\ and \cite[Lemma
4.4]{harjulehto2016generalized}, the space $L^{\Phi_{n}\left(  \cdot\right)
}$ can be imbedded into $L^{m}$ with the imbedding constant independent of $m$
and $n$ for which $p_{n}^{-}\geqslant m$, i.e.
\begin{equation}
\left\Vert \nabla u_{n}\right\Vert _{L^{m}}\leqslant K_{n}\left\Vert \nabla
u_{n}\right\Vert _{L^{\Phi_{n}\left(  \cdot\right)  }},
\label{thm_bounded_inW1d_all_g_Inq1}%
\end{equation}
where%
\[
K_{n}=2\left(  \frac{1}{\beta_{n}^{m}}+\frac{1}{\beta_{n}^{m}}\left\vert
\Omega\right\vert +1\right)  \leqslant4\left(  1+\left\vert \Omega\right\vert
\right)  .
\]
This in combination with the boundedness of $\left\Vert \nabla u_{n}%
\right\Vert _{L^{\Phi_{n}\left(  \cdot\right)  }}$ proves that $\left\{
\nabla u_{n}:p_{n}^{-}\geqslant m\right\}  $ is bounded in $L^{m}$.
\end{proof}

Going back a little, we can see that the proof of Theorem
\ref{thm_convergence_main} relies on Theorem \ref{thm_bounded_W1d}.\ A review
of the proof of Theorem \ref{thm_bounded_W1d} shows that the
assumption$\ \left\Vert g\right\Vert _{Lip\left(  \partial\Omega\right)
}\leqslant1$ is only used to ensure the boundedness of $E_{n}\left(
u_{n}\right)  $, which in turn yields the boundedness of $\left\Vert
u_{n}\right\Vert _{W^{1,m}\left(  \Omega\right)  }$. With Theorem
\ref{thm_bounded_inW1d_all_g} at hand, the boundedness of $\left\Vert
u_{n}\right\Vert _{W^{1,m}\left(  \Omega\right)  }$ could be obtained directly
from Lemma \ref{lm_bounded_W1d1} (iii). Hence Theorem
\ref{thm_convergence_main} still holds, but this time without the assumption
$\left\Vert g\right\Vert _{Lip\left(  \partial\Omega\right)  }\leqslant1.$
Therefore we get the following result.

\begin{theorem}
\label{thm_convergence_main_general}Assume that $\left(  \ref{assumption_c_n}%
\right)  $ $\left(  \ref{assumption_limit_1}\right)  \left(
\ref{assumption_limit_0}\right)  $ and $\left(  \ref{assumption_limit_3}%
\right)  $ are satisfied. Let $u_{n}$ be the unique continuous weak solution
for equation $\left(  \ref{quasilinear_PDE_n}\right)  $. Then there exists a
subseqence of $u_{n}$ such that $u_{n}$ converges uniformly to $u$ which is a
viscosity solution of the limiting equation,%
\[
\left\{
\begin{array}
[c]{ll}%
-\Delta_{\infty}u-\left\vert \nabla u\right\vert \left\langle \Lambda\left(
x,\left\vert \nabla u\right\vert \right)  ,\nabla u\right\rangle =0, &
x\in\Omega;\\
u=g, & x\in\partial\Omega.
\end{array}
\right.
\]
In particular $u\in C\left(  \bar{\Omega}\right)  \cap W^{1,\infty}\left(
\Omega\right)  $ and there is a constant $c\left(  \Omega,g,\mu\right)  >0$
such that%
\[
\left\Vert \nabla u\right\Vert _{L^{\infty}}\leqslant c\left(  \Omega
,g,\mu\right)  \text{.}%
\]

\end{theorem}

As a final remark to this section, we give a brief review on the connections
between Gamma convergence of the natural energies $E_{n}$ and the convergence
of $u_{n\text{.}}.$ Clearly the convergence of $E_{n}$ cannot be cut off from
the convergence of $u_{n\text{.}}$. They are closed related. With $\left\Vert
g\right\Vert _{Lip\left(  \partial\Omega\right)  }\leqslant1,$ the convergence
of $u_{n}$ has been suggested by the Gamma convergence, i.e. Theorem
\ref{Thm_gamma_convergence}. It is also implied that, with $\left\Vert
g\right\Vert _{Lip\left(  \partial\Omega\right)  }>1$, the energies
$E_{n}\left(  u_{n}\right)  $ would explode as we have proved in Lemma
\ref{lm_energies_unbounded}. In this sense, in the case $\left\Vert
g\right\Vert _{Lip\left(  \partial\Omega\right)  }\leqslant1$, Theorem
\ref{Thm_gamma_convergence} serves as a prelude to Theorem
\ref{thm_convergence_main}. Therefore, in the case $\left\Vert g\right\Vert
_{Lip\left(  \partial\Omega\right)  }>1,$ it is natural to ask whether there
are some some form of Gamma convergence results that echoes the convergence of
$u_{n}$? The answer is clearly affirmative in the classical case: Let $\Omega$
be bounded. If $\varpi_{n}$ converges weakly to $\varpi\in L^{\infty}\left(
\Omega\right)  $ in $L^{m}$ for some $m\geqslant1$, then%
\begin{equation}
\left\Vert \varpi\right\Vert _{L^{\infty}}\leqslant\liminf_{n\rightarrow
\infty}\left\Vert \varpi_{n}\right\Vert _{L^{n}}.
\label{Inq_classical_norm_lsc}%
\end{equation}
This can be proved by employing H\"{o}lder's inequality%
\[
\left\Vert \varpi_{n}\right\Vert _{L^{m}}\leqslant\left\vert \Omega\right\vert
^{\frac{1}{m}-\frac{1}{n}}\left\Vert \varpi_{n}\right\Vert _{L^{n}}\text{ for
all }n\geqslant m.
\]
Successively sending $n\rightarrow\infty$ and then $m\rightarrow\infty$ yields
$\left(  \ref{Inq_classical_norm_lsc}\right)  .$ This, expressed in the
language of Gamma convergence, is $\Gamma\left(  L^{1}\left(  \Omega\right)
\right)  $-$\lim\limits_{n\rightarrow\infty}\mathcal{\tilde{E}}_{n}%
=\mathcal{\tilde{E}}_{\infty},$ where%
\[
\mathcal{\tilde{E}}_{n}:L^{1}\left(  \Omega\right)  \mapsto\left\{
\begin{array}
[c]{ll}%
\left\Vert \nabla u\right\Vert _{L^{n}}, & u\in W^{1,n}\left(  \Omega\right)
;\\
\infty, & \text{otherwise},
\end{array}
\right.
\]
and%
\[
\mathcal{\tilde{E}}_{\infty}:L^{1}\left(  \Omega\right)  \mapsto\left\{
\begin{array}
[c]{ll}%
\left\Vert \nabla u\right\Vert _{L^{\infty}}, & u\in W^{1,\infty}\left(
\Omega\right)  ;\\
\infty, & \text{otherwise}.
\end{array}
\right.
\]
One implication of this Gamma convergence is that: if $u_{n}$ minimizes
$\mathcal{\tilde{E}}_{n}$ over $g+W^{1,n}\left(  \Omega\right)  $, then any
cluster point, say $u$, of $u_{n}$ should minnimize $\mathcal{\tilde{E}%
}_{\infty},$ i.e.%
\[
u=\arg\min_{w\in g+W_{0}^{1,\infty}\left(  \Omega\right)  }\left\Vert \nabla
w\right\Vert _{L^{\infty}}.
\]
Hence $u$ is infinity harmonic, i.e. the Gamma convergence of $\mathcal{\tilde
{E}}_{n}$ echoes the convergence of $u_{n}$. This can be regarded as an
alternative proof for the classical result that $-\Delta_{p}$ is approximately
an infinity Laplacian when the constant exponent $p$ is large.

Coming back to the Orlicz setting, in general, we cannot simply replace
$\left\Vert \nabla u\right\Vert _{L^{n}}$ with $\left\Vert \nabla u\right\Vert
_{L^{\Phi_{n}\left(  \cdot\right)  }}$ and expect a clean analogy for the
above to hold. One reason is that the Luxemburg norm of an Orlicz space is not
a complete parallel for the usual $L^{p}$ space. For example, $\Phi\left(
x,s\right)  =s^{p}$ $\left(  p\geqslant1\right)  $ a generalized $\Phi
$-function, its Luxemburg norm $\left\Vert \cdot\right\Vert _{L^{\Phi\left(
\cdot\right)  }}$ is equivalent but not equal to $\left\Vert \cdot\right\Vert
_{L^{p}}$. Some scalings are needed to make them agree. For detailed
dicussions of the scalings of generalized Orlicz spaces and its consequences
on Luxemburg norms, we refer to \cite[Section 3.4]{harjulehto2019orlicz}. That
being said, there are some situations where such analogy exists.

\begin{theorem}
Assume that $\Phi_{n}$ satisfies $\left(  \ref{assumption_Phi_n}\right)  $ and
the unit normalization condition $\left(  \ref{assumption_PHI_1_n}\right)  $.
Assume in addition that (a stronger version of $\left(  \ref{assumption_c_n}%
\right)  \,$holds, i.e.)%
\[
\lim_{n\rightarrow\infty}p_{n}^{-}=\infty\text{ and }\lim_{n\rightarrow\infty
}\frac{p_{n}^{+}}{p_{n}^{-}}=1.
\]
Let%
\[
\tilde{E}_{n}:L^{1}\left(  \Omega\right)  \mapsto\left\{
\begin{array}
[c]{ll}%
\left\vert \left\Vert \nabla u\right\Vert \right\vert _{L^{\Phi_{n}\left(
\cdot\right)  }}, & u\in W^{1,\Phi_{n}\left(  \cdot\right)  }\left(
\Omega\right)  ;\\
\infty, & \text{otherwise},
\end{array}
\right.
\]
and%
\[
\tilde{E}_{\infty}:L^{1}\left(  \Omega\right)  \mapsto\left\{
\begin{array}
[c]{ll}%
\left\Vert \nabla u\right\Vert _{L^{\infty}}, & u\in W^{1,\infty}\left(
\Omega\right)  ;\\
\infty, & \text{otherwise}.
\end{array}
\right.
\]
Then $\Gamma\left(  L^{1}\left(  \Omega\right)  \right)  $-$\lim
\limits_{n\rightarrow\infty}\tilde{E}_{n}=\tilde{E}_{\infty}.$
\end{theorem}

This can be proved by similar arguments as \cite[Theorem 2]{bocea2015gamma}.
We omit the details.

\section{\label{sec_comparison}Comparison Principle and Uniqueness of the
Limit Equation}

Let $u_{n}$ be the sequence of the unique continuous weak solution for
equation $\left(  \ref{quasilinear_PDE_n}\right)  $. In Theorem
\ref{thm_convergence_main} or Theorem \ref{thm_convergence_main_general}, we
have established that $\left\{  \left\Vert \nabla u_{n}\right\Vert _{L^{m}%
}:p_{n}^{-}\geqslant m\right\}  $ is bounded for all $m>d$ and there is a
subsequence of $u_{n}$ which converges uniformly to the following limit
equation
\begin{equation}
\left\{
\begin{array}
[c]{ll}%
-\Delta_{\infty,\Lambda}u=0, & x\in\Omega;\\
u=g, & x\in\partial\Omega,
\end{array}
\right.  \label{Equ_limit}%
\end{equation}
where%
\[
-\Delta_{\infty,\Lambda}u=-\Delta_{\infty}u-\left\vert \nabla u\right\vert
\left\langle \Lambda\left(  x,\left\vert \nabla u\right\vert \right)  ,\nabla
u\right\rangle .
\]

\begin{remark}
\label{Rmk:convergence_whole}We shall show that the equation $\left(
\ref{Equ_limit}\right)  $ has a unique continuous viscosity solution. Since
$u_{n}$ is continuous and $\left\{  \left\Vert \nabla u_{n}\right\Vert
_{L^{m}}:p_{n}^{-}\geqslant m\right\}  $ is bounded for all $m>d,$ the uniform
convergence limit of every subsequence of $u_{n}$\ is again continuous.\ Hence
the uniqueness result indicates that the whole sequence $u_{n}$, not just a
subsequence, converges uniformly to the unique continuous viscosity solution
of $\left(  \ref{Equ_limit}\right)  .$
\end{remark}

The equation $\left(  \ref{Equ_limit}\right)  $ is not coercive in $u$, so the
general comparison principle for nonlinear equations does not apply.
Comparison principles for such equations are generally difficult to obtain.
The most notable example is the infinity Laplacian. The uniqueness of the
infinity Laplacian equation was first solved by Jensen
\cite{jensen1993uniqueness} using a combination of viscosity solution
techniques and approximation procedures, which in the end leads to a weak
comparison principle for infinity Laplacian equation. The uniqueness then
follows from the comparison principle. For uniqueness of nonlinear equations
without $\left(  x,u\right)  $-dependency, see also \cite{barles2001existence}%
. Jensen's method was later extended to $p\left(  x\right)  $-Laplacian
problem \cite{lindqvist2010curious}. Here we first prove a weak comparison
principle for $\left(  \ref{Equ_limit}\right)  $ by combining techniques of
\cite{jensen1993uniqueness} and \cite{lindqvist2010curious}, then derive the
uniqueness for the limiting equation. In the course of proving the uniqueness,
the techinical inequalities which we prove in the appendix have a crucial role
to play.

Fix $0<\epsilon<1.$\ We introduce the auxiliary\textit{ upper equation}%
\begin{equation}
\left\{
\begin{array}
[c]{ll}%
\min\left\{  \left\vert \nabla u\right\vert -\epsilon,-\Delta_{\infty,\Lambda
}u\right\}  =0, & x\in\Omega;\\
u=g, & x\in\partial\Omega,
\end{array}
\right.  \label{Equ_limit_upper}%
\end{equation}
which provides us with supersolutions of $\left(  \ref{Equ_limit}\right)  $,
and the auxiliary \textit{lower equation}%
\begin{equation}
\left\{
\begin{array}
[c]{ll}%
\max\left\{  \epsilon-\left\vert \nabla u\right\vert ,-\Delta_{\infty,\Lambda
}u\right\}  =0, & x\in\Omega;\\
u=g, & x\in\partial\Omega,
\end{array}
\right.  \label{Equ_limit_lower}%
\end{equation}
which provides us with subsolutions of $\left(  \ref{Equ_limit}\right)  $.

\begin{theorem}
\label{Thm_solution_upper_lower}Let $g\in Lip\left(  \partial\Omega\right)  $.
Assume that $\left(  \ref{assumption_limit_1}\right)  \left(
\ref{assumption_limit_0}\right)  \left(  \ref{assumption_limit_3}\right)  $
are satisfied$.$ Then the upper equation $\left(  \ref{Equ_limit_upper}%
\right)  $ (resp. the lower equation $\left(  \ref{Equ_limit_lower}\right)  $)
has a viscosity solution $u^{+}\in C\left(  \bar{\Omega}\right)  \cap
W^{1,\infty}\left(  \Omega\right)  $ (resp. $u^{-}\in C\left(  \bar{\Omega
}\right)  \cap W^{1,\infty}\left(  \Omega\right)  $). Moreover $u^{+}$ is a
viscosity supersolution of%
\[
\left\vert \nabla u^{+}\right\vert \geqslant\epsilon.
\]
(resp. $u^{-}$ is a viscosity subsolution of $\left\vert \nabla u^{-}%
\right\vert \geqslant\epsilon.$)
\end{theorem}

\begin{proof}
The proof is given for the upper equation $\left(  \ref{Equ_limit_upper}%
\right)  $ only, the proof for the lower equation $\left(
\ref{Equ_limit_lower}\right)  $ is analogous. We introduce the auxiliary
functionals,
\begin{equation}
\mathcal{E}_{n}^{+}\left(  u\right)  =
{\displaystyle\int_{\Omega}}
\dfrac{\Phi_{n}\left(  x,\left\vert \nabla u\right\vert \right)  }{\varphi
_{n}\left(  x,1\right)  }dx-
{\displaystyle\int_{\Omega}}
a_{n}\left(  x,\epsilon\right)  udx,\label{Thm_solution_upper_lower_n}%
\end{equation}
where $\varphi_{n}$, $\Phi_{n}$ satisfy all conditions in earlier sections
$\left(  \ref{assumption_phi_n}\right)  $, $\left(  \ref{assumption_Phi_n}%
\right)  $, $\left(  \ref{assumption_c_n}\right)  $ and%
\[
a_{n}\left(  x,s\right)  \triangleq\dfrac{\varphi_{n}\left(  x,s\right)
}{\varphi_{n}\left(  x,1\right)  }%
\]
satisfies $\left(  \ref{eq:assume_F_cont}\right)  $ (see Remark
\ref{rmk_operator_continuity}). Critical points of $\mathcal{E}_{n}^{+}\left(
\cdot\right)  $ satisfies the Euler equation%
\begin{equation}
-\operatorname{div}\left(  \dfrac{a_{n}\left(  x,\left\vert \nabla u\left(
x\right)  \right\vert \right)  }{\left\vert \nabla u\left(  x\right)
\right\vert }\nabla u\left(  x\right)  \right)  =a_{n}\left(  x,\epsilon
\right)  ,\text{ }x\in\Omega.\label{Thm_solution_upper_lower_eps}%
\end{equation}

\textbf{Step I}: We claim that the upper equation $\left(
\ref{Equ_limit_upper}\right)  $ is the limit of the Euler equations $\left(
\ref{Thm_solution_upper_lower_eps}\right)  $ in the sense that there is
subsequence $n_{k}$ such that each $u_{n_{k}}$ is a solution of $\left(
\ref{Thm_solution_upper_lower_eps}\right)  $ and $u_{n_{k}}$ converges
uniformly to a solution of the upper equation when $k$ is sent to infinity.
Comparing with the energy $E_{n}$ we have previously seen, $\mathcal{E}%
_{n}^{+}\left(  \cdot\right)  $ has an extra term $\int_{\Omega}a_{n}\left(
x,\epsilon\right)  udx$. But this does not pose additional difficulty, since
the extra term is uniformly bounded. Indeed, in view of $\left(
\ref{assumption_phi_n}\right)  $ and Lemma \ref{lemma_delta2_condition}, we
have%
\[
\epsilon^{p_{n}^{+}-1}\varphi_{n}\left(  x,1\right)  \leqslant\varphi
_{n}\left(  x,\epsilon\right)  \leqslant\epsilon^{p_{n}^{-}-1}\varphi
_{n}\left(  x,1\right)  \text{ for }x\in\bar{\Omega}.
\]
So%
\[
\lim_{n\rightarrow\infty}a_{n}\left(  x,\epsilon\right)  =0\text{ uniformly in
}x\in\bar{\Omega}.
\]
Similar to Theorem \ref{Thm_existence} and Theorem
\ref{thm_weaksol_viscositysol}, for each $n$, $\mathcal{E}_{n}^{+}\left(
\cdot\right)  $ has a unique minimizer $u_{n}^{+}\in C\left(  \bar{\Omega
}\right)  \cap W^{1,\Phi_{n}\left(  \cdot\right)  }\left(  \Omega\right)  $
which is also a viscosity solution of $\left(
\ref{Thm_solution_upper_lower_eps}\right)  $. In addition, there is $c\left(
\Omega,\mu,g\right)  >0$ such that, for all $m>d$ and $n$ for which $p_{n}%
^{-}\geqslant m$,
\begin{equation}
\left\Vert \nabla u_{n}^{+}\right\Vert _{L^{m}}\leqslant c\left(  \Omega
,\mu,g\right)  \text{.} \label{Thm_solution_upper_lower_Inq1}%
\end{equation}
Then proceeding as in Theorem \ref{thm_bounded_W1d}, we can find a subequence
of $u_{n}^{+}$ converging uniformly to some $u^{+}\in C\left(  \bar{\Omega
}\right)  \cap W^{1,\infty}\left(  \Omega\right)  .$ Moreover%
\[
\left\Vert \nabla u^{+}\right\Vert _{L^{\infty}}\leqslant c\left(  \Omega
,\mu,g\right)  \text{.}%
\]

\textbf{Step II}: Now we show that $u^{+}$ is a viscosity solution of the
upper equation $\left(  \ref{Equ_limit_upper}\right)  $. First we show that
$u^{+}$ is a viscosity supersolution. We already know from Step I that $u^{+}$
is continuous up to the boundary and equals $g$ there, so there is no need to
consider $x_{0}\in\partial\Omega$. Suppose that $u^{+}-\phi$ has a strict
local minimum at $x_{0}\in\Omega$ for some $\phi\in C^{2}\left(  \bar{\Omega
}\right)  $, $u^{+}\left(  x_{0}\right)  =\phi\left(  x_{0}\right)  $. We need
to show that%
\begin{equation}
\min\left\{  \left\vert \nabla\phi\right\vert -\epsilon,-\Delta_{\infty
,\Lambda}\phi\right\}  \geqslant0.\label{Thm_solution_upper_lower_Inq1_2}%
\end{equation}
Since $u_{n}$ converges to $u^{+}$ uniformly, there exists $x_{n}\rightarrow
x_{0}$ such that $u_{n}-\phi$ attains a local minimum at $x_{n}$. Note that
due to $\left(  \ref{Thm_solution_upper_lower_eps}\right)  $ and continuity of
the operator (see Remark \ref{rmk_operator_continuity}), we may assume
$\nabla\phi\left(  x_{n}\right)  \neq0$. Repeating some calculations that we
did right before Theorem \ref{thm_convergence_main}, we see that a viscosity
solution of $\left(  \ref{Thm_solution_upper_lower_eps}\right)  $ is also a
viscosity solution of the transformed equation similar to $\left(
\ref{general_quasilinear_transformed}\right)  $, i.e.%
\begin{align}
&  -\left\vert \nabla u\right\vert ^{2}\left\langle \frac{\frac{\partial
_{x}a_{n}\left(  x,\left\vert \nabla u\right\vert \right)  }{\left\vert \nabla
u\right\vert \partial_{s}a_{n}\left(  x,\left\vert \nabla u\right\vert
\right)  }}{1-\frac{a_{n}\left(  x,\left\vert \nabla u\right\vert \right)
}{\left\vert \nabla u\right\vert \partial_{s}a_{n}\left(  x,\left\vert \nabla
u\right\vert \right)  }},\nabla u\right\rangle -\left\vert \nabla u\right\vert
^{2}\left(  \frac{\left\vert \nabla u\right\vert \partial_{s}a_{n}\left(
x,\left\vert \nabla u\right\vert \right)  }{a_{n}\left(  x,\left\vert \nabla
u\right\vert \right)  }-1\right)  ^{-1}\Delta u-\Delta_{\infty}%
u\label{Thm_solution_upper_lower_eps_transformed}\\
&  =\frac{a_{n}\left(  x,\epsilon\right)  }{a_{n}\left(  x,\left\vert \nabla
u\right\vert \right)  }\left\vert \nabla u\right\vert ^{3}\left(
\frac{\left\vert \nabla u\right\vert \partial_{s}a_{n}\left(  x,\left\vert
\nabla u\right\vert \right)  }{a_{n}\left(  x,\left\vert \nabla u\right\vert
\right)  }-1\right)  ^{-1}.\nonumber
\end{align}
Therefore, since $u_{n}$ is a viscosity solution of $\left(
\ref{Thm_solution_upper_lower_eps}\right)  $, it is also a viscosity solution
of $\left(  \ref{Thm_solution_upper_lower_eps_transformed}\right)  $ and\ we
obtain%
\begin{align}
&  -\left\vert \nabla\phi\left(  x_{n}\right)  \right\vert ^{2}\left\langle
\frac{\frac{\partial_{x}a_{n}\left(  x_{n},\left\vert \nabla\phi\left(
x_{n}\right)  \right\vert \right)  }{\left\vert \nabla\phi\left(
x_{n}\right)  \right\vert \partial_{s}a_{n}\left(  x_{n},\left\vert \nabla
\phi\left(  x_{n}\right)  \right\vert \right)  }}{1-\frac{a_{n}\left(
x_{n},\left\vert \nabla\phi\left(  x_{n}\right)  \right\vert \right)
}{\left\vert \nabla\phi\left(  x_{n}\right)  \right\vert \partial_{s}%
a_{n}\left(  x_{n},\left\vert \nabla\phi\left(  x_{n}\right)  \right\vert
\right)  }},\nabla\phi\left(  x_{n}\right)  \right\rangle
\label{Thm_solution_upper_lower_Inq2}\\
&  -\left\vert \nabla\phi\left(  x_{n}\right)  \right\vert ^{2}\left(
\frac{\left\vert \nabla\phi\left(  x_{n}\right)  \right\vert \partial_{s}%
a_{n}\left(  x_{n},\left\vert \nabla\phi\left(  x_{n}\right)  \right\vert
\right)  }{a_{n}\left(  x_{n},\left\vert \nabla\phi\left(  x_{n}\right)
\right\vert \right)  }-1\right)  ^{-1}\Delta\phi\left(  x_{n}\right)
-\Delta_{\infty}\phi\left(  x_{n}\right)  \nonumber\\
&  \geqslant\frac{a_{n}\left(  x_{n},\epsilon\right)  }{a_{n}\left(
x,\left\vert \nabla\phi\left(  x_{n}\right)  \right\vert \right)  }\left\vert
\nabla\phi\left(  x_{n}\right)  \right\vert ^{3}\left(  \frac{\left\vert
\nabla\phi\right\vert \partial_{s}a_{n}\left(  x_{n},\left\vert \nabla
\phi\left(  x_{n}\right)  \right\vert \right)  }{a_{n}\left(  x_{n},\left\vert
\nabla\phi\left(  x_{n}\right)  \right\vert \right)  }-1\right)
^{-1}.\nonumber
\end{align}
By Lemma \ref{lemma_delta2_condition},
\begin{align}
&  \min\left\{  \left\vert \frac{\epsilon}{\left\vert \nabla\phi\left(
x_{n}\right)  \right\vert }\right\vert ^{p_{n}^{-}-1},\left\vert
\frac{\epsilon}{\left\vert \nabla\phi\left(  x_{n}\right)  \right\vert
}\right\vert ^{p_{n}^{+}-1}\right\}  \label{Thm_solution_upper_lower_Inq2_3}\\
&  \leqslant\frac{a_{n}\left(  x_{n},\epsilon\right)  }{a_{n}\left(
x,\left\vert \nabla\phi\left(  x_{n}\right)  \right\vert \right)  }\nonumber\\
&  \leqslant\max\left\{  \left\vert \frac{\epsilon}{\left\vert \nabla
\phi\left(  x_{n}\right)  \right\vert }\right\vert ^{p_{n}^{-}-1},\left\vert
\frac{\epsilon}{\left\vert \nabla\phi\left(  x_{n}\right)  \right\vert
}\right\vert ^{p_{n}^{+}-1}\right\}  .\nonumber
\end{align}
If $\left\vert \nabla\phi\left(  x\right)  \right\vert <\epsilon$, then for
some $\gamma\in\left(  0,1\right)  $, $\left\vert \nabla\phi\left(
x_{n}\right)  \right\vert \leqslant\gamma\epsilon<\epsilon$ for all $n$
sufficiently large. Then the right hand side of $\left(
\ref{Thm_solution_upper_lower_Inq2}\right)  $ diverges to infinity due to
$\left(  \ref{Thm_solution_upper_lower_Inq2_3}\right)  $. This leads to a
contradiction, since the left hand side of $\left(
\ref{Thm_solution_upper_lower_Inq2}\right)  $ is bounded due to $\left(
\ref{assumption_limit_1}\right)  $ and $\left(  \ref{assumption_limit_0}%
\right)  $. Therefore $\left\vert \nabla\phi\left(  x\right)  \right\vert
\geqslant\epsilon$ and $\left(  \ref{Thm_solution_upper_lower_Inq1_2}\right)
$ is true. To see that $u^{+}$ is a viscosity subsolution of the upper
equation. Suppose that $u^{+}-\phi$ has a strict local maximum at $x_{0}%
\in\Omega$ for some $\phi\in C^{2}\left(  \bar{\Omega}\right)  $,
$u^{+}\left(  x_{0}\right)  =\phi\left(  x_{0}\right)  $. We need to show that%
\begin{equation}
\min\left\{  \left\vert \nabla\phi\right\vert -\epsilon,-\Delta_{\infty
,\Lambda}\phi\right\}  \leqslant0.\label{Thm_solution_upper_lower_Inq3}%
\end{equation}
Since $u_{n}$ converges to $u^{+}$ uniformly, there exists $x_{n}\rightarrow
x_{0}$ such that $u_{n}-\phi$ attains a local minimum at $x_{n}$. As before,
since $u_{n}$ is a viscosity solution of $\left(
\ref{Thm_solution_upper_lower_eps}\right)  $, we have%
\begin{align}
&  -\left\vert \nabla\phi\left(  x_{n}\right)  \right\vert ^{2}\left\langle
\frac{\frac{\partial_{x}a_{n}\left(  x_{n},\left\vert \nabla\phi\left(
x_{n}\right)  \right\vert \right)  }{\left\vert \nabla\phi\left(
x_{n}\right)  \right\vert \partial_{s}a_{n}\left(  x_{n},\left\vert \nabla
\phi\left(  x_{n}\right)  \right\vert \right)  }}{1-\frac{a_{n}\left(
x_{n},\left\vert \nabla\phi\left(  x_{n}\right)  \right\vert \right)
}{\left\vert \nabla\phi\left(  x_{n}\right)  \right\vert \partial_{s}%
a_{n}\left(  x_{n},\left\vert \nabla\phi\left(  x_{n}\right)  \right\vert
\right)  }},\nabla\phi\left(  x_{n}\right)  \right\rangle
\label{Thm_solution_upper_lower_Inq4}\\
&  -\left\vert \nabla\phi\left(  x_{n}\right)  \right\vert ^{2}\left(
\frac{\left\vert \nabla\phi\left(  x_{n}\right)  \right\vert \partial_{s}%
a_{n}\left(  x_{n},\left\vert \nabla\phi\left(  x_{n}\right)  \right\vert
\right)  }{a_{n}\left(  x_{n},\left\vert \nabla\phi\left(  x_{n}\right)
\right\vert \right)  }-1\right)  ^{-1}\Delta\phi\left(  x_{n}\right)
-\Delta_{\infty}\phi\left(  x_{n}\right)  \nonumber\\
&  \leqslant\frac{a_{n}\left(  x_{n},\epsilon\right)  }{a_{n}\left(
x,\left\vert \nabla\phi\left(  x_{n}\right)  \right\vert \right)  }\left\vert
\nabla\phi\left(  x_{n}\right)  \right\vert ^{3}\left(  \frac{\left\vert
\nabla\phi\right\vert \partial_{s}a_{n}\left(  x_{n},\left\vert \nabla
\phi\left(  x_{n}\right)  \right\vert \right)  }{a_{n}\left(  x_{n},\left\vert
\nabla\phi\left(  x_{n}\right)  \right\vert \right)  }-1\right)
^{-1}.\nonumber
\end{align}
If $\left\vert \nabla\phi\left(  x\right)  \right\vert \leqslant\epsilon$ then
$\left(  \ref{Thm_solution_upper_lower_Inq3}\right)  $ is clearly true.
Otherwise by virtue of $\left(  \ref{Thm_solution_upper_lower_Inq2_3}\right)
$ and the assumption $\left(  \ref{assumption_limit_0}\right)  $, the right
hand side of $\left(  \ref{Thm_solution_upper_lower_Inq4}\right)  $ goes to
zero as $n\rightarrow\infty$ and $\left(  \ref{Thm_solution_upper_lower_Inq3}%
\right)  $ is also true. So $u^{+}$ is a supersolution.
\end{proof}

Before we can prove the comparison principle, we need additional assumptions
on the limit function $\Lambda\left(  x,s\right)  $ (ref. $\left(
\ref{assumption_limit_1}\right)  $).

Let%
\[
\Theta\left(  x,s\right)  =\frac{\Lambda\left(  x,s\right)  }{s}.
\]
We assume that there is $\lambda>0$ such that for all $s>0$,
\[
\sup_{x\in\bar{\Omega},\text{ }t\in\left[  s,s+1\right]  }\left\vert
\frac{\Theta\left(  x,t\right)  -\Theta\left(  x,s\right)  }{t-s}\right\vert
\leqslant\lambda\frac{1}{s}.
\]
This in particular implies that, for all $a\in\left(  1,2\right)  ,$%
\begin{equation}
\left\vert \Theta\left(  x,as\right)  -\Theta\left(  x,s\right)  \right\vert
\leqslant\lambda\left\vert a-1\right\vert \text{ for all }x\in\bar{\Omega
},\text{ }s>0. \label{assumption_Lambda}%
\end{equation}
Obviously, the inequality $\left(  \ref{assumption_Lambda}\right)  $ is true
for any function $\Theta\left(  x,s\right)  $ for which $\left\vert
\partial_{s}\Theta\left(  x,s\right)  \right\vert $, up to a constant
multiplier, is no more than $1/s$. This requirement roughly says that the
mapping $s\mapsto\Theta\left(  x,s\right)  $ grows slower than $\ln s.$

First introduce the following utility function from \cite[Lemma 2.6]%
{juutinen1999eigenvalue}.

\begin{lemma}
\label{lm_untility_func}Let $\alpha>0$, $A>1$. Then the function
\[
\zeta\left(  s\right)  =\frac{1}{\alpha}\ln\left(  1+A\left(  e^{\alpha
s}-1\right)  \right)
\]
possesses the following properties:

(a) $\zeta\left(  s\right)  \in C^{2}\left(  \left[  0,\infty\right)  \right)
,$ $\zeta\left(  0\right)  =0,\zeta^{\prime}\left(  s\right)  >1$ and
$\zeta^{\prime\prime}\left(  s\right)  <0$ for all $s\geqslant0;$

(b) $\zeta$ approximate the idendity function as $A\rightarrow1+,$%
\[
0<\zeta\left(  s\right)  -s<\frac{A-1}{\alpha},\text{ }0<\zeta^{\prime}\left(
s\right)  -1<A-1;
\]

(c) for all $s\geqslant0$,
\[
-\frac{\zeta^{\prime\prime}\left(  s\right)  }{\zeta^{\prime}\left(  s\right)
}=\alpha\left(  \zeta^{\prime}\left(  s\right)  -1\right)  =\alpha\frac
{A-1}{1+A\left(  e^{\alpha s}-1\right)  }.
\]

\end{lemma}

\begin{theorem}
\label{Thm_comparison_limit_equ}Assume that $\Lambda\left(  x,s\right)  $
satisfies $\left(  \ref{assumption_Lambda}\right)  $, $u\in C\left(
\bar{\Omega}\right)  \cap W^{1,\infty}\left(  \Omega\right)  $ is a
subsolution of the equation $\left(  \ref{Equ_limit}\right)  $ and $v\in
C\left(  \bar{\Omega}\right)  \cap W^{1,\infty}\left(  \Omega\right)  $ is a
supersolution of the upper equation $\left(  \ref{Equ_limit_upper}\right)  $.
Then%
\[
\max_{x\in\Omega}\left(  u-v\right)  \leqslant\max_{x\in\partial\Omega}\left(
u-v\right)  .
\]
Analogously assume that $\tilde{v}\in C\left(  \bar{\Omega}\right)  \cap
W^{1,\infty}\left(  \Omega\right)  $ is a subsolution of the lower equation
$\left(  \ref{Equ_limit_lower}\right)  $. Then%
\[
\max_{x\in\Omega}\left(  \tilde{v}-u\right)  \leqslant\max_{x\in\partial
\Omega}\left(  \tilde{v}-u\right)  .
\]

\end{theorem}

\begin{proof}
We prove the theorem for the upper equation only. Suppose otherwise that%
\begin{equation}
\max_{x\in\Omega}\left(  u-v\right)  >\max_{x\in\partial\Omega}\left(
u-v\right)  . \label{Thm_comparison_limit_equ_Inq0}%
\end{equation}
Since $v$ is bounded, we may assume that $v>0$, otherwise replace $v$ with $v$
plus a large constant. The proof is completed in two steps.

\textbf{Step I}: We show that, with appropriate choice of constants,
$w=\zeta\left(  v\right)  $ is a strict supersolution of the equation $\left(
\ref{Equ_limit_upper}\right)  $ such that%
\begin{equation}
\max_{x\in\Omega}\left(  u-w\right)  >\max_{x\in\partial\Omega}\left(
u-w\right)  , \label{Thm_comparison_limit_equ_Inq1}%
\end{equation}
and there is $\delta>0$ such that%
\begin{equation}
-\Delta_{\infty,\Lambda}w\geqslant\delta>0.
\label{Thm_comparison_limit_equ_Inq2}%
\end{equation}
Let $\varrho\in C^{2}$ be a function which touches $v$ from below at $x_{0}%
\in\bar{\Omega}$, i.e. $v-\varrho$ has a local minimum at $x_{0}$ and
$v\left(  x_{0}\right)  =\varrho\left(  x_{0}\right)  .$ Then, since $\zeta\in
C^{2}\left(  \left[  0,\infty\right)  \right)  $ is strictly monotone,
$\zeta\circ\varrho\in C^{2}$ also touches $w$ from below. Since $v$ is a
solution of $\left(  \ref{Equ_limit_upper}\right)  $,%
\[
\left\vert \nabla\varrho\right\vert \geqslant\epsilon\text{ and}%
-\Delta_{\infty,\Lambda}\varrho\geqslant0.
\]
So%
\[
-\left(  \zeta^{\prime}\left(  \varrho\right)  \right)  ^{3}\Delta
_{\infty,\Lambda}\varrho\geqslant0,
\]
or%
\begin{align*}
&  -\left(  \zeta^{\prime}\left(  \varrho\right)  \right)  ^{3}\Delta_{\infty
}\varrho-\left(  \zeta^{\prime}\left(  \varrho\right)  \right)  ^{3}\left\vert
\nabla\varrho\right\vert \left\langle \Lambda\left(  x,\left\vert
\nabla\varrho\right\vert \right)  ,\nabla\varrho\right\rangle \\
&  =-\left(  \zeta^{\prime}\left(  \varrho\right)  \right)  ^{3}\Delta
_{\infty}\varrho-\left\vert \nabla\zeta\left(  \varrho\right)  \right\vert
\left\langle \zeta^{\prime}\left(  \varrho\right)  \Lambda\left(  x,\left\vert
\nabla\varrho\right\vert \right)  ,\nabla\zeta\left(  \varrho\right)
\right\rangle \\
&  \geqslant0.
\end{align*}
Straightforward calculations yield%
\[
-\Delta_{\infty}\zeta\left(  \varrho\right)  =-\left(  \zeta^{\prime}\left(
\varrho\right)  \right)  ^{3}\Delta_{\infty}\varrho-\left(  \zeta^{\prime
}\left(  \varrho\right)  \right)  ^{2}\zeta^{\prime\prime}\left(
\varrho\right)  \left\vert \nabla\varrho\right\vert ^{4}%
\]
and%
\begin{align*}
-\Delta_{\infty,\Lambda}\zeta\left(  \varrho\right)   &  =-\Delta_{\infty
}\zeta\left(  \varrho\right)  -\left\vert \nabla\zeta\left(  \varrho\right)
\right\vert \left\langle \Lambda\left(  x,\left\vert \nabla\zeta\left(
\varrho\right)  \right\vert \right)  ,\nabla\zeta\left(  \varrho\right)
\right\rangle \\
&  =-\left(  \zeta^{\prime}\left(  \varrho\right)  \right)  ^{3}\Delta
_{\infty}\varrho-\left(  \zeta^{\prime}\left(  \varrho\right)  \right)
^{2}\zeta^{\prime\prime}\left(  \varrho\right)  \left\vert \nabla
\varrho\right\vert ^{4}\\
&  -\left\vert \nabla\zeta\left(  \varrho\right)  \right\vert \left\langle
\Lambda\left(  x,\left\vert \nabla\zeta\left(  \varrho\right)  \right\vert
\right)  ,\nabla\zeta\left(  \varrho\right)  \right\rangle .
\end{align*}
It follows that%
\begin{align*}
-\Delta_{\infty,\Lambda}\zeta\left(  \varrho\right)   &  \geqslant\left\vert
\nabla\zeta\left(  \varrho\right)  \right\vert \left\langle \zeta^{\prime
}\left(  \varrho\right)  \Lambda\left(  x,\left\vert \nabla\varrho\right\vert
\right)  ,\nabla\zeta\left(  \varrho\right)  \right\rangle -\left\vert
\nabla\zeta\left(  \varrho\right)  \right\vert \left\langle \Lambda\left(
x,\left\vert \nabla\zeta\left(  \varrho\right)  \right\vert \right)
,\nabla\zeta\left(  \varrho\right)  \right\rangle \\
&  -\left(  \zeta^{\prime}\left(  \varrho\right)  \right)  ^{2}\zeta
^{\prime\prime}\left(  \varrho\right)  \left\vert \nabla\varrho\right\vert
^{4}\\
&  =\left\vert \nabla\zeta\left(  \varrho\right)  \right\vert \left\langle
\left\vert \nabla\zeta\left(  \varrho\right)  \right\vert \left(
\frac{\Lambda\left(  x,\left\vert \nabla\varrho\right\vert \right)
}{\left\vert \nabla\varrho\right\vert }-\frac{\Lambda\left(  x,\left\vert
\nabla\zeta\left(  \varrho\right)  \right\vert \right)  }{\left\vert
\nabla\zeta\left(  \varrho\right)  \right\vert }\right)  ,\nabla\zeta\left(
\varrho\right)  \right\rangle \\
&  -\left(  \zeta^{\prime}\left(  \varrho\right)  \right)  ^{2}\zeta
^{\prime\prime}\left(  \varrho\right)  \left\vert \nabla\varrho\right\vert
^{4}\\
&  \geqslant\left\vert \nabla\zeta\left(  \varrho\right)  \right\vert
^{3}\left(  -\left\vert \frac{\Lambda\left(  x,\left\vert \nabla
\varrho\right\vert \right)  }{\left\vert \nabla\varrho\right\vert }%
-\frac{\Lambda\left(  x,\left\vert \nabla\zeta\left(  \varrho\right)
\right\vert \right)  }{\left\vert \nabla\zeta\left(  \varrho\right)
\right\vert }\right\vert -\frac{\zeta^{\prime\prime}\left(  \varrho\right)
}{\zeta^{\prime}\left(  \varrho\right)  }\left\vert \nabla\varrho\right\vert
\right) \\
&  =\left(  \zeta^{\prime}\left(  \varrho\right)  \right)  ^{3}\left\vert
\nabla\varrho\right\vert ^{3}\left(  -\left\vert \Theta\left(  x,\zeta
^{\prime}\left(  \varrho\right)  \left\vert \nabla\varrho\right\vert \right)
-\Theta\left(  x,\left\vert \nabla\varrho\right\vert \right)  \right\vert
-\frac{\zeta^{\prime\prime}\left(  \varrho\right)  }{\zeta^{\prime}\left(
\varrho\right)  }\left\vert \nabla\varrho\right\vert \right)
\end{align*}
Using Lemma \ref{lm_untility_func} (c) and $\left(  \ref{assumption_Lambda}%
\right)  $,\ the terms in the bracket of the rightmost term of the inequality
are estimated as below%
\[
\left\vert \Theta\left(  x,\zeta^{\prime}\left(  \varrho\right)  \left\vert
\nabla\varrho\right\vert \right)  -\Theta\left(  x,\left\vert \nabla
\varrho\right\vert \right)  \right\vert \leqslant\lambda\left(  \zeta^{\prime
}\left(  \varrho\right)  -1\right)  ,
\]%
\[
-\frac{\zeta^{\prime\prime}\left(  \varrho\right)  }{\zeta^{\prime}\left(
\varrho\right)  }\left\vert \nabla\varrho\right\vert =\alpha\left(
\zeta^{\prime}\left(  \varrho\right)  -1\right)  \left\vert \nabla
\varrho\right\vert \geqslant\alpha\epsilon\left(  \zeta^{\prime}\left(
\varrho\right)  -1\right)  .
\]
At this point we fix $\alpha>\lambda/\epsilon$ so that the sign of the bracket
is positive. Then the previous inequality continues%
\begin{align*}
-\Delta_{\infty,\Lambda}\zeta\left(  \varrho\right)   &  \geqslant\left(
\zeta^{\prime}\left(  \varrho\right)  \right)  ^{3}\left\vert \nabla
\varrho\right\vert ^{3}\left(  -\lambda\left(  \zeta^{\prime}\left(
\varrho\right)  -1\right)  +\alpha\epsilon\left(  \zeta^{\prime}\left(
\varrho\right)  -1\right)  \right) \\
&  \geqslant\epsilon^{3}\left(  \zeta^{\prime}\left(  \varrho\right)  \right)
^{3}\left(  -\lambda+\alpha\epsilon\right)  \left(  \zeta^{\prime}\left(
\varrho\right)  -1\right)  .
\end{align*}
By Lemma \ref{lm_untility_func} (a), $\zeta^{\prime}\left(  \varrho\right)
>1,$ $\zeta^{\prime}\left(  \varrho\right)  $ is monotonously decreasing in
$\varrho$, we get%
\[
-\Delta_{\infty,\Lambda}\zeta\left(  \varrho\right)  \geqslant\epsilon
^{3}\left(  -\lambda+\alpha\epsilon\right)  \left(  \zeta^{\prime}\left(
\left\Vert v\right\Vert _{L^{\infty}}\right)  -1\right)  .
\]
Now let%
\[
\delta=\epsilon^{3}\left(  -\lambda+\alpha\epsilon\right)  \left(
\zeta^{\prime}\left(  \left\Vert v\right\Vert _{L^{\infty}}\right)  -1\right)
>0,
\]
then the viscosity inequality $\left(  \ref{Thm_comparison_limit_equ_Inq2}%
\right)  $ is verified. Furthermore, the inequality $\left(
\ref{Thm_comparison_limit_equ_Inq1}\right)  $ can be achieved, by virtue of
Lemma \ref{lm_untility_func} (b) and the assumption $\left(
\ref{Thm_comparison_limit_equ_Inq0}\right)  $, so long as $A$ is close to $1$.
Therefore with $\alpha>\lambda/\epsilon$ and $A>1$ fixed as above, the
function $w=\zeta\left(  v\right)  $ is indeed a strict supersolution of the
equation $\left(  \ref{Equ_limit}\right)  .$

\textbf{Step II}: After the preparation in Step I, the second step is to reach
a contradiction via the use of "Theorem of Sums" or "doubling argument". Since
the proof follows standard approach in the viscosity theory, we omit it.
\end{proof}

\begin{theorem}
\label{thm_uniqueness_limit_equ}Assume that the unit normalization condition
$\left(  \ref{assumption_PHI_1_n}\right)  $ holds. Then there are viscosity
solutions $u^{-},$ $u$ and $u^{+}$ of $\left(  \ref{Equ_limit_lower}\right)
$, $\left(  \ref{Equ_limit}\right)  $ and $\left(  \ref{Equ_limit_upper}%
\right)  $ respectively such that $u^{-},$ $u,$ $u^{+}\in C\left(  \bar
{\Omega}\right)  \cap W^{1,\infty}\left(  \Omega\right)  $ and%
\[
u^{-}\leqslant u\leqslant u^{+}\text{ for all }x\in\Omega.
\]
Moreover, there is $0<\kappa<1$ such that%
\[
\left\vert u^{+}-u^{-}\right\vert \leqslant4\left(  1+\left\vert
\Omega\right\vert \right)  diam\left(  \Omega\right)  \frac{\epsilon}{\kappa
}\text{ for all }x\in\Omega.
\]

\end{theorem}

\begin{proof}
Suppose that $n$ is large so that $p_{n}^{-}>d$. Let $u_{n}^{+}\in
g+W_{0}^{1,\Phi_{n}\left(  \cdot\right)  }$ be the unique weak solution of%
\begin{equation}
-\operatorname{div}\left(  \dfrac{a_{n}\left(  x,\left\vert \nabla
u\right\vert \right)  }{\left\vert \nabla u\right\vert }\nabla u\right)
=a_{n}\left(  x,\epsilon\right)  ,x\in\Omega,
\label{thm_uniqueness_limit_equ_1}%
\end{equation}
$u_{n}^{-}\in g+W_{0}^{1,\Phi_{n}\left(  \cdot\right)  }$ the unique weak
solution of%
\begin{equation}
-\operatorname{div}\left(  \dfrac{a_{n}\left(  x,\left\vert \nabla
u\right\vert \right)  }{\left\vert \nabla u\right\vert }\nabla u\right)
=-a_{n}\left(  x,\epsilon\right)  ,x\in\Omega,
\label{thm_uniqueness_limit_equ_2}%
\end{equation}
and $u_{n}\in g+W_{0}^{1,\Phi_{n}\left(  \cdot\right)  }$ the unique weak
solution of%
\begin{equation}
-\operatorname{div}\left(  \dfrac{a_{n}\left(  x,\left\vert \nabla
u\right\vert \right)  }{\left\vert \nabla u\right\vert }\nabla u\right)
=0,x\in\Omega. \label{thm_uniqueness_limit_equ_3}%
\end{equation}
We have seen that $u_{n}^{+}$, $u_{n}^{-}$ and $u_{n}$ are also continuous
viscosity solutions of their respective equations, and along a subsequence
they converge uniformly, say to $u^{+}$, $u^{-}$ and $u$ respectively.
Furthermore, all of $u_{n}^{+}$, $u_{n}^{-}$ and $u_{n}$ assume the same
Dirichlet boundary data $g$, so $u^{+}$, $u^{-}$ and $u$ all equal to $g$ on
the boundary $\partial\Omega$. Recalling $\left(
\ref{Thm_solution_upper_lower_Inq1}\right)  ,$ there is a constant $c\left(
\Omega,\mu,g\right)  >0$ such that, for all $m>d$ and $n$ for which $p_{n}%
^{-}\geqslant m$,
\begin{equation}
\left\Vert \nabla u_{n}^{-}\right\Vert _{L^{m}},\left\Vert \nabla u_{n}%
^{+}\right\Vert _{L^{m}}\leqslant c\left(  \Omega,\mu,g\right)  \text{.}
\label{thm_uniqueness_limit_equ_4}%
\end{equation}
This in particular implies as earlier that both $u^{+}$ and $u^{-}$ belongs to
$C\left(  \bar{\Omega}\right)  \cap W^{1,\infty}\left(  \Omega\right)  $.
Since $u_{n}^{+}$ is a strict viscosity supersolution of $\left(
\ref{thm_uniqueness_limit_equ_3}\right)  $, and $u_{n}^{-}$ is a strict
viscosity subsolution of $\left(  \ref{thm_uniqueness_limit_equ_3}\right)  $,
we can employ the classical comparison principle to conclude that%
\[
u_{n}^{-}\leqslant u_{n}\leqslant u_{n}^{+}\text{ for all }x\in\Omega.
\]
Hence%
\[
u^{-}\leqslant u\leqslant u^{+}\text{ for all }x\in\Omega.
\]
Since $u_{n}^{+}-u_{n}^{-}\in W_{0}^{1,\Phi_{n}\left(  \cdot\right)  }$, we
can use it as a test function to multiply both $\left(
\ref{thm_uniqueness_limit_equ_1}\right)  $ and $\left(
\ref{thm_uniqueness_limit_equ_2}\right)  $, and then integrate by parts. This
leads us to%
\[
\int_{\Omega}\dfrac{a_{n}\left(  x,\left\vert \nabla u_{n}^{+}\right\vert
\right)  }{\left\vert \nabla u_{n}^{+}\right\vert }\left\langle \nabla
u_{n}^{+},\nabla u_{n}^{+}-\nabla u_{n}^{-}\right\rangle dx-\int_{\Omega}%
a_{n}\left(  x,\epsilon\right)  \left(  u_{n}^{+}-u_{n}^{-}\right)  dx=0,
\]
and%
\[
\int_{\Omega}\dfrac{a_{n}\left(  x,\left\vert \nabla u_{n}^{-}\right\vert
\right)  }{\left\vert \nabla u_{n}^{-}\right\vert }\left\langle \nabla
u_{n}^{-},\nabla u_{n}^{+}-\nabla u_{n}^{-}\right\rangle dx+\int_{\Omega}%
a_{n}\left(  x,\epsilon\right)  \left(  u_{n}^{+}-u_{n}^{-}\right)  dx=0.
\]
Now subtracting the latter from the former and using Lemma
\ref{inequality_main}, we obtain, for some constant $0<\kappa<1$ independent
of $n$, that
\begin{align*}
&  \min\left\{  \kappa^{-p_{n}^{-}},\frac{p_{n}^{-}}{4}\right\}  \int_{\Omega
}\Phi_{n}\left(  x,\kappa\left\vert \nabla u_{n}^{+}-\nabla u_{n}%
^{-}\right\vert \right)  dx\\
&  \leqslant2\int_{\Omega}a_{n}\left(  x,\epsilon\right)  \left(  u_{n}%
^{+}-u_{n}^{-}\right)  dx\\
&  =\frac{2}{\epsilon}\int_{\Omega}\frac{\epsilon\varphi_{n}\left(
x,\epsilon\right)  }{\varphi_{n}\left(  x,1\right)  }\left(  u_{n}^{+}%
-u_{n}^{-}\right)  dx\\
&  \leqslant\frac{2}{\epsilon}p_{n}^{+}\int_{\Omega}\frac{\Phi_{n}\left(
x,\epsilon\right)  }{\varphi_{n}\left(  x,1\right)  }\left(  u_{n}^{+}%
-u_{n}^{-}\right)  dx\\
&  \leqslant\frac{2}{\epsilon}p_{n}^{+}\int_{\Omega}\frac{\Phi_{n}\left(
x,\epsilon\right)  }{\varphi_{n}\left(  x,1\right)  }\left(  u_{n}^{+}%
-u_{n}^{-}\right)  dx.
\end{align*}
Note for large $n$, $\kappa^{-p_{n}^{-}}>p_{n}^{-}/4$. By the unit
normalization condition $\left(  \ref{assumption_PHI_1_n}\right)  $ and
$\left(  \ref{assumption_Phi_n}\right)  $, $\varphi_{n}\left(  x,1\right)
\geqslant p_{n}^{-}$. Therefore%
\[
\int_{\Omega}\Phi_{n}\left(  x,\kappa\left\vert \nabla u_{n}^{+}-\nabla
u_{n}^{-}\right\vert \right)  dx\leqslant\frac{8}{\epsilon}\frac{p_{n}^{+}%
}{p_{n}^{-2}}\int_{\Omega}\Phi_{n}\left(  x,\epsilon\right)  \left(  u_{n}%
^{+}-u_{n}^{-}\right)  dx.
\]
Recalling that $u_{n}^{+}\in C\left(  \bar{\Omega}\right)  $ (resp. $u_{n}%
^{-}\in C\left(  \bar{\Omega}\right)  $) converges uniformly to $u^{+}$ (resp.
$u^{-}$), we see that the sequence $\left\Vert u_{n}^{+}-u_{n}^{-}\right\Vert
_{L^{\infty}}$ is bounded by some constant $C>0.$ It follows that%
\begin{align*}
\int_{\Omega}\Phi_{n}\left(  x,\kappa\left\vert \nabla u_{n}^{+}-\nabla
u_{n}^{-}\right\vert \right)  dx  &  \leqslant\frac{8}{\epsilon}\frac
{p_{n}^{+}}{p_{n}^{-2}}\left\Vert u_{n}^{+}-u_{n}^{-}\right\Vert _{L^{\infty}%
}\int_{\Omega}\Phi_{n}\left(  x,\epsilon\right)  dx\\
&  \leqslant8C\frac{p_{n}^{+}}{p_{n}^{-2}}\left\vert \Omega\right\vert
\epsilon^{p_{n}^{-}-1}.
\end{align*}
Invoking Lemma \ref{lm_modulo_finite}, we obtain%
\[
\min\left\{  \left\Vert \kappa\left\vert \nabla u_{n}^{+}-\nabla u_{n}%
^{-}\right\vert \right\Vert _{L^{\Phi_{n}\left(  \cdot\right)  }}^{p_{n}^{-}%
},\left\Vert \kappa\left\vert \nabla u_{n}^{+}-\nabla u_{n}^{-}\right\vert
\right\Vert _{L^{\Phi_{n}\left(  \cdot\right)  }}^{p_{n}^{+}}\right\}
\leqslant8C\epsilon^{p_{n}^{-}-1}\frac{p_{n}^{+}}{p_{n}^{-2}}\left\vert
\Omega\right\vert .
\]
Hence%
\[
\left\Vert \kappa\left\vert \nabla u_{n}^{+}-\nabla u_{n}^{-}\right\vert
\right\Vert _{L^{\Phi_{n}\left(  \cdot\right)  }}\leqslant\max\left\{  \left(
8C\epsilon^{p_{n}^{-}-1}\frac{p_{n}^{+}}{p_{n}^{-2}}\left\vert \Omega
\right\vert \right)  ^{\frac{1}{p_{n}^{-}}},\left(  8C\epsilon^{p_{n}^{-}%
-1}\frac{p_{n}^{+}}{p_{n}^{-2}}\left\vert \Omega\right\vert \right)
^{\frac{1}{p_{n}^{+}}}\right\}
\]
Using the embedding inequality $\left(  \ref{thm_bounded_inW1d_all_g_Inq1}%
\right)  ,$ we have, for all $m\geqslant1$ and $n$ for which $p_{n}%
^{-}\geqslant m$, that
\[
\left\Vert \kappa\left\vert \nabla u_{n}^{+}-\nabla u_{n}^{-}\right\vert
\right\Vert _{L^{m}}\leqslant4\left(  1+\left\vert \Omega\right\vert \right)
\left\Vert \kappa\left\vert \nabla u_{n}^{+}-\nabla u_{n}^{-}\right\vert
\right\Vert _{L^{\Phi_{n}\left(  \cdot\right)  }}.
\]
In view of $\left(  \ref{thm_uniqueness_limit_equ_4}\right)  $, we may assume
up to a subsequence that%
\[
\left\{  \nabla u_{n}^{+}-\nabla u_{n}^{-}:n\text{ satisfies }p_{n}%
^{-}\geqslant m\right\}
\]
is weakly convergent in $L^{m}$. It is readily seen that the weak limit is
$\nabla u^{+}-\nabla u^{-}$. So\ sending $n\rightarrow\infty$ yields%
\[
\left\Vert \kappa\left\vert \nabla u^{+}-\nabla u^{-}\right\vert \right\Vert
_{L^{m}}\leqslant4\left(  1+\left\vert \Omega\right\vert \right)  \epsilon.
\]
Lastly sending $m\rightarrow\infty$ gives us%
\[
\left\Vert \nabla u^{+}-\nabla u^{-}\right\Vert _{L^{\infty}}\leqslant4\left(
1+\left\vert \Omega\right\vert \right)  \frac{\epsilon}{\kappa}.
\]
The conclusion then follows.
\end{proof}

\begin{theorem}
\label{thm_limit_equ_unique_main}Let $\epsilon>0$, $u^{-}$, $u$ and $^{+}$ be
respectively solutions of $\left(  \ref{Equ_limit_lower}\right)  ,$ $\left(
\ref{Equ_limit}\right)  $ and $\left(  \ref{Equ_limit_upper}\right)  $
obtained in Theorem \ref{thm_uniqueness_limit_equ}. Then for an arbitrary
viscosity solution $u_{\infty}\in C\left(  \bar{\Omega}\right)  $\ of $\left(
\ref{Equ_limit}\right)  $, it holds that%
\[
u^{-}\leqslant u_{\infty}\leqslant u^{+}\text{ for all }x\in\Omega.
\]
In particular, $\left(  \ref{Equ_limit}\right)  $ has a unique continuous
viscosity solution which equals $u$.
\end{theorem}

\begin{proof}
Since $u^{-},$ $u_{\infty},$ $u^{+}\in C\left(  \bar{\Omega}\right)  $, they
all equal $g$ on the boundary $\partial\Omega$. By comparison, i.e. Theorem
\ref{Thm_comparison_limit_equ},%
\[
u^{-}\leqslant u_{\infty}\leqslant u^{+}\text{ for all }x\in\Omega.
\]
The same inequality holds also for $u$ by Theorem
\ref{thm_uniqueness_limit_equ}. Therefore using Theorem
\ref{Thm_comparison_limit_equ} again, we get%
\[
\left\vert u-u_{\infty}\right\vert \leqslant4\left(  1+\left\vert
\Omega\right\vert \right)  diam\left(  \Omega\right)  \frac{\epsilon}{\kappa
}\text{ for all }x\in\Omega.
\]
Since $\epsilon>0$ is arbitrary in our construction, this shows that $u$ is
the unique viscosity solution of $\left(  \ref{Equ_limit}\right)  .$
\end{proof}

\section{\label{sec_Poincare}Poincare's Inequality for Discontinuous
$\Phi$}

We already seen that Poincare's inequality, i.e. Theorem
\ref{thm_poincare_inequality}, holds for generalized $\Phi$-functions under
local continuity condition. We shall show that, in cases where the local
continuity condition $\left(  \ref{PHI_condition_A1}\right)  $ fails, a
version of Poincare's inequality still holds, however, under the following
jump condition.

\begin{definition}
\label{def_p_jump}Let $\Phi:$ $\Omega\times\left[  0,\infty\right)
\longmapsto\mathbb{R}$ be a generalized $\Phi$-function. We write for a set
$A$,%
\begin{equation}
p_{\Phi,A}^{-}=\inf_{y\in A\cap\Omega\text{, }s>0}\frac{s\varphi\left(
y,s\right)  }{\Phi\left(  y,s\right)  }\text{, }p_{\Phi,A}^{+}=\sup_{y\in
A\cap\Omega\text{, }s>0}\frac{s\varphi\left(  y,s\right)  }{\Phi\left(
y,s\right)  }. \label{def_p_jump_inq1}%
\end{equation}
The subscript $A$ is omitted if $A=\Omega$. $\Phi$ satisfies the jump
condition if there is a $\delta>0$ such that for each $x\in\Omega,$ either%
\[
p_{\Phi,B_{\delta}\left(  x\right)  }^{-}\geqslant d
\]
or%
\[
p_{\Phi,B_{\delta}\left(  x\right)  }^{-}<d\text{, }p_{\Phi,B_{\delta}\left(
x\right)  }^{+}\leqslant p_{\Phi,B_{\delta}\left(  x\right)  }^{-\ast},
\]
hereafter, for an arbitrary $0<p<d$,
\[
p^{\ast}=\dfrac{dp}{d-p}.
\]

\end{definition}

Note that by the definition of the generalized $\Phi$-function, $1<p_{\Phi
}^{-}\leqslant p_{\Phi,A}^{-}\leqslant p_{\Phi,A}^{+}\leqslant p_{\Phi}%
^{+}<\infty$.

\begin{lemma}
\label{Lm_orlicz_embedding}Let $\Omega\subset\mathbb{R}^{d}$ be bounded.
$\Phi$, $\Psi:$ $\Omega\times\left[  0,\infty\right)  \longmapsto\mathbb{R}$
are generalized $\Phi$-functions satisfying the normalization condition
$\left(  \ref{assumption_PHI_1}\right)  $ with constant $\beta\in\left(
0,1\right]  $. If $p_{\Phi}^{+}\leqslant p_{\Psi}^{-}$, then $L^{\Psi\left(
\cdot\right)  }\left(  \Omega\right)  \hookrightarrow L^{\Phi\left(
\cdot\right)  }\left(  \Omega\right)  $, i.e. there is a constant $c\left(
\Omega,d,p_{\Phi}^{+},p_{\Psi}^{-},\beta\right)  $ such that for $u\in
L^{\Psi\left(  \cdot\right)  }\left(  \Omega\right)  $,%
\[
\left\Vert u\right\Vert _{L^{\Phi\left(  \cdot\right)  }}\leqslant c\left(
\Omega,d,p_{\Phi}^{+},p_{\Psi}^{-},\beta\right)  \left\Vert u\right\Vert
_{L^{\Psi\left(  \cdot\right)  }}.
\]

\end{lemma}

\begin{proof}
It suffices to show that $L^{\Psi\left(  \cdot\right)  }\left(  \Omega\right)
\hookrightarrow L^{p_{\Psi}^{-}}\left(  \Omega\right)  $ and $L^{p_{\Phi}^{+}%
}\left(  \Omega\right)  \hookrightarrow L^{\Phi\left(  \cdot\right)  }\left(
\Omega\right)  $.

\textbf{Step I}: We prove that $L^{p_{\Phi}^{+}}\left(  \Omega\right)
\hookrightarrow L^{\Phi\left(  \cdot\right)  }\left(  \Omega\right)  .$ To see
this, we note that by Lemma \ref{lemma_delta2_condition}, the mapping
$s\rightarrow\Phi\left(  x,s\right)  /s^{p_{\Phi}^{+}}$ is decreasing and the
normalization condition is equivalent to $\left(  \ref{PHI_normalization_x_}%
\right)  .$ Then%
\[
\frac{\Phi\left(  x,s\right)  }{s^{p_{\Phi}^{+}}}\leqslant\frac{\Phi\left(
x,\beta\right)  }{\beta^{p_{\Phi}^{+}}}\leqslant\frac{1}{\beta^{p_{\Phi}^{+}}%
}\text{ for }s\geqslant\beta\text{, }x\in\Omega.
\]
It follows that%
\[
\Phi\left(  x,s\right)  \leqslant\frac{1}{\beta^{p_{\Phi}^{+}}}s^{p_{\Phi}%
^{+}}\text{ for }s\geqslant\beta\text{, }x\in\Omega.
\]
Since $\Phi\left(  x,s\right)  $ is increasing in $s$, so for all
$s\geqslant0$, $x\in\Omega,$
\[
\Phi\left(  x,s\right)  \leqslant\frac{1}{\beta^{p_{\Phi}^{+}}}s^{p_{\Phi}%
^{+}}+\Phi\left(  x,\beta\right)  \leqslant\frac{1}{\beta^{p_{\Phi}^{+}}%
}s^{p_{\Phi}^{+}}+1.
\]
Let $K=1+2\left\vert \Omega\right\vert $. Then%
\[
\Phi\left(  x,\frac{s}{K}\right)  \leqslant\frac{1}{K^{p_{\Phi}^{-}}}%
\Phi\left(  x,s\right)  \leqslant\frac{1}{K}\Phi\left(  x,s\right)
\leqslant\frac{1}{\beta^{p_{\Phi}^{+}}}s^{p_{\Phi}^{+}}+\frac{1}{2\left\vert
\Omega\right\vert },\text{ for }s\geqslant0\text{, }x\in\Omega.
\]
Let $u\in L^{p_{\Phi}^{+}}$ and
\[
\lambda>\beta^{-1}2^{\frac{1}{p_{\Phi}^{+}}}\left\Vert u\right\Vert
_{L^{p_{\Phi}^{+}}}.
\]
Setting $s=\left\vert u\right\vert /\lambda$ in the previous inequality and
integrating over $\Omega$ gives%
\[
\varrho_{\Phi\left(  \cdot\right)  }\left(  \frac{\left\vert u\right\vert
}{K\lambda}\right)  \leqslant\frac{1}{\beta^{p_{\Phi}^{+}}}\frac{\left\Vert
u\right\Vert _{L^{p_{\Phi}^{+}}}^{p_{\Phi}^{+}}}{\lambda^{p_{\Phi}^{+}}}%
+\frac{1}{2}\leqslant1.
\]
Hence%
\[
\left\Vert u\right\Vert _{\Phi\left(  \cdot\right)  }\leqslant K\lambda.
\]
Letting $\lambda\downarrow\beta^{-1}2^{\frac{1}{p_{\Phi}^{+}}}\left\Vert
u\right\Vert _{L^{p_{\Phi}^{+}}}$ yields that%
\[
\left\Vert u\right\Vert _{L^{\Phi\left(  \cdot\right)  }}\leqslant\left(
1+2\left\vert \Omega\right\vert \right)  \beta^{-1}2^{\frac{1}{p_{\Phi}^{+}}%
}\left\Vert u\right\Vert _{L^{p_{\Phi}^{+}}}.
\]

\textbf{Step II}: We prove that $L^{\Psi\left(  \cdot\right)  }\left(
\Omega\right)  \hookrightarrow L^{p_{\Psi}^{-}}\left(  \Omega\right)  .$
Indeed, similar to Step I the normalization condition $\left(
\ref{assumption_PHI_1}\right)  $ gives%
\[
\beta^{p_{\Phi}^{+}}\leqslant\Phi\left(  x,1\right)  \leqslant\beta^{-p_{\Phi
}^{+}}\text{ and }\beta^{p_{\Psi}^{+}}\leqslant\Psi\left(  x,1\right)
\leqslant\beta^{-p_{\Psi}^{+}}.
\]
For $s\geqslant1,$%
\[
\Phi\left(  x,s\right)  \leqslant s^{p_{\Phi}^{+}}\Phi\left(  x,1\right)
\]
and since $p_{\Phi}^{+}\leqslant p_{\Psi}^{-}$,%
\[
\Psi\left(  x,s\right)  \geqslant s^{p_{\Psi}^{-}}\Psi\left(  x,1\right)
\geqslant s^{p_{\Phi}^{+}}\Psi\left(  x,1\right)  .
\]
It follows that%
\[
\Phi\left(  x,s\right)  \leqslant\frac{\Phi\left(  x,1\right)  }{\Psi\left(
x,1\right)  }\Psi\left(  x,s\right)  \leqslant\beta^{-p_{\Phi}^{+}-p_{\Psi
}^{+}}\Psi\left(  x,s\right)  \text{ for }s\geqslant1.
\]
Hence for $s\geqslant0$, $x\in\Omega,$%
\[
\Phi\left(  x,s\right)  \leqslant\Phi\left(  x,1\right)  +\beta^{-p_{\Phi}%
^{+}-p_{\Psi}^{+}}\Psi\left(  x,s\right)  \leqslant\beta^{-p_{\Phi}^{+}}%
+\beta^{-p_{\Phi}^{+}-p_{\Psi}^{+}}\Psi\left(  x,s\right)  .
\]
For any $K\geqslant1,$%
\[
\Phi\left(  x,\frac{s}{K}\right)  \leqslant\frac{1}{K}\Phi\left(  x,s\right)
\leqslant\frac{1}{K}\left[  \beta^{-p_{\Phi}^{+}}+\beta^{-p_{\Phi}^{+}%
-p_{\Psi}^{+}}\Psi\left(  x,s\right)  \right]  .
\]
Let $u\in L^{\Psi}$ and $\lambda>\left\Vert u\right\Vert _{L^{\Psi\left(
\cdot\right)  }}.$ Then%
\[
\varrho_{\Phi\left(  \cdot\right)  }\left(  \frac{\left\vert u\right\vert
}{K\lambda}\right)  \leqslant\frac{1}{K}\left[  \beta^{-p_{\Phi}^{+}%
}\left\vert \Omega\right\vert +\beta^{-p_{\Phi}^{+}-p_{\Psi}^{+}}\varrho
_{\Psi\left(  \cdot\right)  }\left(  \frac{\left\vert u\right\vert }{\lambda
}\right)  \right]  .
\]
We may also increase $K$ (e.g. $K=1+2\beta^{-p_{\Phi}^{+}}\left\vert
\Omega\right\vert +2\beta^{-p_{\Phi}^{+}-p_{\Psi}^{+}}$) so that%
\[
\varrho_{\Phi\left(  \cdot\right)  }\left(  \frac{\left\vert u\right\vert
}{K\lambda}\right)  \leqslant\frac{1}{2}\left[  1+\varrho_{\Psi\left(
\cdot\right)  }\left(  \frac{\left\vert u\right\vert }{\lambda}\right)
\right]  \leqslant1.
\]
Therefore
\[
\left\Vert u\right\Vert _{L^{\Phi\left(  \cdot\right)  }}\leqslant K\lambda.
\]
Letting $\lambda\downarrow\left\Vert u\right\Vert _{L^{\Psi\left(
\cdot\right)  }}$ gives the desired conclusion.
\end{proof}

\begin{theorem}
[{\cite[Theorem 6.10]{harjulehto2016generalized}}]\label{Thm_extension_W1phi0}%
Let $\Omega\subset\mathbb{R}^{d}$ and $\Phi:$ $\mathbb{R}^{d}\times\left[
0,\infty\right)  \longmapsto\mathbb{R}$ be a generalized $\Phi$-function$.$
Denote by $1_{\Omega}$ the indicator function on $\Omega$. Then the extension
operator%
\[
\mathcal{E}:u\in W_{0}^{1,\Phi\left(  \cdot\right)  }\left(  \Omega\right)
\rightarrow\tilde{u}=1_{\Omega}\cdot u
\]
defines an isometry from $W_{0}^{1,\Phi\left(  \cdot\right)  }\left(
\Omega\right)  $ to $W^{1,\Phi\left(  \cdot\right)  }\left(  \mathbb{R}%
^{d}\right)  $.
\end{theorem}

Next we prove Poincare's inequality for \textit{discontinuous} generalized
$\Phi$-functions, which extends early result on variable exponent Lebesgue
spaces (ref. \cite[Theorem 8.2.18]{diening2011lebesgue} or \cite[Theorem
4.3]{harjulehto2006dirichlet}).

\begin{theorem}
[$\Phi\left(  \cdot\right)  $-Poincare's inequality]%
\label{Thm_poincare_discontinuous}Let $\Omega\subset\mathbb{R}^{d}$ be
bounded. $\Phi:$ $\Omega\times\left[  0,\infty\right)  \longmapsto\mathbb{R}$
is a generalized $\Phi$-function satisfying the normalization condition
$\left(  \ref{assumption_PHI_1}\right)  $ with constant $\beta\in\left(
0,1\right]  $. If $\Phi$ satisfies the jump condition with constant $\delta
>0$, then there is a constant $c\left(  \Omega,d,\delta,p_{\Phi}^{+},p_{\Phi
}^{-},\beta\right)  $ such that for $u\in W_{0}^{1,\Phi\left(  \cdot\right)
}\left(  \Omega\right)  $,%
\[
\left\Vert u\right\Vert _{L^{\Phi\left(  \cdot\right)  }}\leqslant c\left(
\Omega,d,\delta,p_{\Phi}^{+},p_{\Phi}^{-},\beta\right)  \left\Vert \nabla
u\right\Vert _{L^{\Phi\left(  \cdot\right)  }}.
\]

\end{theorem}

\begin{proof}
In view of the jump condition and the boundedness of $\Omega$, there are a
finite number of balls with radius $\delta>0,$ say $B_{1},...,B_{J},$ such
that%
\[
\Omega\subset\bigcup_{j=1}^{J}B_{j},
\]
and on each ball $B_{j}$ either $p_{\Phi,B_{j}}^{-}\geqslant d$ or
$p_{\Phi,B_{j}}^{+}\leqslant p_{\Phi,B_{j}}^{-\ast}.$ Fix $x_{0}\in\Omega$.
Then%
\[
\tilde{\Phi}\left(  x,s\right)  =\left\{
\begin{array}
[c]{ll}%
\Phi\left(  x,s\right)  , & x\in\Omega,\text{ }s\geqslant0;\\
\Phi\left(  x_{0},s\right)  , & x\in\Omega^{c},\text{ }s\geqslant0,
\end{array}
\right.
\]
which defines an extension of $\Phi$ to all of $\mathbb{R}^{d}\times\left[
0,\infty\right)  \longmapsto\mathbb{R}$, is a generalized $\Phi$-function
satisfying the normalization condition $\left(  \ref{assumption_PHI_1}\right)
$ with constant $\beta$. Denote by $\tilde{u}$ the zero extension of $u$ as in
Theorem \ref{Thm_extension_W1phi0}, then $\left\Vert u\right\Vert
_{W_{0}^{1,\Phi\left(  \cdot\right)  }\left(  \Omega\right)  }=\left\Vert
\tilde{u}\right\Vert _{W^{1,\tilde{\Phi}\left(  \cdot\right)  }\left(
\mathbb{R}^{d}\right)  }.$ Let%
\[
\Phi_{j}:\left(  x,s\right)  \in B_{j}\times\left[  0,\infty\right)
\rightarrow\left\{
\begin{array}
[c]{ll}%
\Phi\left(  x,s\right)  , & x\in B_{j}\cap\Omega,\text{ }s\geqslant0;\\
s^{p_{\Phi,B_{j}}^{-}}, & x\in B_{j}\cap\Omega^{c},\text{ }s\geqslant0.
\end{array}
\right.
\]
Note that $p_{\Phi,B_{j}}^{-}$ is defined as in $\left(  \ref{def_p_jump_inq1}%
\right)  $, the infimum in $p_{\Phi,B_{j}}^{-}$ is taken over $B_{j}\cap
\Omega$, not $B_{j}.$ Then%
\begin{align*}
\left\Vert u\right\Vert _{L^{\Phi\left(  \cdot\right)  }\left(  \Omega\right)
}  &  =\left\Vert \tilde{u}\right\Vert _{L^{\tilde{\Phi}\left(  \cdot\right)
}\left(  \mathbb{R}^{d}\right)  }\\
&  \leqslant\sum_{j=1}^{J}\left\Vert \tilde{u}1_{B_{j}\cap\Omega}\right\Vert
_{L^{\tilde{\Phi}\left(  \cdot\right)  }\left(  \mathbb{R}^{d}\right)  }%
=\sum_{j=1}^{J}\left\Vert \tilde{u}1_{B_{j}\cap\Omega}\right\Vert
_{L^{\Phi_{j}\left(  \cdot\right)  }\left(  B_{j}\right)  }\\
&  \leqslant\sum_{j=1}^{J}\left\Vert \tilde{u}\right\Vert _{L^{\Phi_{j}\left(
\cdot\right)  }\left(  B_{j}\right)  }\\
&  \leqslant\sum_{j=1}^{J}\left\Vert \tilde{u}-\tilde{u}_{B_{j}}\right\Vert
_{L^{\Phi_{j}\left(  \cdot\right)  }\left(  B_{j}\right)  }+\sum_{j=1}%
^{J}\left\vert \tilde{u}_{B_{j}}\right\vert \left\Vert 1_{B_{j}}\right\Vert
_{L^{\Phi_{j}\left(  \cdot\right)  }\left(  B_{j}\right)  },
\end{align*}
where%
\[
\tilde{u}_{B_{j}}=\frac{1}{\left\vert B_{j}\right\vert }\int_{B_{j}}\tilde
{u}\left(  x\right)  dx.
\]
Now we estimate the two sums in the above inequality. For each term in the
first sum, if $p_{\Phi,B_{j}}^{+}\leqslant p_{\Phi,B_{j}}^{-\ast},$ then we
apply Lemma \ref{Lm_orlicz_embedding} and the generalized Poincare's
inequality (e.g. \cite[Theorem 8.2.13]{diening2011lebesgue}) to get%
\begin{align*}
\left\Vert \tilde{u}-\tilde{u}_{B_{j}}\right\Vert _{L^{\Phi_{j}\left(
\cdot\right)  }\left(  B_{j}\right)  }  &  \leqslant c\left(  \Omega
,d,\delta,p_{\Phi}^{+},p_{\Phi}^{-}\right)  \left\Vert \tilde{u}-\tilde
{u}_{B_{j}}\right\Vert _{L^{p_{\Phi_{j},B_{j}}^{+}}\left(  B_{j}\right)  }\\
&  \leqslant c\left(  \Omega,d,\delta,p_{\Phi}^{+},p_{\Phi}^{-}\right)
\left\Vert \tilde{u}-\tilde{u}_{B_{j}}\right\Vert _{L^{p_{\Phi_{j},B_{j}%
}^{-\ast}}\left(  B_{j}\right)  }\\
&  \leqslant c\left(  \Omega,d,\delta,p_{\Phi}^{+},p_{\Phi}^{-}\right)
\left\Vert \nabla\tilde{u}\right\Vert _{L^{p_{\Phi_{j},B_{j}}^{-}}\left(
B_{j}\right)  }\\
&  \leqslant c\left(  \Omega,d,\delta,p_{\Phi}^{+},p_{\Phi}^{-}\right)
\left\Vert \nabla u\right\Vert _{L^{\Phi\left(  \cdot\right)  }\left(
\Omega\right)  }.
\end{align*}
Otherwise if $p_{\Phi,B_{j}}^{-}\geqslant d,$ then we may find $q<d$ so that
$q^{\ast}=p_{\Phi_{j},B_{j}}^{+}$. Now $q<p_{\Phi,B_{j}}^{-}.$ By the same
arguments, the above inequalities continue to hold except $\left\Vert
\tilde{u}-\tilde{u}_{B_{j}}\right\Vert _{L^{p_{\Phi_{j},B_{j}}^{-\ast}}\left(
B_{j}\right)  }$ in the second being replaced with $\left\Vert \nabla\tilde
{u}\right\Vert _{L^{q}\left(  B_{j}\right)  }$. For the second sum, applying
Poincare's inequality again we have%
\begin{align*}
\left\vert \tilde{u}_{B_{j}}\right\vert  &  \leqslant\frac{1}{\left\vert
B_{j}\right\vert }\int_{\Omega}\left\vert u\right\vert dx\leqslant c\left(
\Omega,d,\delta\right)  \int_{\Omega}\left\vert \nabla u\right\vert
dx\leqslant c\left(  \Omega,d,\delta\right)  \left\Vert \nabla u\right\Vert
_{L^{p_{\Phi}^{-}}\left(  \Omega\right)  }\\
&  \leqslant c\left(  \Omega,d,\delta\right)  \left\Vert \nabla u\right\Vert
_{L^{\Phi\left(  \cdot\right)  }\left(  \Omega\right)  }.
\end{align*}
By Lemma \ref{lm_modulo_finite},
\begin{align*}
\left\Vert 1_{B_{j}}\right\Vert _{L^{\Phi_{i}\left(  \cdot\right)  }\left(
B_{j}\right)  }  &  \leqslant\left\Vert 1_{\Omega}\right\Vert _{L^{\Phi\left(
\cdot\right)  }\left(  \Omega\right)  }\\
&  \leqslant\max\left\{  \left(  \int_{\Omega}\Phi\left(  x,1\right)
dx\right)  ^{\frac{1}{p_{\Phi}^{-}}},\left(  \int_{\Omega}\Phi\left(
x,1\right)  dx\right)  ^{\frac{1}{p_{\Phi}^{+}}}\right\} \\
&  \leqslant c\left(  \Omega,d,p_{\Phi}^{+},p_{\Phi}^{-},\beta\right)  .
\end{align*}
In the last inequality, we have used the normalization condition $\left(
\ref{assumption_PHI_1}\right)  $. Putting these estimates together we obtain%
\begin{align*}
\left\Vert u\right\Vert _{L^{\Phi\left(  \cdot\right)  }\left(  \Omega\right)
}  &  \leqslant\sum_{j=1}^{J}\left\Vert \tilde{u}-\tilde{u}_{B_{j}}\right\Vert
_{L^{\Phi_{j}\left(  \cdot\right)  }\left(  B_{j}\right)  }+\sum_{j=1}%
^{J}\left\vert \tilde{u}_{B_{j}}\right\vert \left\Vert 1_{B_{j}}\right\Vert
_{L^{\Phi_{j}\left(  \cdot\right)  }\left(  B_{j}\right)  }\\
&  \leqslant\sum_{j=1}^{J}c\left(  \Omega,d,\delta,p_{\Phi}^{+},p_{\Phi}%
^{-}\right)  \left\Vert \nabla u\right\Vert _{L^{\Phi\left(  \cdot\right)
}\left(  \Omega\right)  }+\sum_{j=1}^{J}c\left(  \Omega,d,\delta,p_{\Phi}%
^{+},p_{\Phi}^{-},\beta\right)  \left\Vert \nabla u\right\Vert _{L^{\Phi
\left(  \cdot\right)  }\left(  \Omega\right)  }\\
&  \leqslant c\left(  \Omega,d,\delta,p_{\Phi}^{+},p_{\Phi}^{-},\beta\right)
\left\Vert \nabla u\right\Vert _{L^{\Phi\left(  \cdot\right)  }\left(
\Omega\right)  }.
\end{align*}

\end{proof}

\section{Unbounded coefficients in a convex subdomain}

The main purpose of this section is to demonstrate the uses of the results
that have been established in previous sections. For this we consider the
existence of a continuous solution to the subdomain problem $\left(
\ref{general_quasilinear_n_sub}\right)  $ which is summarized in Theorem
\ref{thm_convergence_sub}. This extends the main result for the $p\left(
x\right)  $-Harmonic equation \cite{manfredi2009p} to our general framework.
We note that, $p\left(  x\right)  $-Harmonic equation is a special case of the
general equation in this paper, although they are similar in the approximation
approach taken to deal with the unboundedness in the subdomain, they differ in
technical details, in particular the proof of the existence and uniqueness of
the general equation requires new technical ingredients, moreover
Poincare's Inequality for Discontinuous $\Phi$ provided in previous
section is crucial for the existence of a continuous solution.

With Poincare's Inequality for Discontinuous $\Phi$ at hand, we are then
in a position to sudy to the subudomain problem. Let $\Omega,$ $\Omega
_{\infty}$ be bounded $C^{1}$ domains of $\mathbb{R}^{d}$ and $\Omega_{\infty
}\subset\Omega$. Throughout this section, we assume that%
\begin{equation}
\Omega_{\infty}\text{ is convex.} \label{assumption_convexity_bdry}%
\end{equation}

A remark on the convexity is in order. The convexity of $\Omega_{\infty}$
plays two roles here. If the boundary data is given on all of $\partial
\Omega_{\infty}$, then the $L^{\infty}$ norm of the gradient of the AMLE, say
$\left\Vert \nabla u_{\infty}\right\Vert _{\infty}$, equals the Lipschitz
constant of boundary data (e.g. \cite[Proposition 3.1]{lindqvist2016notes}).
However, if the boundary data is only given on part of $\partial\Omega
_{\infty}$, then $\left\Vert \nabla u_{\infty}\right\Vert _{\infty}$ might be
greater than the Lipschitz constant of boundary data. So assuming that
$\Omega_{\infty}$ is convex excludes this possibility. This is used in the
proof that if the Lipschitz constant of the boundary data is strictly greater
than $1$, then the energy sequence in question is unbounded. Additionally the
convexity $\left(  \ref{assumption_convexity_bdry}\right)  $ is a sufficient
condition which ensures that the extramal solution satisfies a mixed boundary
problem in $\Omega_{\infty}$.

We consider the extremal problem%

\begin{equation}
\left\{
\begin{array}
[c]{ll}%
-\operatorname{div}\left(  \dfrac{a_{\infty}\left(  x,\left\vert \nabla
u\left(  x\right)  \right\vert \right)  }{\left\vert \nabla u\left(  x\right)
\right\vert }\nabla u\left(  x\right)  \right)  =0, & x\in\Omega;\\
u\left(  x\right)  =g\left(  x\right)  , & x\in\partial\Omega,
\end{array}
\right.  \label{general_quasilinear_sub}%
\end{equation}
where $a_{\infty}\left(  x,s\right)  :\Omega\times\left[  0,\infty\right)
\rightarrow\mathbb{R\cup}\left\{  +\infty\right\}  $ is defined piecewise,
\[
a_{\infty}\left(  x,s\right)  =\left\{
\begin{array}
[c]{ll}%
0, & s<1,x\in\Omega_{\infty};\\
1, & s=1,x\in\Omega_{\infty};\\
\infty, & s>1,x\in\Omega_{\infty}.
\end{array}
\right.
\]
and%
\[
a_{\infty}\left(  x,s\right)  =a\left(  x,s\right)  \text{ for }x\in
\Omega\backslash\Omega_{\infty},\text{ }s\geqslant0,
\]
with $a\left(  x,s\right)  $ being given as in $\left(  \ref{eq:def_a_xs}%
\right)  $. A natural approach to the solution of $\left(
\ref{general_quasilinear_sub}\right)  $ is approximation. To be precise, we
replace $a_{\infty}$ in $\Omega_{\infty}\times\left[  0,\infty\right)  $ with
some properly chosen sequence $a_{n}$ (ref. $\left(
\ref{general_quasilinear_n_sub}\right)  $) and investigate the limit as
$n\rightarrow\infty$. This is what we will do in the following. In this sense,
by a solution of $\left(  \ref{general_quasilinear_sub}\right)  $ we mean a
solution of the limiting equation of $\left(  \ref{general_quasilinear_n_sub}%
\right)  $, and the uniqueness is also meant to be in reference to the
corresponding limiting equation. So there is a caveat that the solution of
$\left(  \ref{general_quasilinear_sub}\right)  $ depends on the choice of the
approximation sequence.

Now we turn to the approximation problem%
\begin{equation}
\left\{
\begin{array}
[c]{ll}%
-\operatorname{div}\left(  \dfrac{\tilde{a}_{n}\left(  x,\left\vert \nabla
u\left(  x\right)  \right\vert \right)  }{\left\vert \nabla u\left(  x\right)
\right\vert }\nabla u\left(  x\right)  \right)  =0, & x\in\Omega;\\
u\left(  x\right)  =g\left(  x\right)  , & x\in\partial\Omega,
\end{array}
\right.  \label{general_quasilinear_n_sub}%
\end{equation}
where%
\begin{equation}
\tilde{a}_{n}\left(  x,s\right)  =\left\{
\begin{array}
[c]{ll}%
a_{n}\left(  x,s\right)  , & x\in\Omega_{\infty};\\
a\left(  x,s\right)  , & x\in\Omega\backslash\Omega_{\infty}.
\end{array}
\right.  \label{def_a_n_sub}%
\end{equation}
As before, we write%
\[
\Phi\left(  x,s\right)  =\int_{0}^{s}\varphi\left(  x,t\right)  dt,\text{
}a\left(  x,s\right)  =\frac{\varphi\left(  x,s\right)  }{\varphi\left(
x,1\right)  }\text{ for all }x\in\Omega\backslash\Omega_{\infty}\text{ and
}s\geqslant0,
\]
and assume that $\varphi\left(  x,s\right)  \in C\left(  \left(
\Omega\backslash\Omega_{\infty}\right)  \times\left[  0,\infty\right)
\right)  \cap C^{1}\left(  \left(  \Omega\backslash\Omega_{\infty}\right)
\times\left(  0,\infty\right)  \right)  $ such that there are cosntants
$p^{-}<p^{+}$ for which
\begin{equation}
1\leqslant p^{-}-1\leqslant\frac{s\partial_{s}\varphi\left(  x,s\right)
}{\varphi\left(  x,s\right)  }\leqslant p^{+}-1\text{ for }x\in\Omega
\backslash\Omega_{\infty},\text{ }s\geqslant0. \label{assumption_phi_x_sub}%
\end{equation}
Comparing with $\left(  \ref{assumption_phi_x}\right)  $, $\left(
\ref{assumption_phi_x_sub}\right)  $ requires a little more than before, i.e.
$p^{-}\geqslant2$. This helps avoid the difficulty coming from the singularity
of the differential operator when $1<p^{-}<2$. On $\Omega_{\infty}$, we write%

\[
\Phi_{n}\left(  x,s\right)  =\int_{0}^{s}\varphi_{n}\left(  x,t\right)
dt,\text{ }a_{n}\left(  x,s\right)  =\frac{\varphi_{n}\left(  x,s\right)
}{\varphi_{n}\left(  x,1\right)  }\text{ for all }x\in\Omega_{\infty}\text{
and }s\geqslant0,
\]
and assume that $\varphi_{n}\left(  x,s\right)  \in C\left(  \left(
\Omega_{\infty}\right)  \times\left[  0,\infty\right)  \right)  \cap
C^{1}\left(  \left(  \Omega_{\infty}\right)  \times\left(  0,\infty\right)
\right)  $ satisfies $\left(  \ref{assumption_limit_1}\right)  $, $\left(
\ref{assumption_Lambda_continuous}\right)  $, $\left(
\ref{assumption_limit_0}\right)  $, $\left(  \ref{assumption_limit_3}\right)
$ and the associated Harnack type inequality $\left(  \ref{assumption_c_n}%
\right)  .$

Finally we write%
\[
\tilde{\varphi}_{n}\left(  x,s\right)  =\left\{
\begin{array}
[c]{ll}%
\varphi_{n}\left(  x,s\right)  , & \text{for }x\in\Omega_{\infty}\text{ and
}s\geqslant0;\\
\varphi\left(  x,s\right)  , & \text{for }x\in\Omega\backslash\Omega_{\infty
}\text{ and }s\geqslant0.
\end{array}
\right.
\]%
\[
\tilde{\Phi}_{n}\left(  x,s\right)  =\left\{
\begin{array}
[c]{ll}%
\Phi_{n}\left(  x,s\right)  , & \text{for }x\in\Omega_{\infty}\text{ and
}s\geqslant0;\\
\Phi\left(  x,s\right)  , & \text{for }x\in\Omega\backslash\Omega_{\infty
}\text{ and }s\geqslant0.
\end{array}
\right.
\]

Some proof in the section will be sketchy, since the arguments are similar to
previous sections. Analogous to Theorem \ref{Thm_gamma_convergence}, we have

\begin{theorem}
\label{Thm_gamma_convergence_sub}The functionals $E_{n}\left(  u\right)
:L^{1}\left(  \Omega\right)  \longmapsto\mathbb{R}$ given by%
\[
E_{n}\left(  u\right)  =\left\{
\begin{array}
[c]{ll}
{\displaystyle\int_{\Omega_{\infty}}}
\dfrac{\Phi_{n}\left(  x,\left\vert \nabla u\right\vert \right)  }{\varphi
_{n}\left(  x,1\right)  }dx+
{\displaystyle\int_{\Omega\backslash\Omega_{\infty}}}
\dfrac{\Phi\left(  x,\left\vert \nabla u\right\vert \right)  }{\varphi\left(
x,1\right)  }dx, & u\in W^{1,\tilde{\Phi}_{n}\left(  \cdot\right)  }\left(
\Omega\right)  ;\\
+\infty, & \text{otherwise}.
\end{array}
\right.
\]
and%
\[
E_{\infty}\left(  u\right)  =\left\{
\begin{array}
[c]{ll}
{\displaystyle\int_{\Omega\backslash\Omega_{\infty}}}
\dfrac{\Phi\left(  x,\left\vert \nabla u\right\vert \right)  }{\varphi\left(
x,1\right)  }dx, & u\in W^{1,\infty}\left(  \Omega_{\infty}\right)  \text{ and
}\left\Vert \nabla u\right\Vert _{L^{\infty}\left(  \Omega_{\infty}\right)
}\leqslant1;\\
+\infty, & \text{otherwise}.
\end{array}
\right.
\]
Then $\Gamma\left(  L^{1}\left(  \Omega\right)  \right)  $-$\lim
\limits_{n\rightarrow\infty}E_{n}=E_{\infty}.$
\end{theorem}

Consider the problem of minimizing%
\[
E_{n}\left(  u\right)  =\int_{\Omega}\frac{\tilde{\Phi}_{n}\left(
x,\left\vert \nabla u\right\vert \right)  }{\tilde{\varphi}_{n}\left(
x,1\right)  }dx=
{\displaystyle\int_{\Omega_{\infty}}}
\dfrac{\Phi_{n}\left(  x,\left\vert \nabla u\right\vert \right)  }{\varphi
_{n}\left(  x,1\right)  }dx+
{\displaystyle\int_{\Omega\backslash\Omega_{\infty}}}
\dfrac{\Phi\left(  x,\left\vert \nabla u\right\vert \right)  }{\varphi\left(
x,1\right)  }dx
\]
over $g+W_{0}^{1,\tilde{\Phi}_{n}\left(  \cdot\right)  }\left(  \Omega\right)
.$

\begin{theorem}
\label{Thm_existence_n}For each $p_{n}^{-}\geqslant d+1$. Equation $\left(
\ref{general_quasilinear_n_sub}\right)  $ has a unique weak solution $u_{n}$
which minimizes $E_{n}\left(  \cdot\right)  $ on $g+W_{0}^{1,\tilde{\Phi}%
_{n}\left(  \cdot\right)  }\left(  \Omega\right)  .$ In addition $u_{n}\in
C\left(  \bar{\Omega}\right)  .$
\end{theorem}

\begin{proof}
The proof is similar to Theorem \ref{Thm_existence}, the difference is that,
in the current setting, $\Phi_{n}$ is not continuous on $\Omega,$ hence the
local continuity condition fails. This means we can not invoke the usual
Poincare's inequality, i.e. Theorem \ref{thm_poincare_inequality}.
Instead, since $\Phi_{n}$ satisfies the jump condition (Definition
\ref{def_p_jump}), Theorem \ref{Thm_poincare_discontinuous} should be used to
get an equivalent norm on $W_{0}^{1,\tilde{\Phi}_{n}\left(  \cdot\right)
}\left(  \Omega\right)  $. After we recover the coerciveness of the
functional, the remaining proof follows as before.
\end{proof}

\begin{lemma}
\label{lm_weak_formulation}Equation $\left(  \ref{general_quasilinear_n_sub}%
\right)  $ has an equivalent formulation,%
\begin{equation}
\left\{
\begin{array}
[c]{ll}%
-\operatorname{div}\left(  \dfrac{a\left(  x,\left\vert \nabla u\right\vert
\right)  }{\left\vert \nabla u\left(  x\right)  \right\vert }\nabla u\right)
=0, & x\in\Omega\backslash\bar{\Omega}_{\infty};\\
-\operatorname{div}\left(  \dfrac{a_{n}\left(  x,\left\vert \nabla
u\right\vert \right)  }{\left\vert \nabla u\right\vert }\nabla u\right)  =0, &
x\in\Omega_{\infty},\\
\dfrac{a_{n}\left(  x,\left\vert \nabla u\right\vert \right)  }{\left\vert
\nabla u\right\vert }\dfrac{\partial u}{\partial\nu}=\dfrac{a\left(
x,\left\vert \nabla u\right\vert \right)  }{\left\vert \nabla u\right\vert
}\dfrac{\partial u}{\partial\nu}, & x\in\partial\Omega_{\infty}\cap\Omega,\\
u=g, & x\in\partial\Omega.
\end{array}
\right.  \label{equ_equivalent_form}%
\end{equation}
where $\nu$ is the exterior unit normal to the boundary of $\Omega_{\infty
}\cap\Omega$.
\end{lemma}

\begin{proof}
Write down the definition for $u_{n}$ being a weak solution, i.e., for any
$\phi\in C_{0}^{\infty}\left(  \Omega\right)  $,%
\begin{equation}
\int_{\Omega\backslash\bar{\Omega}_{\infty}}\frac{a\left(  x,\left\vert \nabla
u\right\vert \right)  }{\left\vert \nabla u\right\vert }\left\langle \nabla
u,\nabla\phi\right\rangle dx+\int_{\Omega_{\infty}}\frac{a_{n}\left(
x,\left\vert \nabla u\right\vert \right)  }{\left\vert \nabla u\right\vert
}\left\langle \nabla u,\nabla\phi\right\rangle dx=0.
\label{lm_weak_formulation_inq1}%
\end{equation}
Using Green's formula we get%
\begin{multline*}
\int_{\partial\Omega_{\infty}\cap\Omega}\left(  \frac{a_{n}\left(
x,\left\vert \nabla u\right\vert \right)  }{\left\vert \nabla u\right\vert
}-\frac{a\left(  x,\left\vert \nabla u\right\vert \right)  }{\left\vert \nabla
u\right\vert }\right)  \frac{\partial u}{\partial\nu}\phi dx\\
-\int_{\Omega\backslash\bar{\Omega}_{\infty}}\operatorname{div}\left(
\frac{a\left(  x,\left\vert \nabla u\right\vert \right)  }{\left\vert \nabla
u\right\vert }\nabla u\right)  \phi dx-\int_{\Omega_{\infty}}%
\operatorname{div}\left(  \frac{a_{n}\left(  x,\left\vert \nabla u\right\vert
\right)  }{\left\vert \nabla u\right\vert }\nabla u\right)  \phi dx=0.
\end{multline*}
The required formulation then follows by appropriately choosing test functions.
\end{proof}

As before the solution of the equation $\left(  \ref{equ_equivalent_form}%
\right)  $ is understood in the viscosity sense \cite{crandall1992user}%
\cite[Definition 2.3]{manfredi2009p}.

\begin{theorem}
\label{thm_weak_viscosity_sub}Let $p_{n}^{-}\geqslant d+1.$ If $u\in
W^{1,\tilde{\Phi}_{n}\left(  \cdot\right)  }\left(  \Omega\right)  $ is a
continuous weak solution for equation $\left(  \ref{equ_equivalent_form}%
\right)  $, then it is also a viscosity solution for equation $\left(
\ref{equ_equivalent_form}\right)  $.
\end{theorem}

\begin{proof}
Owing to Theorem \ref{Thm_existence_n}, $u$ is continuous up to the boundary
$\partial\Omega$. So we only need to check the viscosity inequalities for
$x_{0}\in\Omega\backslash\bar{\Omega}_{\infty}$, $x_{0}\in\Omega_{\infty}$ or
$x_{0}\in\partial\Omega_{\infty}\cap\Omega$. The first two scenarios are
similar to\ Theorem \ref{thm_weaksol_viscositysol}. Now fix $x_{0}\in
\partial\Omega_{\infty}\cap\Omega$. We prove the viscosity inequality for
subsolution, the proof for supersolution is similar. We need to show that the
required viscosity inequalities are satisfied. Suppose not, then there exists
$\phi\in C^{2}\left(  \bar{\Omega}\right)  $ such that $u-\phi$ has a strict
local maximum at $x_{0}$, $u\left(  x_{0}\right)  =\phi\left(  x_{0}\right)  $
and%
\[
\min\left\{
\begin{array}
[c]{l}%
-\operatorname{div}\left(  \dfrac{a\left(  x_{0},\left\vert \nabla\phi\left(
x_{0}\right)  \right\vert \right)  }{\left\vert \nabla\phi\left(
x_{0}\right)  \right\vert }\nabla\phi\left(  x_{0}\right)  \right)  ,\\
-\operatorname{div}\left(  \dfrac{a_{n}\left(  x_{0},\left\vert \nabla
\phi\left(  x_{0}\right)  \right\vert \right)  }{\left\vert \nabla\phi\left(
x_{0}\right)  \right\vert }\nabla\phi\left(  x_{0}\right)  \right)  ,\\
\dfrac{a_{n}\left(  x_{0},\left\vert \nabla\phi\left(  x_{0}\right)
\right\vert \right)  }{\left\vert \nabla\phi\left(  x_{0}\right)  \right\vert
}\dfrac{\partial\phi\left(  x_{0}\right)  }{\partial\nu}-\dfrac{a\left(
x_{0},\left\vert \nabla\phi\left(  x_{0}\right)  \right\vert \right)
}{\left\vert \nabla\phi\left(  x_{0}\right)  \right\vert }\dfrac{\partial
\phi\left(  x_{0}\right)  }{\partial\nu}%
\end{array}
\right\}  >0.
\]
In view of Remark \ref{rmk_operator_continuity} and the continuity of
$\nabla\phi$, the above inequality should also hold for all $x$ close to
$x_{0}$. Precisely, there is some open ball $B_{r}\left(  x_{0}\right)  $ with
radius $r>0$ around $x_{0}$ such that%
\begin{equation}
-\operatorname{div}\left(  \dfrac{a\left(  x,\left\vert \nabla\phi\right\vert
\right)  }{\left\vert \nabla\phi\right\vert }\nabla\phi\right)  >0\text{ for
}x\in B_{r}\left(  x_{0}\right)  \cap\left(  \Omega\backslash\bar{\Omega
}_{\infty}\right)  , \label{thm_weak_viscosity_sub_inq1}%
\end{equation}%
\begin{equation}
-\operatorname{div}\left(  \dfrac{a_{n}\left(  x,\left\vert \nabla
\phi\right\vert \right)  }{\left\vert \nabla\phi\right\vert }\nabla
\phi\right)  >0\text{ for }x\in B_{r}\left(  x_{0}\right)  \cap\Omega_{\infty
}, \label{thm_weak_viscosity_sub_inq2}%
\end{equation}
and%
\begin{equation}
\dfrac{a_{n}\left(  x,\left\vert \nabla\phi\right\vert \right)  }{\left\vert
\nabla\phi\right\vert }\dfrac{\partial\phi}{\partial\nu}-\dfrac{a\left(
x,\left\vert \nabla\phi\right\vert \right)  }{\left\vert \nabla\phi\right\vert
}\dfrac{\partial\phi}{\partial\nu}>0\text{ for }x\in B_{r}\left(
x_{0}\right)  \cap\partial\Omega_{\infty}. \label{thm_weak_viscosity_sub_inq3}%
\end{equation}
Let%
\[
\zeta\left(  x\right)  =\phi\left(  x\right)  -\frac{1}{2}\inf_{\left\vert
x-x_{0}\right\vert =r}\left(  \phi-u\right)  .
\]
Clearly $\zeta\left(  x_{0}\right)  <u\left(  x_{0}\right)  ,$ $\zeta-u>0$ on
$\partial B_{r}\left(  x_{0}\right)  $ and
\[
-\operatorname{div}\left(  \frac{a\left(  x,\left\vert \nabla\zeta\right\vert
\right)  }{\left\vert \nabla\zeta\right\vert }\nabla\zeta\right)  >0\text{ for
}x\in B_{r}\left(  x_{0}\right)  .
\]
Note that $\max\left\{  u-\zeta,0\right\}  $ is supported in $B_{r}\left(
x_{0}\right)  $.\ As in the proof of Theorem \ref{thm_weaksol_viscositysol},
we can legitimately use $\max\left\{  u-\zeta,0\right\}  $ as a test function.
Hence multiplying $\left(  \ref{thm_weak_viscosity_sub_inq1}\right)  ,$
$\left(  \ref{thm_weak_viscosity_sub_inq2}\right)  $\ with $\max\left\{
u-\zeta,0\right\}  $, and using Green's formula we obtain%
\begin{align*}
&  \int_{B_{r}\left(  x_{0}\right)  \cap\left(  \Omega\backslash\bar{\Omega
}_{\infty}\right)  }\frac{a\left(  x,\left\vert \nabla\phi\right\vert \right)
}{\left\vert \nabla\phi\right\vert }\left\langle \nabla\phi,\nabla\max\left\{
u-\zeta,0\right\}  \right\rangle dx+\\
&  \int_{B_{r}\left(  x_{0}\right)  \cap\Omega_{\infty}}\frac{a_{n}\left(
x,\left\vert \nabla\phi\right\vert \right)  }{\left\vert \nabla\phi\right\vert
}\left\langle \nabla\phi,\nabla\max\left\{  u-\zeta,0\right\}  \right\rangle
dx\\
&  >\int_{\partial\Omega_{\infty}\cap\Omega}\left(  \frac{a_{n}\left(
x,\left\vert \nabla u\right\vert \right)  }{\left\vert \nabla u\right\vert
}-\frac{a\left(  x,\left\vert \nabla u\right\vert \right)  }{\left\vert \nabla
u\right\vert }\right)  \nabla\max\left\{  u-\zeta,0\right\}  dx.
\end{align*}
Using $\left(  \ref{thm_weak_viscosity_sub_inq3}\right)  $, we get%
\begin{align*}
&  \int_{B_{r}\left(  x_{0}\right)  \cap\left(  \Omega\backslash\bar{\Omega
}_{\infty}\right)  \cap\left\{  \zeta<u\right\}  }\frac{a\left(  x,\left\vert
\nabla\phi\right\vert \right)  }{\left\vert \nabla\phi\right\vert
}\left\langle \nabla\phi,\nabla\left(  u-\zeta\right)  \right\rangle dx+\\
&  \int_{B_{r}\left(  x_{0}\right)  \cap\Omega_{\infty}\cap\left\{
\zeta<u\right\}  }\frac{a_{n}\left(  x,\left\vert \nabla\phi\right\vert
\right)  }{\left\vert \nabla\phi\right\vert }\left\langle \nabla\phi
,\nabla\left(  u-\zeta\right)  \right\rangle dx\\
&  >0.
\end{align*}
Now we subtract from the inequality the equation for weak solution $\left(
\ref{lm_weak_formulation_inq1}\right)  $ with the test function there replaced
with $\max\left\{  u-\zeta,0\right\}  $, to get an inequality similar to
$\left(  \ref{thm_weaksol_viscositysol_inq3}\right)  $. The remaining proof
then proceeds as in Theorem \ref{thm_weaksol_viscositysol}.
\end{proof}

\begin{theorem}
\label{thm_unqiue_sub}Assume $p_{n}^{-}\geqslant d+1$ and the set
\[
\mathcal{A}=\left\{  u\in W^{1,d+1}\left(  \Omega\right)  :u\in W^{1,\Phi
\left(  \cdot\right)  }\left(  \Omega\backslash\bar{\Omega}_{\infty}\right)
,\text{ }\left\Vert \nabla u\right\Vert _{L^{\infty}\left(  \Omega_{\infty
}\right)  }\leqslant1,\text{ }u|_{\partial\Omega}=g\right\}
\]
is non-empty. We write $u_{n}$ for the unique minimizer of $E_{n}\left(
\cdot\right)  $ which is also the unique continuous weak solution for the
equation $\left(  \ref{general_quasilinear_n_sub}\right)  .$The sequence of
weak solutions $u_{n}$ for $\left(  \ref{general_quasilinear_n_sub}\right)  $
is bounded in $W^{1,d+1}\left(  \Omega\right)  $.
\end{theorem}

\begin{proof}
Recalling the definition of $\tilde{\Phi}_{n},$ we note $\mathcal{A}\subset
g+W_{0}^{1,\tilde{\Phi}_{n}\left(  \cdot\right)  }\left(  \Omega\right)  $.
Since $u_{n}$ minimizes $E_{n}$ over $g+W_{0}^{1,\tilde{\Phi}_{n}\left(
\cdot\right)  }\left(  \Omega\right)  $ by Theorem \ref{Thm_existence_n}$,$ we
have
\[
E_{n}\left(  u_{n}\right)  \leqslant E_{n}\left(  v\right)  \text{ for all
}v\in\mathcal{A}\text{.}%
\]
We now fix an arbitrary $v\in\mathcal{A}$. Then%
\begin{align*}
E_{n}\left(  u_{n}\right)   &  \leqslant E_{n}\left(  v\right) \\
&  =\int_{\Omega\backslash\Omega_{\infty}}\frac{\Phi\left(  x,\nabla v\right)
}{\varphi\left(  x,1\right)  }dx+\int_{\Omega_{\infty}}\frac{\Phi_{n}\left(
x,\nabla v\right)  }{\varphi_{n}\left(  x,1\right)  }dx\\
&  \leqslant\int_{\Omega\backslash\Omega_{\infty}}\frac{\Phi\left(  x,\nabla
v\right)  }{\varphi\left(  x,1\right)  }dx+\int_{\Omega_{\infty}}\frac
{1}{p_{n}^{-}}dx\\
&  \leqslant\int_{\Omega\backslash\Omega_{\infty}}\frac{\Phi\left(  x,\nabla
v\right)  }{\varphi\left(  x,1\right)  }dx+\left\vert \Omega\right\vert .
\end{align*}
Note that the proof for Lemma \ref{lm_bounded_W1d1} does not require any
continuity of $\Phi_{n}$ in $x$, so it equally applies to the current setting
where $\Phi_{n}$ might present jumps$.$ So at this point, we can proceed as in
Theorem \ref{thm_bounded_W1d} to conclude the boundedness of $u_{n}$ in
$W^{1,d+1}\left(  \Omega\right)  $ by invoking Lemma \ref{lm_bounded_W1d1}.
For general $g\in Lip\left(  \partial\Omega\right)  $ for which the assumption
$\left(  \ref{assumption_g_lip}\right)  $ might not be satisfied, we can also
still employ Theorem \ref{thm_bounded_inW1d_all_g}, which does not require
local continuity of $\Phi_{n}$ in $x$, and then follow Theorem
\ref{thm_bounded_W1d} to complete the proof.
\end{proof}

\begin{theorem}
\label{thm_convergence_sub}Suppose that the set $\mathcal{A}$ defined in
Theorem \ref{thm_unqiue_sub} is non-empty. Let $u_{n}$ be the unique
continuous weak solution for the equation $\left(
\ref{general_quasilinear_n_sub}\right)  $. Then there is $u\in C\left(
\bar{\Omega}\right)  $ such that the whole sequence $u_{n}$ converges
uniformly to $u\in C\left(  \bar{\Omega}\right)  $. Moreover $u$ is a
viscosity solution of the equation,%
\begin{equation}
\left\{
\begin{array}
[c]{ll}%
-\operatorname{div}\left(  \dfrac{a\left(  x,\left\vert \nabla u\right\vert
\right)  }{\left\vert \nabla u\right\vert }\nabla u\right)  =0, & x\in
\Omega\backslash\Omega_{\infty};\\
-\Delta_{\infty}u-\left\vert \nabla u\right\vert \left\langle \Lambda\left(
x,\left\vert \nabla u\right\vert \right)  ,\nabla u\right\rangle =0, &
x\in\Omega_{\infty};\\
sgn\left(  \left\vert \nabla u\right\vert -1\right)  sgn\left(  \dfrac
{\partial u}{\partial\nu}\right)  =0, & x\in\partial\Omega_{\infty}\cap
\Omega;\\
u=g, & x\in\partial\Omega.
\end{array}
\right.  \label{thm_convergence_sub_limit}%
\end{equation}

\end{theorem}

\begin{proof}
We prove that $u$ is a viscosity supersolution, the proof for subsolution is
analagous. Suppose that $u-\phi$ has a strict local minimum at $x_{0}\in
\Omega$ for some $\phi\in C^{2}\left(  \bar{\Omega}\right)  $, $u\left(
x_{0}\right)  =\phi\left(  x_{0}\right)  $. We distinguish three scenarios
according to where $x_{0}$ locates: $x_{0}\in\Omega\backslash\bar{\Omega
}_{\infty}$, $x_{0}\in\Omega_{\infty}$, $x_{0}\in\partial\Omega$ or $x_{0}%
\in\partial\Omega_{\infty}\cap\Omega$. The first two scenarios are similar
to\ Theorem \ref{thm_convergence_main} and actually it is much simpler when
$x_{0}\in\Omega_{\infty}$. If $x_{0}\in\partial\Omega$, there is nothing to
check, since $u$ is continuous up to the boundary. It remains to show that the
viscosity inequalities are verifiable when $x_{0}\in\partial\Omega_{\infty
}\cap\Omega$, i.e.%
\begin{equation}
\max\left\{
\begin{array}
[c]{l}%
-\operatorname{div}\left(  \dfrac{a\left(  x_{0},\left\vert \nabla\phi\left(
x_{0}\right)  \right\vert \right)  }{\left\vert \nabla\phi\left(
x_{0}\right)  \right\vert }\nabla\phi\left(  x_{0}\right)  \right)  ,\\
-\Delta_{\infty}\phi\left(  x_{0}\right)  -\left\vert \nabla\phi\left(
x_{0}\right)  \right\vert \left\langle \Lambda\left(  x_{0},\left\vert
\nabla\phi\left(  x_{0}\right)  \right\vert \right)  ,\nabla\phi\left(
x_{0}\right)  \right\rangle ,\\
sgn\left(  \left\vert \nabla\phi\left(  x_{0}\right)  \right\vert -1\right)
sgn\left(  \dfrac{\partial\phi\left(  x_{0}\right)  }{\partial\nu}\right)
\end{array}
\right\}  \geqslant0. \label{thm_convergence_sub_Inq1}%
\end{equation}
Since $u_{n}$ converges to $u$ uniformly, there exists $x_{n}\rightarrow
x_{0}$ such that $u_{n}-\phi$ attains a local minimum at $x_{n}\in\Omega$. Now
again three cases might occur. \textbf{(i)} Up to a subsequence, all $x_{n}$
belongs to $\Omega_{\infty}.$ Using Lemma \ref{equ_equivalent_form},%
\[
-\operatorname{div}\left(  \dfrac{a_{n}\left(  x_{n},\left\vert \nabla
\phi\left(  x_{n}\right)  \right\vert \right)  }{\left\vert \nabla\phi\left(
x_{n}\right)  \right\vert }\nabla\phi\left(  x_{n}\right)  \right)
\geqslant0.
\]
Similar to Theoerm \ref{thm_convergence_main},\ upon sending $n\rightarrow
\infty$, we are led to%
\[
-\Delta_{\infty}\phi\left(  x_{0}\right)  -\left\vert \nabla\phi\left(
x_{0}\right)  \right\vert \left\langle \Lambda\left(  x_{0},\left\vert
\nabla\phi\left(  x_{0}\right)  \right\vert \right)  ,\nabla\phi\left(
x_{0}\right)  \right\rangle \geqslant0.
\]
\textbf{(ii)} Up to a subsequence, all $x_{n}$ belongs to $\Omega
\backslash\Omega_{\infty}.$
\[
-\operatorname{div}\left(  \dfrac{a\left(  x_{n},\left\vert \nabla\phi\left(
x_{n}\right)  \right\vert \right)  }{\left\vert \nabla\phi\left(
x_{n}\right)  \right\vert }\nabla\phi\left(  x_{n}\right)  \right)
\geqslant0,
\]
hence%
\[
-\operatorname{div}\left(  \dfrac{a\left(  x_{0},\left\vert \nabla\phi\left(
x_{0}\right)  \right\vert \right)  }{\left\vert \nabla\phi\left(
x_{0}\right)  \right\vert }\nabla\phi\left(  x_{0}\right)  \right)
\geqslant0.
\]
\textbf{(iii)} Up to a subsequence, all $x_{n}$ belongs to $\partial
\Omega_{\infty}\cap\Omega.$
\[
\dfrac{a_{n}\left(  x_{n},\left\vert \nabla\phi\left(  x_{n}\right)
\right\vert \right)  }{\left\vert \nabla\phi\left(  x_{n}\right)  \right\vert
}\dfrac{\partial\phi}{\partial\nu}\left(  x_{n}\right)  -\dfrac{a\left(
x_{n},\left\vert \nabla\phi\left(  x_{n}\right)  \right\vert \right)
}{\left\vert \nabla\phi\left(  x_{n}\right)  \right\vert }\dfrac{\partial\phi
}{\partial\nu}\left(  x_{n}\right)  \geqslant0.
\]
Note that $a_{n}\left(  x,s\right)  =0$ if and only if $s=0$. We clearly have
$\left(  \ref{thm_convergence_sub_Inq1}\right)  $ if $\left\vert \nabla
\phi\left(  x_{0}\right)  \right\vert \in\left\{  0,1\right\}  $. Assume
$\left\vert \nabla\phi\left(  x_{0}\right)  \right\vert \notin\left\{
0,1\right\}  $, then, for $n$ sufficiently large, there is $c_{1}>0$ for which
$\left\vert \nabla\phi\left(  x_{n}\right)  \right\vert \geqslant c_{1}$. It
follows that
\begin{equation}
\dfrac{a_{n}\left(  x_{n},\left\vert \nabla\phi\left(  x_{n}\right)
\right\vert \right)  }{a\left(  x_{n},\left\vert \nabla\phi\left(
x_{n}\right)  \right\vert \right)  }\dfrac{\partial\phi}{\partial\nu}\left(
x_{n}\right)  \geqslant\dfrac{\partial\phi}{\partial\nu}\left(  x_{n}\right)
. \label{thm_convergence_sub_Inq2}%
\end{equation}
Since $a\left(  x_{n},\left\vert \nabla\phi\left(  x_{n}\right)  \right\vert
\right)  $ has strictly positive limit, the limit of $\left(
\ref{thm_convergence_sub_Inq2}\right)  $ depends on whether $\left\vert
\nabla\phi\left(  x_{0}\right)  \right\vert >1$ or $\left\vert \nabla
\phi\left(  x_{0}\right)  \right\vert <1$. Assume the former is true, i.e.
$\left\vert \nabla\phi\left(  x_{0}\right)  \right\vert >1$, then taking limit
in $\left(  \ref{thm_convergence_sub_Inq2}\right)  $ yields%
\[
\dfrac{\partial\phi}{\partial\nu}\left(  x_{0}\right)  \geqslant0.
\]
Then $\left(  \ref{thm_convergence_sub_Inq1}\right)  $ follows. If the latter
is true, i.e. $\left\vert \nabla\phi\left(  x_{0}\right)  \right\vert <1$,
then taking limit in $\left(  \ref{thm_convergence_sub_Inq2}\right)  $ yields%
\[
\dfrac{\partial\phi}{\partial\nu}\left(  x_{0}\right)  \leqslant0.
\]
Still $\left(  \ref{thm_convergence_sub_Inq1}\right)  $ is true. Therefore $u$
is a viscosity supersolution. Finally, suppose $u_{n}$ has two limits in the
sense of uniform convergence, say $u^{1},$ $u^{2}\in C\left(  \bar{\Omega
}\right)  \cap W^{1,\infty}\left(  \Omega\right)  $. Since $\mathcal{A}$ is
non-empty, we see, by following part of the proof of Theorem
\ref{thm_unqiue_sub}, that both $u^{1}$ and $u^{2}$ minimize%
\[
u\longmapsto
{\displaystyle\int_{\Omega\backslash\Omega_{\infty}}}
\dfrac{\Phi\left(  x,\left\vert \nabla u\right\vert \right)  }{\varphi\left(
x,1\right)  }dx
\]
over $\mathcal{A}$. By the convexity of the functional and the continuity of
$u^{1}$ and $u^{2}$, it follows that $u^{1}=u^{2}$ in $\bar{\Omega}%
\backslash\Omega_{\infty}$. By virtue of Theorem
\ref{thm_limit_equ_unique_main}, this implies $u^{1}=u^{2}$ in $\bar{\Omega
}_{\infty}.$ Therefore, the whole sequence $\left\{  u_{n}\right\}  $ is
convergent. The proof is completed.
\end{proof}

\section{Appendix: Some technical inequalities}

We prove a few technical inequalities in this appendix. Lemma
\ref{inequality_main}, which has been used in Theorem
\ref{thm_weaksol_viscositysol} and Theorem \ref{thm_uniqueness_limit_equ},
extends the classical inequality \cite[Lemma 4.4, p13]%
{dibenedetto1993degenerate}. When $\Phi\left(  s\right)  =s^{p}$ for
$p\geqslant2$, the inequality $\left(  \ref{inequality_main_inq2}\right)  $ is
well-known\ and plays essential roles in a number of problems related to
degenerate quasilinear problems, e.g. \cite{juutinen2001equivalence}. As a
corollary, we obtain an inequality (i.e. $\left(
\ref{inequality_main_var_inq1}\right)  $) which is announced in \cite[Lemma
3.2]{fukagai2007existence}, but it seems a proof is not given there and the
references therein. The inequalities we provide here are explicit in the sense
that all constants are identified. Knowing explicit dependency of the
constants on the parameters $p^{-}$ and $p^{+}$ are quite essential, since we
are taking limits with respect to $p^{-}$and $p^{+}$. We believe these
inequalities will also be of interest beyond the current context.

In the following, $\Phi$ is a $\Phi$-function (see Definition $\left(
\ref{equ_PHI_phi}\right)  $) satisfying $\left(  \ref{inequality_PHI}\right)
$. Write
\[
\psi\left(  s\right)  =\dfrac{\varphi\left(  s\right)  }{s}.
\]

\begin{lemma}
\label{inequality_main}Let $\gamma$ be given as in Lemma \ref{lemma_convexity}%
.\ Then, for $\xi,\eta\in\mathbb{R}^{d},$
\begin{equation}
\left\langle \psi\left(  \left\vert \xi\right\vert \right)  \xi-\psi\left(
\left\vert \eta\right\vert \right)  \eta,\xi-\eta\right\rangle \geqslant
c^{-}p^{-}\min\left\{  \frac{\Phi^{1+\gamma}\left(  3^{-\frac{1}{2}}\left\vert
\xi-\eta\right\vert \right)  }{\Phi^{\gamma}\left(  \left\vert \xi\right\vert
+\left\vert \eta\right\vert \right)  },\frac{\Phi^{1+\gamma}\left(
\kappa^{\frac{1}{2}}\left\vert \xi-\eta\right\vert \right)  }{4\Phi^{\gamma
}\left(  \left\vert \xi\right\vert +\left\vert \eta\right\vert \right)
}\right\}  , \label{inequality_main_inq1}%
\end{equation}
where $c^{-}=\min\left\{  p^{-}-1,1\right\}  $ and $0<\kappa<1$ is a constant
independent of $p^{-}$ and $p^{+}$.

If $p^{-}\geqslant2$, we have an improvement,
\begin{equation}
\left\langle \psi\left(  \left\vert \xi\right\vert \right)  \xi-\psi\left(
\left\vert \eta\right\vert \right)  \eta,\xi-\eta\right\rangle \geqslant
\min\left\{  \Phi\left(  \left\vert \xi-\eta\right\vert \right)  ,\frac{p^{-}%
}{4}\Phi\left(  \kappa^{\frac{1}{2}}\left\vert \xi-\eta\right\vert \right)
\right\}  , \label{inequality_main_inq2}%
\end{equation}
where $0<\kappa<1$ is a constant independent of $p^{-}$ and $p^{+}$.
\end{lemma}

We start with the proof for $\left(  \ref{inequality_main_inq2}\right)  .$

\begin{proof}
We have $p^{-}\geqslant2$ in this case.%
\begin{align*}
&  \left\langle \psi\left(  \left\vert \xi\right\vert \right)  \xi-\psi\left(
\left\vert \eta\right\vert \right)  \eta,\xi-\eta\right\rangle \\
&  =\left\langle \int_{0}^{1}\frac{d}{ds}\left(  \psi\left(  \left\vert
s\xi+\left(  1-s\right)  \eta\right\vert \right)  \left(  s\xi+\left(
1-s\right)  \eta\right)  \right)  ds,\xi-\eta\right\rangle \\
&  =\int_{0}^{1}\left[  \psi\left(  \left\vert s\xi+\left(  1-s\right)
\eta\right\vert \right)  \left\vert \xi-\eta\right\vert ^{2}+\frac
{\psi^{\prime}\left(  \left\vert s\xi+\left(  1-s\right)  \eta\right\vert
\right)  }{\left\vert s\xi+\left(  1-s\right)  \eta\right\vert }\left\vert
\left\langle s\xi+\left(  1-s\right)  \eta,\xi-\eta\right\rangle \right\vert
^{2}\right]  ds.
\end{align*}
We denote the integrand by $H\left(  s,\xi,\eta\right)  $. In view of $\left(
\ref{assumption_phi}\right)  $,
\[
\frac{s\varphi^{\prime}\left(  s\right)  }{\varphi\left(  s\right)  }\geqslant
p^{-}-1\geqslant1,
\]
It follows that%
\[
\psi^{\prime}\left(  s\right)  =\frac{s\varphi^{\prime}\left(  s\right)
-\varphi\left(  s\right)  }{s^{2}}\geqslant0.
\]
Hence we have for $s\in\left[  0,1\right]  ,$%
\begin{equation}
H\left(  s,\xi,\eta\right)  \geqslant\psi\left(  \left\vert s\xi+\left(
1-s\right)  \eta\right\vert \right)  \left\vert \xi-\eta\right\vert ^{2}.
\label{main_iproof_inq1}%
\end{equation}

We now distinguish two cases.

Case I: $\left\vert \xi\right\vert \geqslant\left\vert \xi-\eta\right\vert .$
Then
\begin{align*}
\left\vert s\xi+\left(  1-s\right)  \eta\right\vert  &  =\left\vert
\xi+\left(  1-s\right)  \left(  \xi-\eta\right)  \right\vert \geqslant
\left\vert \xi\right\vert -\left(  1-s\right)  \left\vert \xi-\eta\right\vert
\\
&  \geqslant s\left\vert \xi-\eta\right\vert .
\end{align*}
In view of $\left(  \ref{assumption_phi}\right)  $,
\[
\frac{s\varphi^{\prime}}{\varphi}\geqslant p^{-}-1\geqslant1,
\]
which implies that $\psi^{\prime}\geqslant0$. It follows that%
\begin{align*}
\left\langle \psi\left(  \left\vert \xi\right\vert \right)  \xi-\psi\left(
\left\vert \eta\right\vert \right)  \eta,\xi-\eta\right\rangle  &
\geqslant\int_{0}^{1}\psi\left(  \left\vert s\xi+\left(  1-s\right)
\eta\right\vert \right)  \left\vert \xi-\eta\right\vert ^{2}ds\\
&  \geqslant\int_{0}^{1}\psi\left(  s\left\vert \xi-\eta\right\vert \right)
\left\vert \xi-\eta\right\vert ^{2}ds\\
&  \geqslant\int_{0}^{1}s\psi\left(  s\left\vert \xi-\eta\right\vert \right)
\left\vert \xi-\eta\right\vert ^{2}ds\\
&  =\int_{0}^{1}\frac{d}{ds}\Phi\left(  s\left\vert \xi-\eta\right\vert
\right)  ds=\Phi\left(  \left\vert \xi-\eta\right\vert \right)
\end{align*}
Case II: $\left\vert \xi\right\vert <\left\vert \xi-\eta\right\vert .$ Then%
\begin{align*}
\left\vert s\xi+\left(  1-s\right)  \eta\right\vert  &  =\left\vert
\xi+\left(  1-s\right)  \left(  \xi-\eta\right)  \right\vert \leqslant
\left\vert \xi\right\vert +\left(  1-s\right)  \left\vert \xi-\eta\right\vert
\\
&  \leqslant\left(  2-s\right)  \left\vert \xi-\eta\right\vert .
\end{align*}
Using this inequality and Lemma \ref{lemma_lieberman} we have%
\begin{align*}
\left\langle \psi\left(  \left\vert \xi\right\vert \right)  \xi-\psi\left(
\left\vert \eta\right\vert \right)  \eta,\xi-\eta\right\rangle  &
\geqslant\left\vert \xi-\eta\right\vert ^{2}\int_{0}^{1}\frac{\varphi\left(
\left\vert s\xi+\left(  1-s\right)  \eta\right\vert \right)  }{\left\vert
s\xi+\left(  1-s\right)  \eta\right\vert }ds\\
&  \geqslant p^{-}\left\vert \xi-\eta\right\vert ^{2}\int_{0}^{1}\frac
{\Phi\left(  \left\vert s\xi+\left(  1-s\right)  \eta\right\vert \right)
}{\left\vert s\xi+\left(  1-s\right)  \eta\right\vert ^{2}}ds\\
&  \geqslant p^{-}\left\vert \xi-\eta\right\vert ^{2}\int_{0}^{1}\frac
{\Phi\left(  \left\vert s\xi+\left(  1-s\right)  \eta\right\vert \right)
}{\left(  2-s\right)  ^{2}\left\vert \xi-\eta\right\vert ^{2}}ds\\
&  \geqslant\frac{p^{-}}{4}\int_{0}^{1}\Phi\left(  \left\vert s\xi+\left(
1-s\right)  \eta\right\vert \right)  ds
\end{align*}
Note that for some $0<\kappa<1$ small we have%
\begin{equation}
\int_{0}^{1}\left\vert s\xi+\left(  1-s\right)  \eta\right\vert ^{2}%
ds=\frac{1}{3}\left(  \left\vert \xi\right\vert ^{2}+\left\vert \eta
\right\vert ^{2}+\left\langle \xi,\eta\right\rangle \right)  \geqslant
\kappa\left\vert \xi-\eta\right\vert ^{2}. \label{main_iproof_inq2}%
\end{equation}
By Lemma \ref{lemma_convexity}, $s\longmapsto G\left(  s\right)  =\Phi\left(
s^{1/2}\right)  $ is an increasing convex function. Therefore%
\begin{align*}
\frac{p^{-}}{4}\int_{0}^{1}\Phi\left(  \left\vert s\xi+\left(  1-s\right)
\eta\right\vert \right)  ds  &  =\frac{p^{-}}{4}\int_{0}^{1}G\left(
\left\vert s\xi+\left(  1-s\right)  \eta\right\vert ^{2}\right)  ds\\
&  \geqslant\frac{p^{-}}{4}G\left(  \int_{0}^{1}\left\vert s\xi+\left(
1-s\right)  \eta\right\vert ^{2}ds\right) \\
&  \geqslant\frac{p^{-}}{4}G\left(  \kappa\left\vert \xi-\eta\right\vert
^{2}\right) \\
&  =\frac{p^{-}}{4}\Phi\left(  \kappa^{\frac{1}{2}}\left\vert \xi
-\eta\right\vert \right)  .
\end{align*}

\end{proof}

We now turn to the proof for $\left(  \ref{inequality_main_inq1}\right)  .$

\begin{proof}
Note that this time we do not have $p^{-}\geqslant2$, hence $\psi^{\prime
}\left(  s\right)  $ may not be nonnegative, however estimate $\left(
\ref{main_iproof_inq1}\right)  $ is still valid upon multiplication with an
additional coefficient $c^{-}=\min\left\{  p^{-}-1,1\right\}  ,$ i.e.
\begin{align*}
&  \left\langle \psi\left(  \left\vert \xi\right\vert \right)  \xi-\psi\left(
\left\vert \eta\right\vert \right)  \eta,\xi-\eta\right\rangle \geqslant
c^{-}\int_{0}^{1}\frac{\varphi\left(  \left\vert s\xi+\left(  1-s\right)
\eta\right\vert \right)  }{\left\vert s\xi+\left(  1-s\right)  \eta\right\vert
}\left\vert \xi-\eta\right\vert ^{2}ds\\
&  \geqslant c^{-}p^{-}\int_{0}^{1}\frac{\Phi\left(  \left\vert s\xi+\left(
1-s\right)  \eta\right\vert \right)  }{\left\vert s\xi+\left(  1-s\right)
\eta\right\vert ^{2}}\left\vert \xi-\eta\right\vert ^{2}ds
\end{align*}
Two cases are considered as before.

Case I: $\left\vert \xi\right\vert \geqslant\left\vert \xi-\eta\right\vert $,
then $\left\vert s\xi+\left(  1-s\right)  \eta\right\vert \geqslant
s\left\vert \xi-\eta\right\vert .$ Take $\gamma$ such that
\[
\gamma\geqslant\left(  \inf_{t>0}\frac{t\varphi\left(  t\right)  }{\Phi\left(
t\right)  }\right)  ^{-1}.
\]
Since%
\[
\inf_{t>0}\frac{t\varphi\left(  t\right)  }{\Phi\left(  t\right)  }%
\geqslant1,
\]
we have
\[
\frac{2}{1+\gamma}\leqslant\inf_{t>0}\frac{t\varphi\left(  t\right)  }%
{\Phi\left(  t\right)  }.
\]
By virtue of Lemma \ref{lemma_delta2_condition}, the mapping $s\longmapsto
s^{-\alpha}\Phi\left(  s\right)  $ with $\alpha=\frac{2}{1+\gamma}$ is
increasing. It follows that
\[
\frac{\Phi\left(  s\left\vert \xi-\eta\right\vert \right)  }{\left\vert
\xi-\eta\right\vert ^{\alpha}}\leqslant\frac{\Phi\left(  s\left\vert \xi
-\eta\right\vert \right)  }{\left(  s\left\vert \xi-\eta\right\vert \right)
^{\alpha}}\leqslant\frac{\Phi\left(  \left\vert s\xi+\left(  1-s\right)
\eta\right\vert \right)  }{\left\vert s\xi+\left(  1-s\right)  \eta\right\vert
^{\alpha}},\text{ for }s\in\left(  0,1\right)  ,
\]
and then%
\[
\frac{\Phi^{2/\alpha}\left(  s\left\vert \xi-\eta\right\vert \right)  }%
{\Phi^{2/\alpha}\left(  \left\vert s\xi+\left(  1-s\right)  \eta\right\vert
\right)  }\leqslant\frac{\left\vert \xi-\eta\right\vert ^{2}}{\left\vert
s\xi+\left(  1-s\right)  \eta\right\vert ^{2}}.
\]
Note that $2/\alpha-1=\gamma\geqslant\left(  \inf_{t>0}\frac{t\varphi\left(
t\right)  }{\Phi\left(  t\right)  }\right)  ^{-1}$, hence the mapping
$s\longmapsto\Phi^{1+\gamma}\left(  s^{1/2}\right)  $ is convex by Lemma
\ref{lemma_convexity}.\ Therefore%
\begin{align*}
\left\langle \psi\left(  \left\vert \xi\right\vert \right)  \xi-\psi\left(
\left\vert \eta\right\vert \right)  \eta,\xi-\eta\right\rangle  &  \geqslant
c^{-}p^{-}\int_{0}^{1}\frac{\Phi\left(  \left\vert s\xi+\left(  1-s\right)
\eta\right\vert \right)  }{\left\vert s\xi+\left(  1-s\right)  \eta\right\vert
^{2}}\left\vert \xi-\eta\right\vert ^{2}ds\\
&  \geqslant c^{-}p^{-}\int_{0}^{1}\frac{\Phi^{2/\alpha}\left(  s\left\vert
\xi-\eta\right\vert \right)  }{\Phi^{2/\alpha-1}\left(  \left\vert
s\xi+\left(  1-s\right)  \eta\right\vert \right)  }ds\\
&  =c^{-}p^{-}\int_{0}^{1}\frac{\Phi^{1+\gamma}\left(  s\left\vert \xi
-\eta\right\vert \right)  }{\Phi^{\gamma}\left(  \left\vert s\xi+\left(
1-s\right)  \eta\right\vert \right)  }ds\\
&  \geqslant c^{-}p^{-}\frac{1}{\Phi^{\gamma}\left(  \left\vert \xi\right\vert
+\left\vert \eta\right\vert \right)  }\int_{0}^{1}\Phi^{1+\gamma}\left(
s\left\vert \xi-\eta\right\vert \right)  ds\\
&  \geqslant c^{-}p^{-}\frac{1}{\Phi^{\gamma}\left(  \left\vert \xi\right\vert
+\left\vert \eta\right\vert \right)  }\Phi^{1+\gamma}\left(  3^{-\frac{1}{2}%
}\left\vert \xi-\eta\right\vert \right)
\end{align*}

In the last inequality we use the convexity of the mapping $s\longmapsto
\Phi^{1+\gamma}\left(  s^{1/2}\right)  $ and the same reasoning as in the
proof for $\left(  \ref{inequality_main_inq2}\right)  .$

In the other case where $\left\vert \xi\right\vert <\left\vert \xi
-\eta\right\vert ,$ we have as before that%
\[
\left\vert s\xi+\left(  1-s\right)  \eta\right\vert \leqslant\left(
2-s\right)  \left\vert \xi-\eta\right\vert \leqslant2\left\vert \xi
-\eta\right\vert .
\]
Hence continuing previous inequality we get that%
\begin{align*}
\left\langle \psi\left(  \left\vert \xi\right\vert \right)  \xi-\psi\left(
\left\vert \eta\right\vert \right)  \eta,\xi-\eta\right\rangle  &  \geqslant
c^{-}p^{-}\int_{0}^{1}\frac{\Phi\left(  \left\vert s\xi+\left(  1-s\right)
\eta\right\vert \right)  }{\left\vert s\xi+\left(  1-s\right)  \eta\right\vert
^{2}}\left\vert \xi-\eta\right\vert ^{2}ds\\
&  \geqslant\frac{c^{-}p^{-}}{4}\int_{0}^{1}\Phi\left(  \left\vert
s\xi+\left(  1-s\right)  \eta\right\vert \right)  ds\\
&  \geqslant\frac{c^{-}p^{-}}{4}\int_{0}^{1}\frac{\Phi^{1+\gamma}\left(
\left\vert s\xi+\left(  1-s\right)  \eta\right\vert \right)  }{\Phi^{\gamma
}\left(  \left\vert s\xi+\left(  1-s\right)  \eta\right\vert \right)  }ds\\
&  \geqslant\frac{c^{-}p^{-}}{4}\frac{1}{\Phi^{\gamma}\left(  \left\vert
\xi\right\vert +\left\vert \eta\right\vert \right)  }\int_{0}^{1}%
\Phi^{1+\gamma}\left(  \left\vert s\xi+\left(  1-s\right)  \eta\right\vert
\right)  ds\\
&  \geqslant\frac{c^{-}p^{-}}{4}\frac{1}{\Phi^{\gamma}\left(  \left\vert
\xi\right\vert +\left\vert \eta\right\vert \right)  }\Phi^{1+\gamma}\left(
\kappa^{1/2}\left\vert \xi-\eta\right\vert \right)
\end{align*}
where $\kappa$ is the same as in $\left(  \ref{main_iproof_inq2}\right)  $.
\end{proof}

Using Lemma \ref{lemma_delta2_condition}, we have a variant of Lemma
\ref{inequality_main}, which is useful when the explicit form of constants are
not relevant.

\begin{lemma}
\label{inequality_main_var}Let $\gamma$ be given as in Lemma
\ref{lemma_convexity}.\ There is $C_{1}\left(  p^{-},p^{+},\gamma\right)  >0$
so that for $\xi,\eta\in\mathbb{R}^{d},$
\begin{equation}
\left\langle \psi\left(  \left\vert \xi\right\vert \right)  \xi-\psi\left(
\left\vert \eta\right\vert \right)  \eta,\xi-\eta\right\rangle \geqslant
C_{1}\left(  p^{-},p^{+},\gamma\right)  \frac{\Phi^{1+\gamma}\left(
\left\vert \xi-\eta\right\vert \right)  }{\Phi^{\gamma}\left(  \left\vert
\xi\right\vert +\left\vert \eta\right\vert \right)  }.
\label{inequality_main_var_inq1}%
\end{equation}
If $p^{-}\geqslant2$, there is $C_{2}\left(  p^{-},p^{+}\right)  >0$ so that
for $\xi,\eta\in\mathbb{R}^{n},$
\begin{equation}
\left\langle \psi\left(  \left\vert \xi\right\vert \right)  \xi-\psi\left(
\left\vert \eta\right\vert \right)  \eta,\xi-\eta\right\rangle \geqslant
C_{2}\left(  p^{-},p^{+}\right)  \Phi\left(  \left\vert \xi-\eta\right\vert
\right)  . \label{inequality_main_var_inq2}%
\end{equation}

\end{lemma}

For any generalized $\Phi$-function $\Phi\left(  x,s\right)  $ satisfying
$\left(  \ref{inequality_PHI_x}\right)  $, a counterpart of Lemma
\ref{inequality_main} exists. Write
\[
\psi\left(  x,s\right)  =\frac{\varphi\left(  x,s\right)  }{s}.
\]

\begin{lemma}
\label{inequality_main_x}Let $\gamma\geqslant1$.\ There is $C_{1}\left(
p^{-},p^{+},\gamma\right)  >0$ so that for $x\in\Omega,$ $\xi,\eta
\in\mathbb{R}^{d},$
\[
\left\langle \psi\left(  x,\left\vert \xi\right\vert \right)  \xi-\psi\left(
x,\left\vert \eta\right\vert \right)  \eta,\xi-\eta\right\rangle \geqslant
C_{1}\left(  p^{-},p^{+},\gamma\right)  \frac{\Phi^{1+\gamma}\left(
x,\left\vert \xi-\eta\right\vert \right)  }{\Phi^{\gamma}\left(  x,\left\vert
\xi\right\vert +\left\vert \eta\right\vert \right)  }.
\]
If $p^{-}\geqslant2$, we have an improvement,
\[
\left\langle \psi\left(  x,\left\vert \xi\right\vert \right)  \xi-\psi\left(
x,\left\vert \eta\right\vert \right)  \eta,\xi-\eta\right\rangle \geqslant
C_{2}\left(  p^{-},p^{+}\right)  \Phi\left(  x,\left\vert \xi-\eta\right\vert
\right)  ,
\]
where $C_{2}\left(  p^{-},p^{+}\right)  >0.$
\end{lemma}

\bibliographystyle{plain}

\bibliography{phi-limit}

\begin{bibdiv}
\begin{biblist}

\bib{aronsson1996fast}{article}{
      author={Aronsson, Gunnar},
      author={Evans, Lawrence~C},
      author={Wu, Y},
       title={Fast/slow diffusion and growing sandpiles},
        date={1996},
     journal={journal of differential equations},
      volume={131},
      number={2},
       pages={304\ndash 335},
}

\bib{adams2003sobolev}{book}{
      author={Adams, Robert~A},
      author={Fournier, John~JF},
       title={Sobolev spaces},
   publisher={Elsevier},
        date={2003},
      volume={140},
}

\bib{barles2001existence}{article}{
      author={Barles, Guy},
      author={Busca, J{\'e}r{\^o}me},
       title={Existence and comparison results for fully nonlinear degenerate
  elliptic equations without zeroth-order term},
        date={2001},
     journal={Communications in Partial Differential Equations},
      volume={26},
      number={11-12},
       pages={2323\ndash 2337},
}

\bib{bhattacharya1989limits}{article}{
      author={Bhattacharya, Tilak},
      author={DiBenedetto, Emmanuele},
      author={Manfredi, Juan},
       title={Limits as $p \to \infty$ of {$\Delta _p u_p= f$} and related
  extremal problems},
        date={1989},
     journal={Rend. Sem. Mat. Univ. Politec. Torino},
      volume={47},
       pages={15\ndash 68},
}

\bib{bocea2015gamma}{article}{
      author={Bocea, Marian},
      author={Mih{\u{a}}ilescu, Mihai},
       title={{$\Gamma$}-convergence of inhomogeneous functionals in
  orlicz-sobolev spaces},
        date={2015},
     journal={Proceedings of the Edinburgh Mathematical Society},
      volume={58},
      number={2},
       pages={287\ndash 303},
}

\bib{charro2009mixed}{article}{
      author={Charro, Fernando},
      author={Azorero, Jesus~Garc{\'\i}a},
      author={Rossi, Julio~D},
       title={A mixed problem for the infinity laplacian via tug-of-war games},
        date={2009},
     journal={Calculus of Variations and Partial Differential Equations},
      volume={34},
      number={3},
       pages={307\ndash 320},
}

\bib{crandall1992user}{article}{
      author={Crandall, Michael~G},
      author={Ishii, Hitoshi},
      author={Lions, Pierre-Louis},
       title={User’s guide to viscosity solutions of second order partial
  differential equations},
        date={1992},
     journal={Bulletin of the American mathematical society},
      volume={27},
      number={1},
       pages={1\ndash 67},
}

\bib{dacorogna2007direct}{book}{
      author={Dacorogna, Bernard},
       title={Direct methods in the calculus of variations second edition},
   publisher={Springer Science \& Business Media},
        date={2007},
      volume={78},
}

\bib{diening2011lebesgue}{article}{
      author={Diening, Lars},
      author={Harjulehto, Petteri},
      author={H{\"a}st{\"o}, Peter},
      author={Ruzicka, Michael},
       title={Lebesgue and sobolev spaces with variable exponents},
        date={2011},
     journal={Lecture Notes in Mathematics},
      volume={2017},
}

\bib{dibenedetto1993degenerate}{book}{
      author={DiBenedetto, Emmanuele},
       title={Degenerate parabolic equations},
   publisher={Springer Science \& Business Media},
        date={1993},
}

\bib{evans1997fast}{article}{
      author={Evans, Lawrence~C},
      author={Feldman, Mikhail},
      author={Gariepy, RF},
       title={Fast/slow diffusion and collapsing sandpiles},
        date={1997},
     journal={Journal of differential equations},
      volume={137},
      number={1},
       pages={166},
}

\bib{fukagai2006positive}{article}{
      author={Fukagai, Nobuyoshi},
      author={Ito, Masayuki},
      author={Narukawa, Kimiaki},
       title={Positive solutions of quasilinear elliptic equations with
  critical orlicz-sobolev nonlinearity on rn},
        date={2006},
     journal={Funkcialaj Ekvacioj},
      volume={49},
      number={2},
       pages={235\ndash 267},
}

\bib{fukagai2007existence}{article}{
      author={Fukagai, Nobuyoshi},
      author={Narukawa, Kimiaki},
       title={On the existence of multiple positive solutions of quasilinear
  elliptic eigenvalue problems},
        date={2007},
     journal={Annali di Matematica Pura ed Applicata},
      volume={186},
      number={3},
       pages={539\ndash 564},
}

\bib{gilbarg2001elliptic}{book}{
      author={Gilbarg, David},
      author={Trudinger, Neil~S},
       title={Elliptic partial differential equations of second order},
   publisher={Springer Science \& Business Media},
        date={2001},
      volume={224},
}

\bib{harjulehto2019orlicz}{book}{
      author={Harjulehto, Petteri},
      author={H{\"a}st{\"o}, Peter},
       title={Orlicz spaces and generalized orlicz spaces},
   publisher={Springer},
        date={2019},
      volume={2236},
}

\bib{harjulehto2016generalized}{article}{
      author={Harjulehto, Petteri},
      author={H{\"a}st{\"o}, Peter},
      author={Kl{\'e}n, Riku},
       title={Generalized orlicz spaces and related {PDE}},
        date={2016},
     journal={Nonlinear Analysis: Theory, Methods \& Applications},
      volume={143},
       pages={155\ndash 173},
}

\bib{harjulehto2006dirichlet}{article}{
      author={Harjulehto, Petteri},
      author={H{\"a}st{\"o}, Peter},
      author={Koskenoja, Mika},
      author={Varonen, Susanna},
       title={The dirichlet energy integral and variable exponent sobolev
  spaces with zero boundary values},
        date={2006},
     journal={Potential Analysis},
      volume={25},
      number={3},
       pages={205},
}

\bib{jensen1993uniqueness}{article}{
      author={Jensen, Robert},
       title={Uniqueness of lipschitz extensions: minimizing the sup norm of
  the gradient},
        date={1993},
     journal={Archive for Rational Mechanics and Analysis},
      volume={123},
      number={1},
       pages={51\ndash 74},
}

\bib{juutinen2001equivalence}{article}{
      author={Juutinen, Petri},
      author={Lindqvist, Peter},
      author={Manfredi, Juan~J},
       title={On the equivalence of viscosity solutions and weak solutions for
  a quasi-linear equation},
        date={2001},
     journal={SIAM journal on mathematical analysis},
      volume={33},
      number={3},
       pages={699\ndash 717},
}

\bib{juutinen1999eigenvalue}{article}{
      author={Juutinen, Petri},
      author={Lindqvist, Peter},
      author={Manfredi, Juan~J},
       title={The $\infty$-eigenvalue problem},
        date={1999},
     journal={Archive for rational mechanics and analysis},
      volume={148},
      number={2},
       pages={89\ndash 105},
}

\bib{lieberman1991natural}{article}{
      author={Lieberman, Gary~M},
       title={The natural generalizationj of the natural conditions of
  ladyzhenskaya and urall'tseva for elliptic equations},
        date={1991},
     journal={Communications in Partial Differential Equations},
      volume={16},
      number={2-3},
       pages={311\ndash 361},
}

\bib{lindqvist2016notes}{book}{
      author={Lindqvist, Peter},
       title={Notes on the infinity laplace equation},
   publisher={Springer},
        date={2016},
}

\bib{lindqvist2010curious}{article}{
      author={Lindqvist, Peter},
      author={Lukkari, Teemu},
       title={A curious equation involving the $\infty$-laplacian},
        date={2010},
     journal={Advances in Calculus of Variations},
      volume={3},
      number={4},
       pages={409\ndash 421},
}

\bib{manfredi2009p}{inproceedings}{
      author={Manfredi, Juan~J},
      author={Rossi, Julio~D},
      author={Urbano, Jos{\'e}~Miguel},
       title={$p(x)$-harmonic functions with unbounded exponent in a
  subdomain},
organization={Elsevier},
        date={2009},
   booktitle={Annales de l'institut henri poincare (c) non linear analysis},
      volume={26},
       pages={2581\ndash 2595},
}

\bib{manfredi2010limits}{article}{
      author={Manfredi, J.~J.},
      author={Rossi, J.~D.},
      author={Urbano, J.~M.},
       title={Limits as $p(x) \rightarrow \infty $ of $p(x)$-harmonic
  functions},
        date={2010},
     journal={Nonlinear analysis},
      volume={72},
      number={1},
       pages={309\ndash 315},
}

\bib{perez2011anisotropic}{article}{
      author={Perez-Llanos, Mayte},
      author={ROSS, JULIO~D},
       title={An anisotropic infinity laplacian obtained as the limit of the
  anisotropic $(p,q)$-{L}aplacian},
        date={2011},
     journal={Communications in Contemporary Mathematics},
      volume={13},
      number={06},
       pages={1057\ndash 1076},
}

\bib{peres2009tug}{article}{
      author={Peres, Yuval},
      author={Schramm, Oded},
      author={Sheffield, Scott},
      author={Wilson, David},
       title={Tug-of-war and the infinity laplacian},
        date={2009},
     journal={Journal of the American Mathematical Society},
      volume={22},
      number={1},
       pages={167\ndash 210},
}

\bib{stancu2018asymptotic}{article}{
      author={Stancu-Dumitru, Denisa},
       title={The asymptotic behavior of a class of $\varphi$-harmonic
  functions in orlicz--sobolev spaces},
        date={2018},
     journal={Journal of Mathematical Analysis and Applications},
      volume={463},
      number={1},
       pages={365\ndash 376},
}

\end{biblist}
\end{bibdiv}
\end{document}